\documentclass[11pt]{article}
\usepackage{amssymb}
\usepackage{latexsym}
\usepackage{amsmath}
\usepackage{amscd}
\usepackage{amsbsy}
\usepackage[pdftex]{graphics}
\usepackage{graphicx}
 \usepackage{dsfont}

\newtheorem{theorem}{Theorem}

\newtheorem{lemma}{Lemma}

\newtheorem{proposition}{Proposition}

\newenvironment{proof}[1][Proof]{\textbf{#1.} }{\ \rule{0.5em}{0.5em}}

\newcommand{\Var}{\text{Var}}
\newcommand{\Cov}{\text{Cov}}

\def\ds{\displaystyle}

\begin{document}

\title{Sharp minimax tests for large Toeplitz covariance matrices with repeated observations}

\author{Cristina Butucea$^{1,2}$, Rania Zgheib$^{1,2}$ \\
$^1$ Universit\'e Paris-Est Marne-la-Vall\'ee, \\
LAMA(UMR 8050), UPEMLV 
F-77454, Marne-la-Vall\'ee, France\\
$^2$ ENSAE-CREST-GENES \\
3, ave. P. Larousse
92245 MALAKOFF Cedex, FRANCE
}

\maketitle

\begin{abstract}
We observe a sample of $n$ independent $p$-dimensional  Gaussian vectors with Toeplitz covariance matrix $ \Sigma = [\sigma_{|i-j|}]_{1 \leq i,j \leq p}$ and $\sigma_0=1$.
We consider the problem of testing the hypothesis that $\Sigma$  is the identity matrix asymptotically when $n \to \infty$ and $p \to \infty$. We suppose that the covariances $\sigma_k$ decrease  either polynomially ($\sum_{k \geq 1} k^{2\alpha} \sigma^2_{k} \leq L$ for $ \alpha >1/4$ and $L>0$) or exponentially ($\sum_{k \geq 1} e^{2Ak} \sigma^2_{k} \leq L$ for $ A,L>0$).

We consider a test procedure based on a weighted U-statistic of order 2, with optimal weights chosen as solution  of an   extremal problem. We give the asymptotic normality of the test statistic under the null hypothesis for fixed $n$ and $p \to + \infty$ and the asymptotic behavior of the type I error probability of our test procedure. We also show that the maximal type II error probability, either tend to $0$, or is bounded from above. In the latter case, the upper  bound is given  using the asymptotic normality of our test statistic under alternatives close to the separation boundary. Our assumptions imply mild conditions:  $n=o(p^{2\alpha - 1/2})$ (in the polynomial case), $n=o(e^p)$ (in the exponential case).

We prove both rate optimality and sharp optimality of our results,  for $\alpha >1$ in the polynomial case and for any $A>0$ in the exponential case. 

A simulation study illustrates the good behavior of our procedure, in particular for small $n$, large $p$.
\end{abstract}

\noindent {\bf Key Words:} Toeplitz matrix, covariance matrix, high-dimensional data, U-statistic, minimax hypothesis testing, optimal separation rates, sharp asymptotic rates.

\noindent {\bf MSC 2000:}  62G10,  62H15,  62G20, 62H10



\section{Introduction}\label{introduction}

In the last decade, both functional data analysis (FDA) and high-dimensional (HD) problems have known an unprecedented expansion both from a theoretical point of view (as they offer many mathematical challenges) and for the applications (where data have complex structure and grow larger every day). Therefore, both areas share a large number of trends, see \cite{iwfos14} and the review by \cite{Cuevas14}, like regression models with functional or large-dimensional covariates, supervised or unsupervised classification, testing procedures, covariance operators.

Functional data analysis proceeds very often by discretizing curve datasets in time domain or by projecting on suitable orthonormal systems and produces large dimensional vectors with size possibly larger than the sample size. Hence methods and techniques  from HD problems can be successfully implemented (see e.g.~\cite{AneirosVieu14}).However,  in some cases, HD vectors can be transformed into stochastic processes, see \cite{ChenChen11}, and then techniques from FDA bring new insights into HD problems.
Our work is of the former type.


We observe  independent, identically distributed Gaussian vectors $X_1,...,X_n$, $n \geq 2$, which are $p$-dimensional, centered and with a positive definite Toeplitz covariance matrix $\Sigma$. We denote by $X_k = (X_{k,1},...,X_{k,p}) ^\top$ the coordinates of the vector $X_k$ in $\mathbb{R}^p$  for all $k$.

Our model is that of a stationary Gaussian time series, repeatedly and independently observed  $n$ times, for $n\geq 2$. We assume that $n$ and $p$ are large. In functional data analysis, it is quite often that curves are observed in an independent way: electrocardiograms of different patients, power supply for different households and so on, see other data sets in \cite{iwfos14}. After modelisation of the discretized curves, the statistician will study the normality and the whiteness of the residuals in order to validate the model. Our problem is to test from independent samples of high-dimensional residual vectors that the standardized Gaussian coordinates are uncorrelated.

Let us denote by $\sigma_{|j|} = Cov(X_{k,h}, X_{k,h+j})$, for all integer numbers $h$ and $j$, for all $k\in \mathbb{N}^*$, where $\mathbb{N}^*$ is the set of positive integers. We assume that $\sigma_0 = 1$, therefore $\sigma_j$ are correlation coefficients.
We recall that $\{\sigma_j\}_{j \in \mathbb{N}}$ is a sequence of non-negative type, or, equivalently, the associated Toeplitz matrix $\Sigma$ is non-negative definite.
We assume that the sequence $\{\sigma_j\}_{j \in \mathbb{N}}$ belongs to to $\ell_1 (\mathbb{N}) \cap \ell_2 (\mathbb{N}) $, where  $\ell_1 (\mathbb{N})$ (resp. $\ell_2 (\mathbb{N}) $) is the set all absolutely (resp. square) summable sequences. It is therefore possible to construct a positive, periodic function
$$
f(x) = \frac 1{2 \pi} \left(1 + 2\sum_{j = 1}^{\infty} \sigma_j \cos ( jx) \right), \quad
\mbox{ for } x \in (-\pi,\pi),
$$
belonging to $\mathbb{L}_2(-\pi, \pi) $ the set of all square-integrable functions $f$  over $ (-\pi,\pi)$ . This function is known as the spectral density of the stationary series $\{X_{k,i}, i\in \mathbb{Z}\}$.

We solve the following test problem,
\begin{equation}
\label{H_0}
H_0 : \Sigma = I
 \end{equation}
versus the alternative
\begin{equation}\label{H1}
H_1 : \Sigma \in \mathcal{T}(\alpha, L) \text{ such that } \ds \sum_{j \geq 1} \sigma_j^2  \geq \psi^2,
\end{equation}
for $\psi =(\psi_{n,p})_{n,p}$ a positive sequence converging to 0. From now on, $C_{>0}$ denotes the set of squared symmetric and positive definite matrices.
The set $\mathcal{T}(\alpha, L)$ is an ellipsoid of Sobolev type
$$
  \mathcal{T}(\alpha, L) = \{ \Sigma \in C_{>0} , \Sigma \text{ is Toeplitz } ; \ds\sum_{j \geq 1} \sigma_j^2 j^{2\alpha} \leq L  \text{ and } \sigma_0 = 1 \}, \,\alpha >1/4, \, L>0.
$$
We shall also test (\ref{H_0}) against
\begin{equation}\label{H1E}
H_1 : \Sigma \in \mathcal{E}(A, L) \text{ such that } \ds \sum_{j \geq 1} \sigma_j^2  \geq \psi^2, \mbox{ for }\psi >0,
\end{equation}
where the ellipsoid of covariance matrices is given by
$$
  \mathcal{E}(A, L) = \{ \Sigma \in C_{>0} , \Sigma \text{ is Toeplitz } ; \ds\sum_{j \geq 1} \sigma_j^2 e^{2Aj} \leq L  \text{ and } \sigma_0 = 1 \}, A,L>0.
$$
This class contains the covariance matrices whose elements decrease exponentially, when moving away from the diagonal.
We denote by $G(\psi )$ either $ G ( \mathcal{T}(\alpha, L) , \psi) $ the set of matrices under the alternative (\ref{H1}) or $ G ( \mathcal{E}(A, L) , \psi) $ under the alternative (\ref{H1E}).

We stress the fact that a matrix $\Sigma$ in $G(\psi)$ is such that $1/(2p) \|\Sigma - I\|_F^2 \geq \sum_{j \geq 1} \sigma_j^2 \geq \psi^2$, i.e. $\Sigma$ is outside a neighborhood of $I$ with radius  $\psi$ in Frobenius norm. 

\smallskip

Our test can be applied in the context of model fitting for testing the whiteness of the standard Gaussian residuals. In this context, it is natural to assume that the covariance matrix under the alternative hypothesis has small entries like in our classes of covariance matrices. Such tests have been proposed by \cite{FisherGallagher12}, where it is noted that weighted test statistics can be more powerful.

Note that, most of the literature on testing the null hypothesis \eqref{H_0}, either focus on finding the asymptotic behavior of the test statistic under the null hypothesis, or control in addition  the type II error probability  for one fixed unknown matrix under the alternative, whereas our main interest is to quantify the worst type II error probabilities, i.e. uniformly over a large set of possible covariance matrices.

Various test statistics in high dimensional settings have been considered for testing \eqref{H_0}, as it was known for some time that likelihood ratio tests do not converge when dimension grows. Therefore,  a corrected   Likelihood Ratio Test is proposed in \cite{Bai2009} when  $p/n \to c \in (0,1)$, and its asymptotic behavior  is given under the null hypothesis, based on the random matrix theory. In \cite{JiangJiangYang12}  the result is extended to $c=1$. An exact test based on one column of the covariance matrix is constructed by \cite{GuptaBodnar14}.
A series of papers propose  test statistics based on the Frobenius norm of $ \Sigma -I$,   see  \cite{LedoitWolf02},  \cite{Srivastava2005},  \cite{Srivastava2014} and \cite{ChenZZ10}. Different test statistics are introduced and their asymptotic distribution is studied. In particular in \cite{ChenZZ10} the test statistic is a U-statistic with constant weights.   An unbiased estimator of $tr(\Sigma - B_k(\Sigma))^2$ is constructed in  \cite{QiuChen12}, where $B_k(\Sigma) =(\sigma_{ij} \cdot I\{|i-j| \leq k\})$, in order to develop a test statistic for the problem of testing the bandedness of a given  matrix.
Another extension of our test problem is to test  the sphericity hypothesis $\Sigma= \sigma^2 I$, where $\sigma^2 >0$ is unknown.  \cite{FisherSunGallagher2010} introduced a test statistic based on functionals of order 4 of the covariance matrix.
Motivated by these results, the test $H_0: \Sigma = I$ is revisited by \cite{Fisher12}.
The maximum value of non-diagonal elements of the empirical covariance matrix was also investigated as a test statistic. Its asymptotic extreme-value distribution was given under the identity covariance matrix by \cite{CaiJiang11} and for other covariance matrices by \cite{XiaoWu13}.
We propose here a new test statistic to test \eqref{H_0} which is a weighted U-statistic of order 2 and study its probability errors uniformly over the set of matrices given by the alternative hypothesis.


The test problem with alternative (\ref{H1}) and with one sample ($n=1$) was solved in the sharp asymptotic framework, as $p \to \infty$, by \cite{Ermakov94}. Indeed, \cite{Ermakov94} studies sharp minimax testing of the spectral density $f$ of the Gaussian process.
Note that under the null hypothesis we have a constant spectral density $f_0(x) = 1/(2\pi)$ for all $x$ and the alternative can be described in $\mathbb{L}_2$ norm as
we have the following isometry $\|f - f_0\|_2^2 = (2\pi)^{-1} \|\Sigma - I\|_F^2$. Moreover, the ellipsoid of covariance matrices $\mathcal{T} (\alpha, L)$ are in bijection with Sobolev ellipsoids of spectral densities $f$. Let us also recall that the adaptive rates for minimax testing are obtained for the spectral density problem by \cite{GolubevNussbaumZhou10} by a non constructive method using the asymptotic equivalence with a Gaussian white noise model. Finding explicit test procedures which adapt automatically to parameters $\alpha$ and/or $L$ of our class of matrices will be the object of future work. Our efforts go here into finding sharp minimax rates for testing.

Our results generalize the results in \cite{Ermakov94} to the case of repeatedly observed stationary Gaussian process.
We stress the fact that repeated sampling of the stationary process $(X_{1,1},...,X_{1,p})$ to $(X_{n,1},...,X_{n,p})$ can be viewed as one sample of size $n \times p$ under the null hypothesis. However, this sample will not fit the assumptions of our alternative. Indeed, under the alternative, its covariance matrix is not Toeplitz, but block diagonal. 
Moreover, we can summarize the $n$ independent vectors into one $p$-dimensional vector $X = n^{-1/2} \sum_{k=1}^n X_k$ having Gaussian distribution $\mathcal{N}_p(0, \Sigma)$. The results by  \cite{Ermakov94} will produce a test procedure with rate that we expect optimal as a function of $p$, but more biased and suboptimal as a function of $n$. The test statistic that we suggest removes cross-terms and has smaller bias.
Therefore, results in \cite{Ermakov94} do not apply in a straightforward way to our setup.

A conjecture in the sense of asymptotic equivalence of the model of repeatedly observed Gaussian vectors and a Gaussian white noise model was given by \cite{CaiRenZhou13}. Our rates go in the sense of the conjecture. 

\smallskip

The test of $H_0:\Sigma=I$ against (\ref{H1}), with $\Sigma$  not  necessary Toeplitz, is given in \cite{ButuceaZgheib2014A}. Their rates
show a loss of a factor $p$ when compared to the rates for Toeplitz matrices obtained here.
 This can be interpreted  heuristically by the size of the set of unknown parameters which is $p(p-1)/2$ for \cite{ButuceaZgheib2014A} whereas here it is $p$.
%
We can see that the family of Toeplitz matrices is a subfamily of general covariance matrices in \cite{ButuceaZgheib2014A}. Therefore, the lower bounds are different, they are attained through a particular family of Toeplitz large covariance matrices. The upper bounds take into account as well the fact that we have repeated information on the same diagonal elements. The test statistic is different from the one used in \cite{ButuceaZgheib2014A}.

\smallskip

The test problem with alternative hypothesis (\ref{H1E}) has not been studied in this model.
The class $\mathcal{E}(A,L)$ contains matrices with exponentially decaying elements when further from the main diagonal.
The spectral density function associated to this process belongs to the class of functions which are in $\mathbb{L}_2$ and admit an analytic continuation on the strip of complex numbers $z$ with $ |Im(z)| \leq A$. Such classes of analytic functions are very popular in the literature of minimax estimation, see \cite{GolubevLevitTsybakov96} .


In times series analysis such covariance matrices describe among others the linear ARMA processes.
%
The problem of adaptive estimation of the spectral density of an ARMA process has been studied by \cite{Golubev1993} (for known $\alpha$) and adaptively to $\alpha$ via wavelet based methods by \cite{Neumann1996} and by model selection by \cite{Comte2001}.
In the case of an ARFIMA process, obtained by fractional differentiation of order $d \in (-1/2,1/2)$ of a casual invertible ARMA process,
\cite{Soulier2000} gave adaptive estimators of the spectral density based on the log-periodogram regression model when the covariance matrix belongs to $\mathcal{E}(A,L)$.

\smallskip

Before describing our results let us define more precisely the quantities we are interested in evaluating.

\smallskip

\subsection{Formalism of the minimax theory of testing}
Let $ \chi $ be a test, that is a measurable function of the observations $X_1, \dots, X_n$ taking values in $\{0,1 \}$ and recall that $G(\psi)$ corresponds to the set of covariance matrice under the alternative hypothesis. Let
\begin{eqnarray*}
\eta (\chi) &=& \mathbb{E}_I (\chi) \quad \text{ be its  type I error probability, and} \\
\beta (\chi \, ,  G(\psi))  &=& \sup\limits_{\Sigma \in G(\psi)} \mathbb{E}_{\Sigma}( 1 - \chi) \quad \text{be its maximal type II error probability. }
\end{eqnarray*}
We consider two criteria to measure the performance of the test procedure. The first one corresponds to the classical Neyman-Pearson criterion. For $w \in(0,1)$, we define,
\[
\beta_w ( G(\psi)) = \inf\limits_{\chi \, ; \, \eta(\chi) \leq w } \beta (\chi,  G(\psi)).
\]
The test $\chi_{w}$ is asymptotically minimax according to the Neyman-Pearson criterion if
\[
\eta (\chi_{w}) \leq w + o(1) \quad \text{ and} \quad \beta (\chi_{w} \, ,  G(\psi))= \beta_w( G(\psi) ) + o(1).
\]

The second criterion is the total error probability, which is defined as follows:
\[
\gamma ( \, \chi \, ,  G(\psi)) =   \eta (\chi) +  \beta ( \, \chi \, ,  G(\psi)).
\]
Define also the minimax total error probability $\gamma$ as $
\gamma(G(\psi)) = \inf\limits_{\chi} \gamma (\, \chi \, ,  G(\psi))$, 
where the infimum is taken over all possible tests.

Note that the two criteria are related since $ \gamma(G(\psi)) = \inf_{w \in (0,1) } (w + \beta_w ( G(\psi)) )$ (see Ingster and Suslina~\cite{IngsterSuslina03}).
\begin{center}
\begin{table}
$$
\begin{array}{c|ccc}
\Sigma  & \mathcal{T}(\alpha, L)  & \mathcal{E}(A, L) & \mbox{not Toeplitz and } \mathcal{T}(\alpha, L) \text{ \cite{ButuceaZgheib2014A} }\\
 \hline
 \widetilde \psi &
\left( C(\alpha, L)  \cdot  n^2  p^2\, \right)^{-\frac{\alpha }{4\alpha+1}}
&
\left( \ds\frac{2 \ln(n^2 p^2)}{A n^2 p^2} \right)^{1/4}
&
 \left(
 C(\alpha, L) \cdot n^2p \right)^{ - \frac{ \alpha }{4 \alpha +1}}
\\
 b(\psi)^2
&
C(\alpha, L)  \cdot \psi^{\frac{4\alpha + 1}{ \alpha}}
&
\ds\frac{A \psi^4}{2 \ln \Big(\ds\frac{1}{ \psi} \Big)}
&C(\alpha, L)  \cdot \psi^{\frac{4\alpha + 1}{ \alpha}}\\
\end{array}
$$
\caption{Separation rates $\widetilde \psi$ and $b(\psi)$ in the sharp asymptotic bounds}
where  $C(\alpha, L) = (2 \alpha +1)(4\alpha +1)^{-(1 +\frac{1}{2 \alpha }) } L^{-\frac{1}{2\alpha}}$. \label{table}
\end{table}
\end{center}
\vspace{-1 cm}
A test $\chi$ is asymptotically minimax if: $
\gamma (G(\psi))=  \gamma(\, \chi \, ,  G(\psi)) + o(1).$
We say that $ \widetilde{\psi}$ is a (asymptotic) separation rate, if the following lower bounds hold
\[
\gamma (G(\psi)) \longrightarrow 1 \quad \text{ as } \ds\frac{\psi}{\widetilde{\psi}} \longrightarrow  0
\]
together with the following upper bounds: there exists a test $ \chi$ such that,
\[
\gamma ( \, \chi \, ,  G(\psi))  \longrightarrow 0  \quad \text{ as } \ds\frac{\psi}{\widetilde{\psi}} \longrightarrow  + \infty.
\]

The sharp optimality corresponds to the study of the asymptotic behavior of the maximal type II error probability $\beta_w(G(\psi))$ and the total error probability $\gamma(G(\psi))$. In our study we obtain  asymptotic behavior  of Gaussian type, i.e. we   show that, under some assumptions,
\begin{equation}\label{sharp}
\beta_w( G(\psi)) = \Phi ( z_{1-w}  - np b(\psi)) + o(1) \quad \text{and} \quad
\gamma( G(\psi)) = 2 \Phi( -npb(\psi)) + o(1),
\end{equation}
where $\Phi$ is the cumulative distribution function of a standard Gaussian random variable, $z_{1-w} $ is the $1-w$ quantile of the standard Gaussian distribution for any $w \in (0,1)$,  and
 $b(\psi)$ has an explicit form for each ellipsoid of Toeplitz covariance matrices.

Separation rates and sharp asymptotic results for different testing problem were studied under this formalism by \cite{IngsterSpatinas09}. We refer for precise definitions of sharp asymptotic and non asymptotic rates to \cite{MarteauSapatinas14}.
Note that throughout this paper,  asymptotics and  symbols  $o$, $O$, $\sim$ and $\asymp$ are considered as $p$ tends to infinity, unless we specify that $n$ tends to infinity.
Recall that, given sequences of real numbers $u$ and real positive numbers $v$, we say that they are
asymptotically equivalent,  $u \sim v$, if $\lim u/v= 1$. Moreover, we say that the sequences are asymptotically of the same order, $u \asymp v$, if there exist two constants $0<c \leq C < \infty$ such that $c \leq \lim\inf u/v$ and $\limsup u/v \leq C$.
\subsection{Overview of the results}
\label{overviewresults}

In this paper, we describe the separation rates $\widetilde \psi$ and sharp asymptotics for the error probabilities for testing the identity matrix against $G ( \mathcal{T}(\alpha, L) , \psi)$ and $G ( \mathcal{E}(A, L) , \psi)$ respectively.

We propose here a test procedure whose type II error probability tends to 0 uniformly over the set of $G(\psi)$, that is even for a covariance matrix that gets closer to the identity matrix at distance $\widetilde \psi\to 0$ as $n$ and $p$ increase. The radius $\widetilde \psi $ in  Table~\ref{table} is the smallest vicinity around the identity matrix which still allows testing error probabilities to tend to 0. Our test statistic is a weighted quadratic form and we show how to choose these weights in an optimal way over each class of alternative hypotheses.

Under mild assumptions we obtain  the   sharp optimality in (\ref{sharp}),
where $b(\psi)$  is described in Table~\ref{table} and compared to the case of non Toeplitz matrices in \cite{ButuceaZgheib2014A}.

This paper is structured   as follows.
In Section \ref{sec:toeplitz}, we study the test problem with alternative hypothesis defined by the class $G(\mathcal{T}(\alpha, L), \psi)$, $\alpha >1/4$, $L, \, \psi>0$. We define explicitly the test statistic and give its first and second moments under the null and the alternative hypotheses. We derive its Gaussian asymptotic behavior under the null hypothesis and under the alternative submitted to the constraints that $\psi$ is close to the separation rate $\widetilde \psi$ and  that $\Sigma$ is closed to the solution of an extremal problem $\Sigma^*$. We deduce the asymptotic separation rates. Their optimality is shown only for $\alpha >1$. Our lower bounds are original in the literature of minimax lower bounds, as in this case we cannot reduce the proof to the vector case, or diagonal matrices.
We give  the sharp rates for $\psi \asymp \widetilde \psi$. Our assumptions imply that necessarily $n = o(p^{2\alpha - 1/2})$ as $p \to \infty$. That does not prevent $n$ to be larger than $p$ for sufficiently large $\alpha$.

In Section \ref{sec:testprob2}, we derive analogous results over the class $G(\mathcal{E}(A, L), \psi)$, with $A,\, L, \psi >0$. We show how to choose the parameters in this case and study the test procedure similarly. We give asymptotic separation rates. The sharp bounds are attained as $\psi \asymp \widetilde \psi$. Our assumptions involve that $n = o(\exp(p))$ which allows $n$ to grow exponentially fast with $p$. That can be explained by the fact that the elements of $\Sigma$ decay much faster over exponential ellipsoids than over the polynomial ones.
In Section~\ref{sec:simu} we implement our procedure and show the power of testing over two families of covariance matrices. 

The proofs of our results are postponed to the Section \ref{sec:proofs} and to the Supplementary material. 


\section{Testing procedure and results for polynomially decreasing covariances}
\label{sec:toeplitz}

  We introduce a weighted U-statistic of order 2, which is an estimator of the functional $\sum_{j \geq 1} \sigma_j^2$ that defines the separation between a Toeplitz covariance matrix under the alternative hypothesis from  the identity matrix under the null. Indeed, in nonparametric estimation of quadratic functionals such as $\sum_{j \geq 1} \sigma_j^2$ weighted estimators are often considered (see e.g.  \cite{ButuceaMeziani11}). These weights have finite support of length $T$, where $T$ is optimal in some sense. Intuitively, as the coefficients $\{\sigma_j\}_j$ belong to an ellipsoid, they become smaller when $j$ increases and thus the bias due to the truncation and the weights becomes as small as the variance for estimating the weighted finite sum.

 \subsection{Test Statistic }
Let us denote by $T_p( \{ \sigma_j \}_{j \geq 1}) $ the symmetric $p \times p$ Toeplitz matrix $\Sigma = [\sigma_{lk}]_{1 \leq l, k \leq p}$ such that the diagonal elements of $ \Sigma$ are equal to 1, and $   \sigma_{lk} = \sigma_{kl} = \sigma_{|l-k|}$, for all $l \ne k$.
 Now we define the weighted test statistic in this setup
\begin{equation}
\label{esttetplitz}
\widehat{\mathcal{A}}_n: = \widehat{\mathcal{A}}_n^{\mathcal{T}} = \ds\frac{1}{n(n-1)(p-T)^2} \underset{1 \leq k \neq l \leq n}{\ds\sum }\sum_{j=1}^T w_j^* \underset{T+1 \leq i_1, i_2 \leq p}{\ds\sum } X_{k, i_1}X_{k, i_1-j}X_{l, i_2}X_{l, i_2-j}
\end{equation}
where the weights $\{ w_j^* \}_j $ and the parameters $T, \lambda, b^2(\psi)$ are obtained by solving the following extremal problem:
\begin{equation}
\label{probopt}
b(\psi) :=\sum_{j \geq 1} w_j^* \sigma_j^{*2} =  \sup\limits_{ \left\{\substack{ (w_{j})_{j} ~~:~~ w_{j} \geq 0 ;\\ \\
  \sum_{j \geq 1}  w_{j}^2 =\frac{1}{2}} \right\} }
  \inf\limits_{ \left\{ \substack{ \Sigma ~:~ \Sigma=T_p (\{\sigma_{j} \}_{j \geq 1} ) ; \\ \\
 \Sigma \in \mathcal{T}(\alpha,L) ,~~ \sum_{j \geq 1} \sigma_j^2 \geq \psi^2 }\right\} }
  \sum_{j\geq 1} w_{j} \sigma_{j}^2.
\end{equation}
This extremal problem appears heuristically as we want that the expected value of our test statistic for the worst parameter $\Sigma$ under the alternative hypothesis (closest to the null) to be as large as possible for the weights we use. This problem will provide the optimal weights $\{w_j^*\}_{j \geq 1}$ in order to control the worst type II error probability, but also the critical matrix $\Sigma^* = T_p(\{\sigma^*_j\})$ that will be used in the lower bounds. Indeed, $\Sigma^*$ is positive definite for small enough $\psi$ (see \cite{ButuceaZgheib2014A}).

The solution of the extremal problem (\ref{probopt}) can be found in  \cite{IngsterSuslina03}:
\begin{equation}\label{parameters1}
\begin{array}{lcl}
w_j^* &=& \ds\frac{\lambda}{2b (\psi)} \Big( 1 - (\frac{j}{T})^{2 \alpha} \Big) , \quad \sigma_j^{*2}= \lambda \Big( 1 - (\frac{j}{T})^{2 \alpha} \Big), \quad  T = \lfloor (L(4 \alpha +1))^{\frac{1}{2 \alpha}} \cdot \psi^{- \frac{1}{\alpha}} \rfloor  \\  \\
 \lambda &=&  \ds\frac{2 \alpha + 1}{2 \alpha  (L(4\alpha +1))^{\frac{1}{2 \alpha}}}\cdot \psi^{\frac{2\alpha + 1}{ \alpha}} , \quad
b^{2}(\psi)  =\frac{1}{2}\sum_{j} \sigma_j^{*4}= \ds\frac{2 \alpha +1}{L^{\frac{1}{2\alpha}}(4\alpha +1)^{1 +\frac{1}{2 \alpha } }} \cdot \psi^{\frac{4\alpha + 1}{ \alpha}}
\end{array}
\end{equation}
Remark that $T$ is a finite number but grows to infinity as $\psi \to 0$. Moreover, the test statistic will have optimality properties under the additional condition that $T/p \to 0$ which is equivalent to $p \psi^{1/\alpha} \to \infty$. It is obvious that in practice it might happen that $T \geq p$ and then we have no solution but to use $T=p-1$, with the inconvenient that the procedure does not behave as well as the theory predicts.
\begin{proposition}
\label{prop:esttoeplitz}
Under the null hypothesis, the test statistic $\widehat{\mathcal{A}}_n $ is centered, $\mathbb{E}_I(\widehat{\mathcal{A}}_n) =0$,  with variance :
\begin{equation*}
\Var_I (\widehat{\mathcal{A}}_n)  = \ds\frac{1}{n(n-1)(p-T)^2} .
\end{equation*}

Moreover, under the alternative hypothesis with  $ \alpha > 1/4$, if we assume that $\psi \to 0$ we have:
\[
\mathbb{E}_{\Sigma}(\widehat{\mathcal{A}}_n) = \ds\sum_{j=1}^T w^*_j \sigma_j^2  \geq b(\psi) \quad  \text{ and } \quad \Var_{\Sigma}(\widehat{\mathcal{A}}_n) = \ds\frac{R_1}{n(n-1)(p-T)^4} + \ds\frac{R_2}{n(p-T)^2},
\]
uniformly over $\Sigma$ in $ G(\mathcal{T}(\alpha, L), \psi) $, where
\begin{eqnarray}
\label{R_1}
R_1 &\leq & (p-T)^2 \cdot \{ 1 + o(1) +  \mathbb{E}_\Sigma(\widehat{\mathcal{A}}_n) \cdot (O(\ds\sqrt{T}) + O( T^{3/2 - 2 \alpha})) + \mathbb{E}^2_\Sigma(\widehat{\mathcal{A}}_n) \cdot O(T^2)  \}\\
\label{R_2}
R_2 &\leq & \!\!\! (p-T) \cdot \{\mathbb{E}_{\Sigma} (\widehat{\mathcal{A}}_n) \cdot o(1)  +  \mathbb{E}^{3/2}_{\Sigma}(\widehat{\mathcal{A}}_n ) \cdot (O(T^{1/4})+ O(T^{3/4- \alpha})) +\mathbb{E}^2_{\Sigma}(\widehat{\mathcal{A}}_n) \cdot O(T)  \}.
\end{eqnarray}
\end{proposition}
In the next Proposition we prove asymptotic normality of the test statistic under the null and under the alternative hypothesis with additional assumptions. More precisely, we need that $\psi$ is of the same order as the separation rate and that the matrix $\Sigma$ is close to the optimal $\Sigma^*$. This is not a drawback, since the asymptotic constant for probability errors are attained under the same assumptions or tend to 0 otherwise.
 \begin{proposition}
 \label{prop:asymptoticnormality}
 Suppose that   $n,~ p \to + \infty$,  $\alpha > 1/4$, $\psi \to 0$, $p \psi^{1/\alpha} \to + \infty$ and moreover assume that $n(p-T) b(\psi) \asymp 1 $,  the test statistic $\widehat{\mathcal{A}}_n$ defined by (\ref{esttetplitz}) with parameters given in  (\ref{parameters1}), verifies :
 \[
n(p-T) \Big( \widehat{\mathcal{A}}_n -\mathbb{E}_{\Sigma}(\widehat{\mathcal{A}}_n) \Big) \longrightarrow \mathcal{N}(0,1)
 \]
for all $\Sigma \in G( \mathcal{T}(\alpha, L), \psi)$, such that $\mathbb{E}_{\Sigma}(\widehat{\mathcal{A}}_n) = O(b(\psi))$.

Moreover,  $n(p-T) \widehat{\mathcal{A}}_n $ has asymptotical $\mathcal{N}(0,1)$ distribution under $H_0$, as $p\to \infty$ for any fixed $n\geq 2$.

 \end{proposition}

  \subsection{Separation  rate and sharp asymptotic optimality}
  \label{theoandproof}
 Based on the test statistic $\widehat{\mathcal{A}}_n$, we define the test procedure
\begin{equation}
\label{testtoeplitz*}
  \chi^* = \chi^*(t) = \mathds{1} ( \widehat{\mathcal{A}}_n > t) ,
 \end{equation}
for conveniently chosen $t >0$, where $\widehat{\mathcal{A}}_n$ is the estimator defined
in (\ref{esttetplitz}) with parameters in (\ref{parameters1}).

The next theorem gives the separation rate under the assumption that $T=o(p)$, or equivalently, that $p \psi^{1/ \alpha} \to \infty $. The upper bounds are attained for arbitrary $ \alpha > 1/4$, but the lower bounds require $\alpha >1$.
\begin{theorem}
\label{theo:optimalrates}
Suppose that asymptotically
\begin{equation}
\label{conditionasym}
\psi \to 0  \quad \text{  and  } \quad p \psi^{1/\alpha} \to  + \infty
\end{equation}

{\bf Lower bound. }
$
\text{If }~ \alpha >1 \quad  \mbox{ and } \quad n^2 p^2 \, b^2(\psi) = C( \alpha ,L)n^2 p^2 \, \psi^{\frac{4 \alpha +1 }{\alpha}}  \to 0 \quad
$
then 
\[
\gamma = \inf\limits_{\chi } \gamma( \, \chi \, , G(\mathcal{T}(\alpha, L), \psi)) \longrightarrow 1,
\]
where the infimum is taken over all test statistics $\chi$.

{ \bf Upper bound. } The test procedure $\chi^*$ defined in (\ref{testtoeplitz*}) with $t>0$ has the following properties:

\noindent
 Type I error probability :  if $np \cdot t \to + \infty $ then $\eta(\chi^*) \to 0 $.

 \noindent
 Type II error probability : if
\begin{equation}
\label{conditionbornesup1}
 \alpha >1/4 \quad \mbox{ and } \quad n^2 p^2 \, b^2(\psi) = C( \alpha ,L)n^2 p^2 \, \psi^{\frac{4 \alpha +1 }{\alpha}}  \to + \infty
\end{equation}
then, uniformly over t such that $t \leq c \cdot C^{1/2}( \alpha ,L) \cdot \psi^{\frac{4 \alpha +1 }{2\alpha}}  $ , for some constant  $ 0 < c < 1$, we have
\[
\beta( \, \chi^*, G( \mathcal{T}(\alpha, L), \psi)) \longrightarrow 0.
\]
\end{theorem}

Under the assumptions given in (\ref{conditionasym}) and (\ref{conditionbornesup1}), with $t$ verifying the assumptions of Theorem~\ref{theo:optimalrates}, we get :
\[
\gamma ( \,  \chi^* \, , G(\mathcal{T}(\alpha, L), \psi)) \longrightarrow 0
\]

As a consequence of the previous theorem, we get that $\chi^*$ is an asymptotically minimax  test procedure  if $\psi/\widetilde{\psi} \longrightarrow + \infty$. From the lower bounds we deduce that, if $\psi / \widetilde{\psi} \longrightarrow 0$,  there is no test procedure to  distinguish between the null and the alternative hypotheses, with errors tending to $0$. The minimax separation rate  $ \widetilde{\psi}$ is therefore :
\begin{equation}
\label{sharprate}
\widetilde{\psi} =\left( \ds\frac{2 \alpha +1}{L^{\frac{1}{2\alpha}}(4\alpha +1)^{1 +\frac{1}{2 \alpha } }} \cdot  n^2  p^2\, \right)^{-\frac{\alpha }{4\alpha+1}}
\end{equation}
It is obtained from the relation $n^2p^2b^2(\psi) =1$.  Naturally the constant does not play any role here. Remark that  the condition $ T/p \to 0 \asymp p \widetilde{\psi}^{1/\alpha} \to + \infty$  implies that $n=o(p^{2\alpha -\frac 12})$.

\smallskip

The maximal type II error probability either tends to 0, see Theorem~\ref{theo:optimalrates}, or is less than $\Phi(np(t-b(\psi)))+o(1)$ when $np t < np b(\psi) \asymp 1$. The latter case is the object of the next theorem giving sharps bounds for the asymptotic errors. The upper bounds are attained for arbitrary $n \geq 2$ and for $\alpha >1/4$, while our proof of the sharp lower bounds requires additionally that $n \to \infty$ and $\alpha >1$.


\begin{theorem}
 \label{theo:sharprates}
Suppose that $ \psi \to 0$ such that $p/T \asymp p\psi^{1/ \alpha} \to + \infty $ and, moreover, that
\begin{equation}
\label{conditionbornesup2}
 \text  n^2 p^2 \, b^2(\psi) \asymp 1.
\end{equation}

{ \bf Lower bound.} If  $\alpha >1$,
then
\[
\inf\limits_{ \chi: \eta(\chi) \leq w }  \beta ( \, \chi \, , G ( \mathcal{T}(\alpha, L) , \psi)) \geq \Phi(z_{1-w} - np b(\psi) ) +o(1),
\]
where the infimum is taken over all test statistics $\chi$ with type I error probability less than or equal to $w$.
Moreover,
\[\gamma =\inf\limits_{ \chi }  \gamma ( \,  \chi  \,, G ( \mathcal{T}(\alpha, L) , \psi)) \geq 2 \Phi(-n p \,\frac{b(\psi)}2) + o(1).
\]

{ \bf Upper bound. } The test procedure $\chi^*$ defined in (\ref{testtoeplitz*}) with $t>0$ has the following properties.

\noindent
 Type I error probability :  $  \eta(\chi^*) = 1 - \Phi(np \cdot t) + o(1)$.

\noindent
 Type II error probability : under the  assumption (\ref{conditionbornesup2}),  and for all $\alpha > 1/4$,  we have that, uniformly over $t$ :
\[
\beta( \, \chi^*, G( \mathcal{T}(\alpha, L), \psi)) \leq \Phi(n p \cdot ( t - b(\psi)))+o(1).
\]

\end{theorem}
 In particular, for $t=t^w$, such that $np \cdot t^w = z_{1-w} $, we have  $ \eta(\chi^*(t^w)) \leq w+o(1)$
and  also,
\[
\beta( \, \chi^*(t^w), G ( \mathcal{T}(\alpha, L) , \psi)) = \Phi(z_{1-w } - np \cdot b(\psi))+o(1).
\]
Another important  consequence of the previous theorem, is that the test procedure $\chi^*$, with $t^* = b(\psi)/2$ is such that
\[
\gamma ( \,  \chi^*(t^*) \, , G ( \mathcal{T}(\alpha, L) , \psi))  = 2 \, \Phi\left(- n p \,\frac{b(\psi)}2 \right)+o(1) .
\]
Then we can deduce that the minimax separation rate $\widetilde{\psi}$ defined in (\ref{sharprate}) is sharp.


\section{Exponentially decreasing covariances }
\label{sec:testprob2}
In this section we want to test  (\ref{H_0}) against (\ref{H1E}), where the alternative set is $G(\mathcal{E}(A,L) , \psi)$, for some $A,L, \psi >0$.
It is well known in the nonparametric minimax theory that $\mathcal{E}(A,L)$ is in bijection with ellipsoids of analytic spectral densities admiting analytic continuation on the strip $\{z \in \mathbb{C}: |Im(z)|\leq A\}$ of the complex plane. On this class nearly parametric rates are attained for testing in the Gaussian noise model, see Ingster~\cite{IngsterI93}.

Let us define  $\widehat{\mathcal{A}}_n ^{\mathcal{E}} $ in \eqref{esttetplitz}
\begin{equation}
\label{esttoeplitz2}
\widehat{\mathcal{A}}_n ^{\mathcal{E}} = \ds\frac{1}{n(n-1)(p- T)^2} \underset{1 \leq k \neq l \leq n}{\ds\sum }\sum_{j=1}^{T} w_j^* \underset{ T + 1 \leq i_1, i_2 \leq p}{\ds\sum } X_{k, i_1}X_{k, i_1-j}X_{l, i_2}X_{l, i_2-j} ,
\end{equation}
where the  weights $ \{w_j^*\}_{j \geq 1}$, are  obtained by solving the optimization problem (\ref{probopt}), with  the class $\mathcal{T}(\alpha , L)$ replaced by $\mathcal{E}(A,L)$. The solution given in  \cite{IngsterI93} is as follows :
\begin{equation}\label{parameters}
\begin{array}{lcl}
w_j^* &=& \ds\frac{\lambda}{2b(\psi)} \Big( 1 - (\frac{e^{j}}{e^{T}})^{2A} \Big)_+, \quad 
\sigma_j^* =   \ds\sqrt{\lambda} \Big( 1 - (\frac{e^{j}}{e^{T}})^{2A} \Big)_+^{1/2}, \quad
 T = \Big\lfloor \ds\frac{1}{A} \ln\Big( \frac{1}{\psi} \Big) \Big\rfloor, \\  \\
   \lambda &=&  \ds\frac{A \psi^2}{\ln \Big(\ds\frac{1}{ \psi} \Big)} \, , \quad
b^2(\psi) =  \ds\frac{A \psi^4}{2 \ln \Big(\ds\frac{1}{ \psi} \Big)}.
\end{array}
\end{equation}
Note that all parameters above are free of the radius $L>0$. Moreover, we have :
\[
\sup\limits_j w_j^* \leq \ds\frac{\lambda }{2b(\psi)} \asymp \ds\frac{1}{2(\ln (1/\psi))^{1/2}} \longrightarrow 0
\]

Under the null hypothesis, we still have $\mathbb{E}_I(\widehat{\mathcal{A}}_n^{\mathcal{E}}) = 0 \, , \Var_{I}(\widehat{\mathcal{A}}_n^{\mathcal{E}})= 1/(n(n-1)(p- T)^2)$ \, and
\[
n(p- T)  \widehat{\mathcal{A}}_n ^{\mathcal{E}} \stackrel{\mathcal{L}}{\longrightarrow} \mathcal{N}(0,1) \quad \text{ for fixed } n \geq 2 \text{ and } p \to + \infty. 
\]
In the following proposition, we see how the upper bounds of the variance have changed under $\Sigma$ in $G(\mathcal{E}(A,L), \psi)$.
\begin{proposition}
\label{prop:est2}

 Under the alternative,  for all $ \Sigma \in G(\mathcal{E}(A,L) , \psi) $, we have :
\[
 \mathbb{E}_{\Sigma}(\widehat{\mathcal{A}}_n^{\mathcal{E}}) = \sum_{j=1}^{T} w_j^* \sigma_j^2 \geq b(\psi) \quad  \text { and }  \quad \Var_{\Sigma} (\widehat{\mathcal{A}}_n ^{\mathcal{E}} ) = \ds\frac{R_1}{n(n-1)(p-T)^4} + \ds\frac{R_2}{n(p-T)^2}
\]
where, for all  $A > 0$, and as $\psi \longrightarrow 0$ :
\begin{eqnarray}
\label{R'_1}
R_1 &\leq & (p- T)^2 \cdot \{ 1 + o(1) +  \mathbb{E}_\Sigma(\widehat{\mathcal{A}}_n ^{\mathcal{E}}) \cdot O(\ds\sqrt{T})  + \mathbb{E}^2_\Sigma(\widehat{\mathcal{A}}_n^{\mathcal{E}}) \cdot O(T^2)  \}\\
\label{R'_2}
R_2 &\leq & \!\!\! (p- T) \cdot \{\mathbb{E}_{\Sigma} (\widehat{\mathcal{A}}_n^{\mathcal{E}}) \cdot o(1)  +  \mathbb{E}^{3/2}_{\Sigma}(\widehat{\mathcal{A}}_n ^{\mathcal{E}}) \cdot O( T^{1/4}) +\mathbb{E}^2_{\Sigma}(\widehat{\mathcal{A}}_n^{\mathcal{E}}) \cdot O(T)  \}
\end{eqnarray}
Moreover, if   $n(p- T) b(\psi) \asymp 1$, we show that $n(p- T) ( \widehat{\mathcal{A}}_n ^{\mathcal{E}}- \mathbb{E}_{\Sigma}(\widehat{\mathcal{A}}_n^{\mathcal{E}})) \to \mathcal{N}(0,1)$, for all $\Sigma  \in \mathcal{E}(A,L)$, such that $\mathbb{E}_{\Sigma}(\widehat{\mathcal{A}}_n^{\mathcal{E}}) = O(b(\psi))$.
\end{proposition}
Now we define the test procedure as follows,
\[
\Delta^*= \Delta^*(t) = \mathds{1} ( \widehat{\mathcal{A}}_n^{\mathcal{E}} > t).
\]

We describe next the separation rate. We stress the fact that Lemma~\ref{lem1} shows that the optimal sequence $\{\sigma^*_j\}_j$ in (\ref{parameters}) provides a Toeplitz  positive definite covariance matrix.
The sharp results are obtained under the additional assumption that $\psi \asymp \widetilde \psi$ and the lower bounds require that $n$ tends also to infinity.

\begin{theorem}
\label{theoremanalyticaltern}
Suppose that asymptotically  $\psi \to 0$ and  $p/T \asymp p/\ln(1/\psi) \to \infty$.

{\bf 1. Separation rate.} ~ {\bf Lower bound:}
$
 \text{if  } ~~  n^2 p^2 b^2(\psi) =
n^2 p^2 \cdot  A \psi^4 / (2 \ln (1/ \psi))   \longrightarrow 0 ~
$
then 
\[
\gamma = \inf\limits_{\Delta } \gamma( \, \Delta \,,G( \psi)) \longrightarrow 1,
\]
where the infimum is taken over all test statistics $\Delta$.

{\bf Upper bound:} the test procedure $\Delta^*$ defined  previously with $t>0$ has the following properties:

 \noindent Type I error probability:  if $np \cdot t \to + \infty $ then $\eta(\Delta^*) \to 0 $.

 \noindent  Type II error probability: $\text{if } ~~\quad  n^2 p^2 \,  b^2 (\psi)=n^2 p^2 \cdot  A \psi^4 /(2 \ln (1/ \psi)) \longrightarrow + \infty $ 
then, uniformly over t such that $t \leq c \cdot A^{\frac{1}{2}} \psi^2 / (2 \ln (1/ \psi))^{\frac{1}{2}} $ , for some constant $c \, ; \, 0 < c < 1$, 
\[
\beta( \, \Delta^*, G( \psi)) \longrightarrow 0.
\]

{\bf 2. Sharp asymptotic bounds.} ~ {\bf Lower bound:}  suppose that $n \to +\infty$ and that
\begin{equation}
\label{conditionsharprate}
  \quad n^2 p^2 \, b^2(\psi) \asymp 1,
 \end{equation}
then we get 
$
\inf\limits_{ \Delta: \eta(\Delta) \leq w }  \beta ( \Delta \, , G( \psi)) \geq \Phi(z_{1-w} - np b(\psi) ) +o(1),
$
where the infimum is taken over all test statistics $\Delta$ with type I error probability less than or equal to $w$ for $w \in (0,1)$.
Moreover,
\[\gamma =\inf\limits_{ \Delta }  \gamma ( \,  \Delta \,  , \psi)) \geq 2 \Phi(-n p \,\frac{b(\psi)}2) + o(1).
\]

{\bf Upper bound:} we have

\noindent Type I error probability : $\eta(\Delta^*) = 1 - \Phi(npt) + o(1)$.

\noindent
 Type II error probability :
under the condition (\ref{conditionsharprate}), we get that,
 uniformly over t,
\[
\beta( \, \Delta^* \, , G(\psi)) \leq \Phi(n p \cdot ( t - b(\psi)))+o(1).
\]
\end{theorem}
In particular, the test procedure $\Delta^*(b(\psi)/2)$, is such that
$
\gamma ( \, \Delta^*(b(\psi)/2) \, , G(  \psi)) = 2 \Phi(-n p \,\frac{b(\psi)}2) + o(1).
$
We get the sharp minimax separation rate :
$
 \widetilde{\psi} = \Big( \ds\frac{2\ln (n^2 p^2 )}{ An^2 p^2} \Big)^{1/4 }.
$
 Remark that, in this case the  condition  $T/p \to 0$ implies that $n =o(e^p)$, which is considerably  less restrictive than the condition $n=o(p^{2\alpha -\frac 12}) $ of the previous case and allows for exponentially large $n$, e.g. $n = e^{p/2}$.

\section{Numerical implementation and extensions} \label{sec:simu}


In this section we implement the test procedure $\chi$ in (\ref{testtoeplitz*}) with empirically chosen threshold $t>0$ and study its numerical performance over two families of covariance matrices. We estimate the type I and type II errors by Monte Carlo sampling with 1000 repetitions. First, we choose $\Sigma= \Sigma(M)=[\sigma_j]_j$ ; $\sigma_j= j^{-2}/M$ under the  alternative hypothesis, for various values of $M \in \{ 2, 2.5, 3, 4, 6, 8, 16, 30, 60, 80 \}$. We implement the test statistic $ \widehat{\mathcal{A}}_n^{\mathcal{T}} $ defined in \eqref{esttetplitz} and \eqref{parameters1}, for parameters $\alpha=1, L=1$ and $\psi =\psi(M) = \Big( \sum_{j=1}^{p-1} j^{-4} \Big)^{ \frac 12} / M$. Our choice of the values for $M$ provides positive definite matrices. We denote by $A(M)$ the random variable  $n(p-T) \widehat{\mathcal{A}}_n^{\mathcal{T}}$ when  $\Sigma =\Sigma(M)$, and by $A(0)$ when $\Sigma =I$. Note that large values of $M$ give $\Sigma(M)$ with small off-diagonal entries, which is very close to the identity matrix.
\begin{center}
\begin{figure}[hbp!]
\includegraphics[width=10cm , height=6cm]{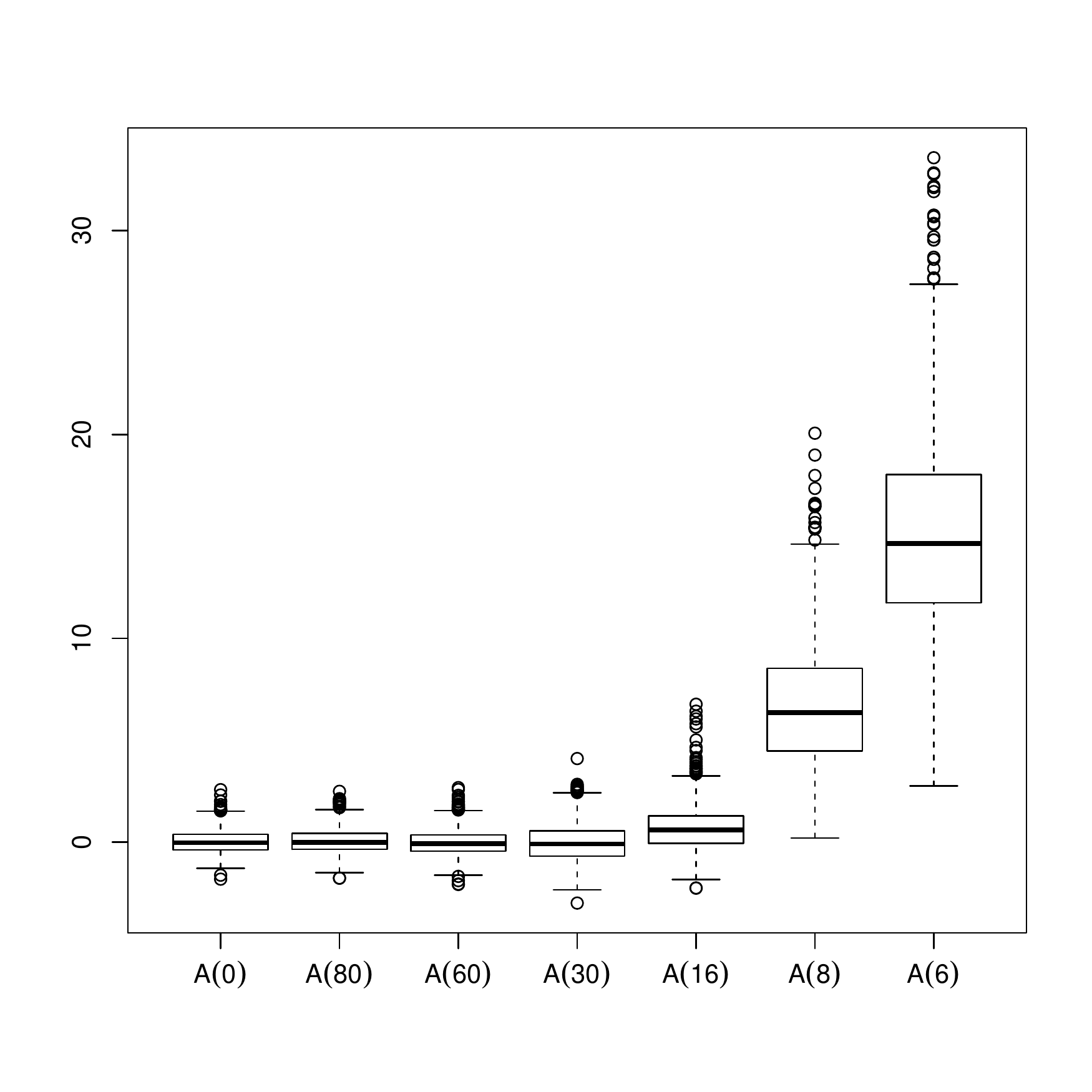} 
\caption{Distributions of $ A(M) = n(p-T)\widehat{\mathcal{A}}_n^{\mathcal{T}} $ for $I = \Sigma (0) $ and $\Sigma=\Sigma(M)$, when $p=60$ and $n=40$.}
\label{boxplotdeest}
\end{figure}
\end{center}
Figure \ref{boxplotdeest},  shows that $ n(p-T)\widehat{\mathcal{A}}_n^{\mathcal{T}} $  is distributed as a standard normal random variable, when $\Sigma =I$ and $\Sigma(M)$ close  enough to the identity. And as a non-centered normal distribution when $\Sigma(M)$ is far from the identity matrix.

To evaluate the performance of our test procedure  we compute it's power. For each value of $n$ and $p$, we estimate the 95th percentile $t$ of the distribution of $n(p-T) \widehat{\mathcal{A}}_n^{\mathcal{T}}$ under the null hypothesis $\Sigma =I$. We use  $t$  previously defined to estimate the type II error probability, and then plot the associated power. In Figure \ref{plusieursp}, we plot the  power function of our test procedure $\chi$-test as function of $\psi(M)$, for a fixed value of $n$ and different values of $p$.
 \begin{center}
 \begin{figure}
\includegraphics[width= 10cm, height=7cm]{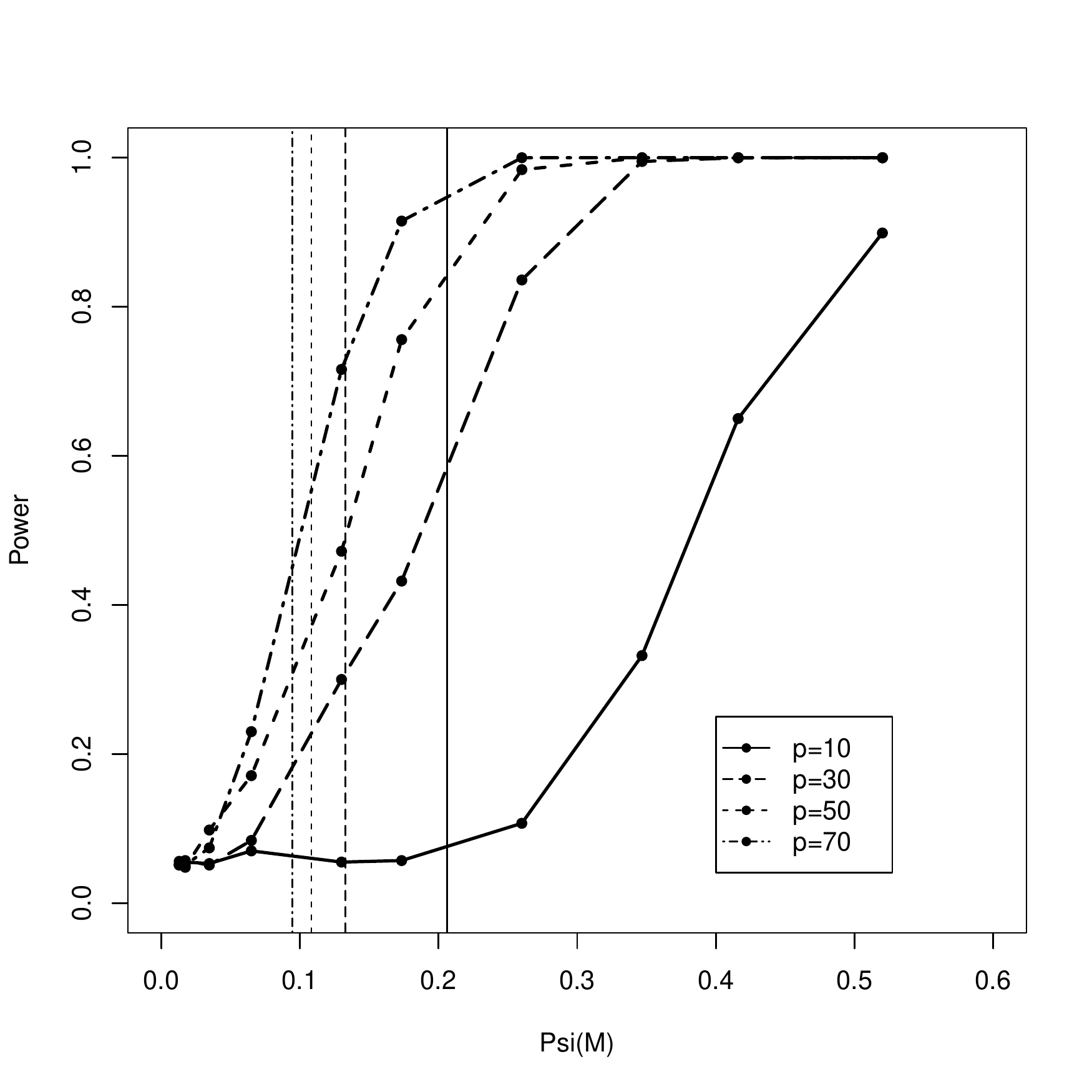}
\caption{Power curves of the $\chi$-test as function of $\psi(M)$ for $n=10 $ and $p \in \{10, 30, 50, 70\}$}
\label{plusieursp}
\end{figure}
\end{center}
\begin{center}
\begin{figure}[hb!]
\includegraphics[width= 5cm, height=5cm]{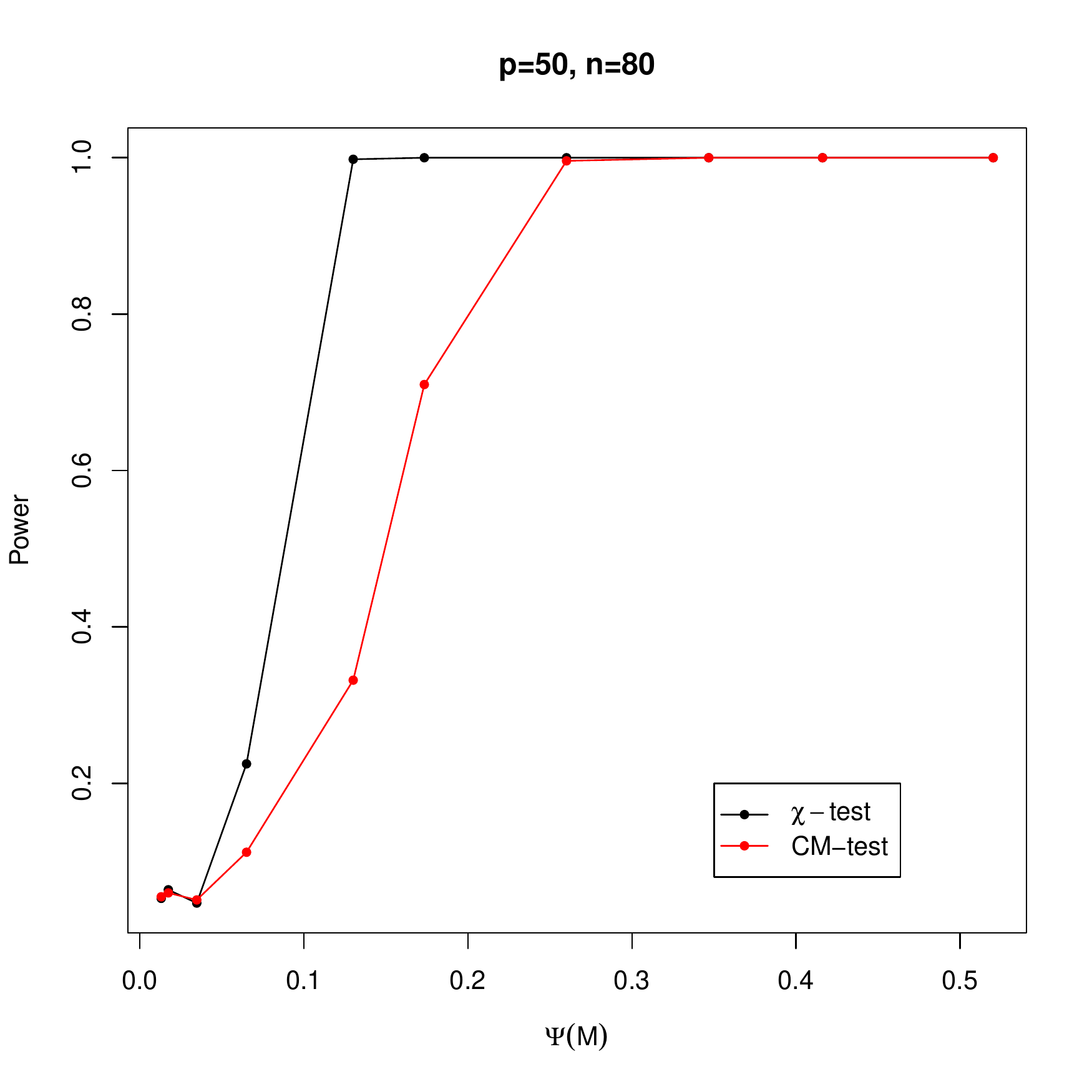}
\includegraphics[width= 5cm, height=5cm]{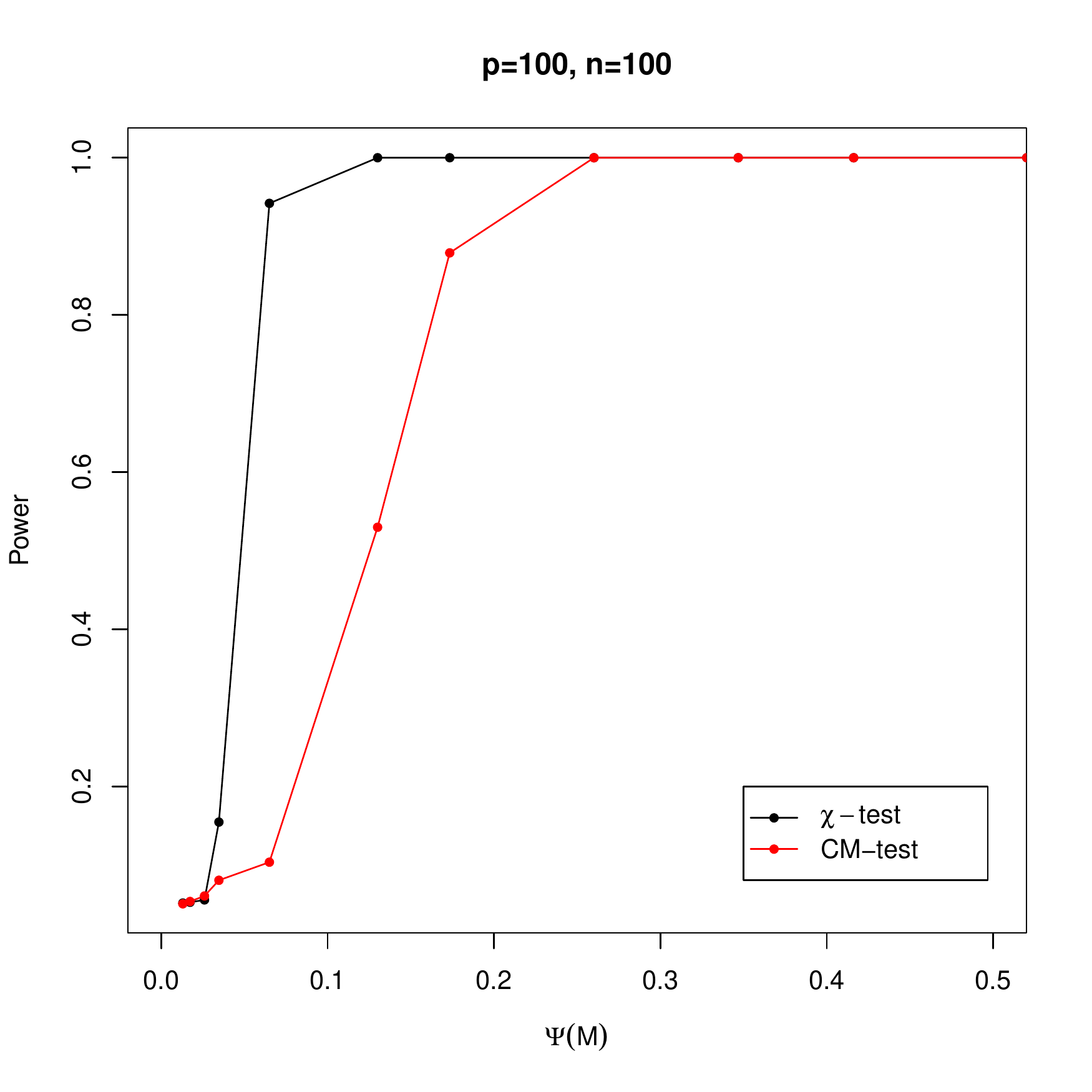}
\includegraphics[width= 5cm, height=5cm]{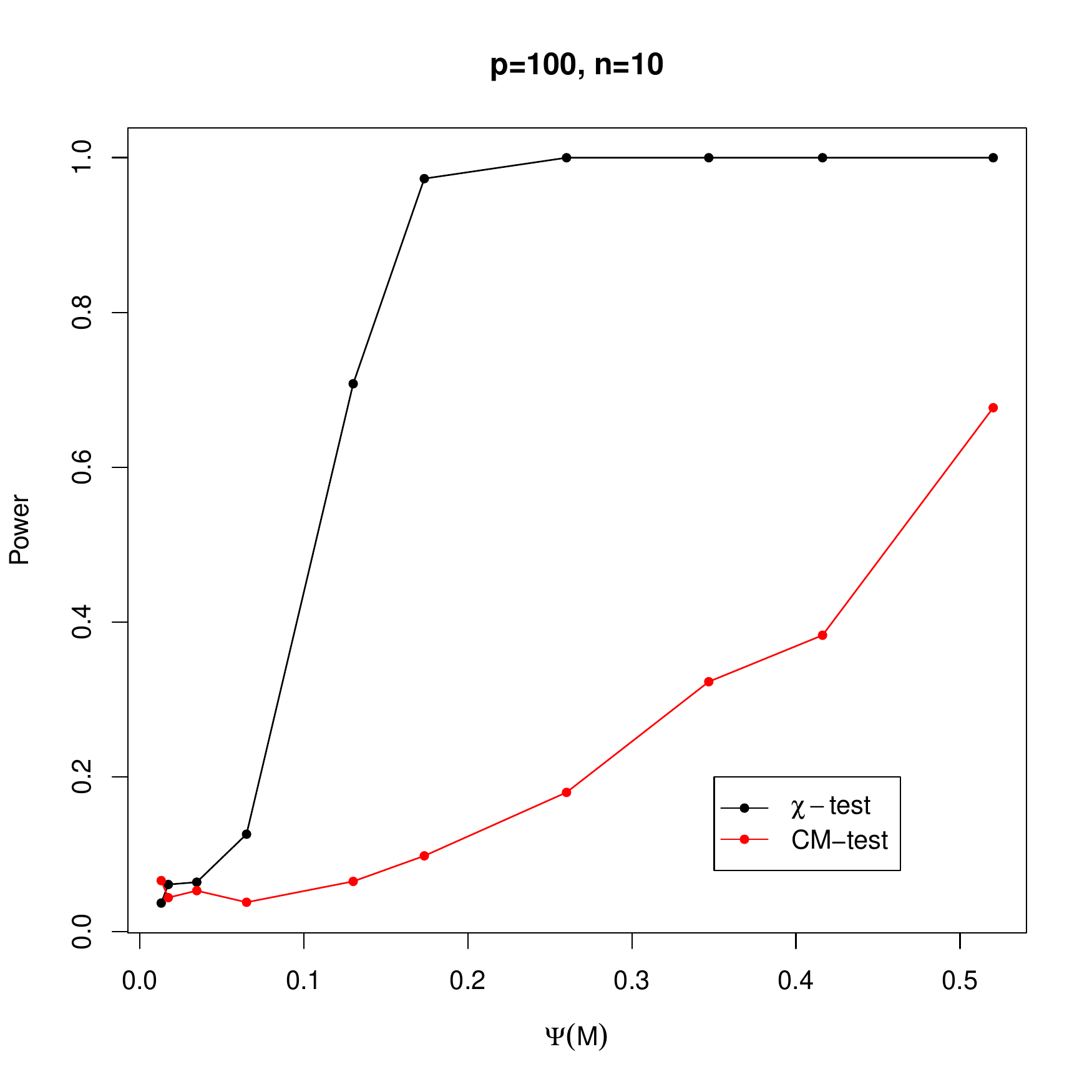}
\caption{Power curves of the $\chi$-test  and the CM-test as functions of $\psi(M)$, when the alternative consists of matrices whose elements decrease polynomially when moving away from the main diagonal }
\label{polynomiallydecrease}
\end{figure}
\end{center}
The vertical lines in figure \ref{plusieursp} represent the different $ \tilde{\psi} (n,p)$ associated to different values of $p$ and  $n=10$. We remark that, on the one hand the power grows with $\psi(M)$ for all $ p \in \{10, 30, 50, 70 \}$. On the other hand the power is an increasing function of $p$ for a fixed covariance matrix $\Sigma(M)$.

We also compare our test procedure with the one defined in  \cite{CaiMa13}. Recall that the test statistic defined by  \cite{CaiMa13} is given by:
\[
\widehat{T}_n^{CM} = \frac{2}{n(n-1)} \underset{1 \leq k < l \leq n}{\ds\sum \sum} \Big( (X_k^\top X_l)^2 -  X_k^\top X_k - X_l^\top X_l + p \Big).
\]
Note that for matrices $\Sigma \in \mathcal{T}(1, 1)$, we have  $(1/p) \| \Sigma - I \|_F^2  \sim \sum_{j=1}^{p-1} \sigma_j^2$, thus we implement   $ \widehat{T}_n^{CM}/p $ as CM-test statistic. To have fair comparison,  we estimate the 95th percentile under the null hypothesis for both tests.
Figures \ref{polynomiallydecrease}, shows that when $n$ is bigger than or equal to $p$  the powers of the $\chi$-test and the CM-test take close values. While when $n$ is smaller then $p$, the gap between the power values of the two tests is large, and the $\chi$-test is more powerful than the CM-test.

Second, we consider tridiagonal matrices under the alternative. We define  $\Sigma=\Sigma(\rho) =[\sigma_j]_j$ ; $\sigma_j= \rho \cdot \mathds{1} \{ j =1 \}$, for $\rho \in (0,1)$. In this case the parameter $\psi$ is $\psi(\rho) = \rho$, for a grid of 10 points $\rho$ belonging to the interval $(0, 0.35]$ and as previously we take $\alpha =1$ and $L=1$.
\begin{center}
\begin{figure}[ht!]
\includegraphics[width= 5cm, height=5cm]{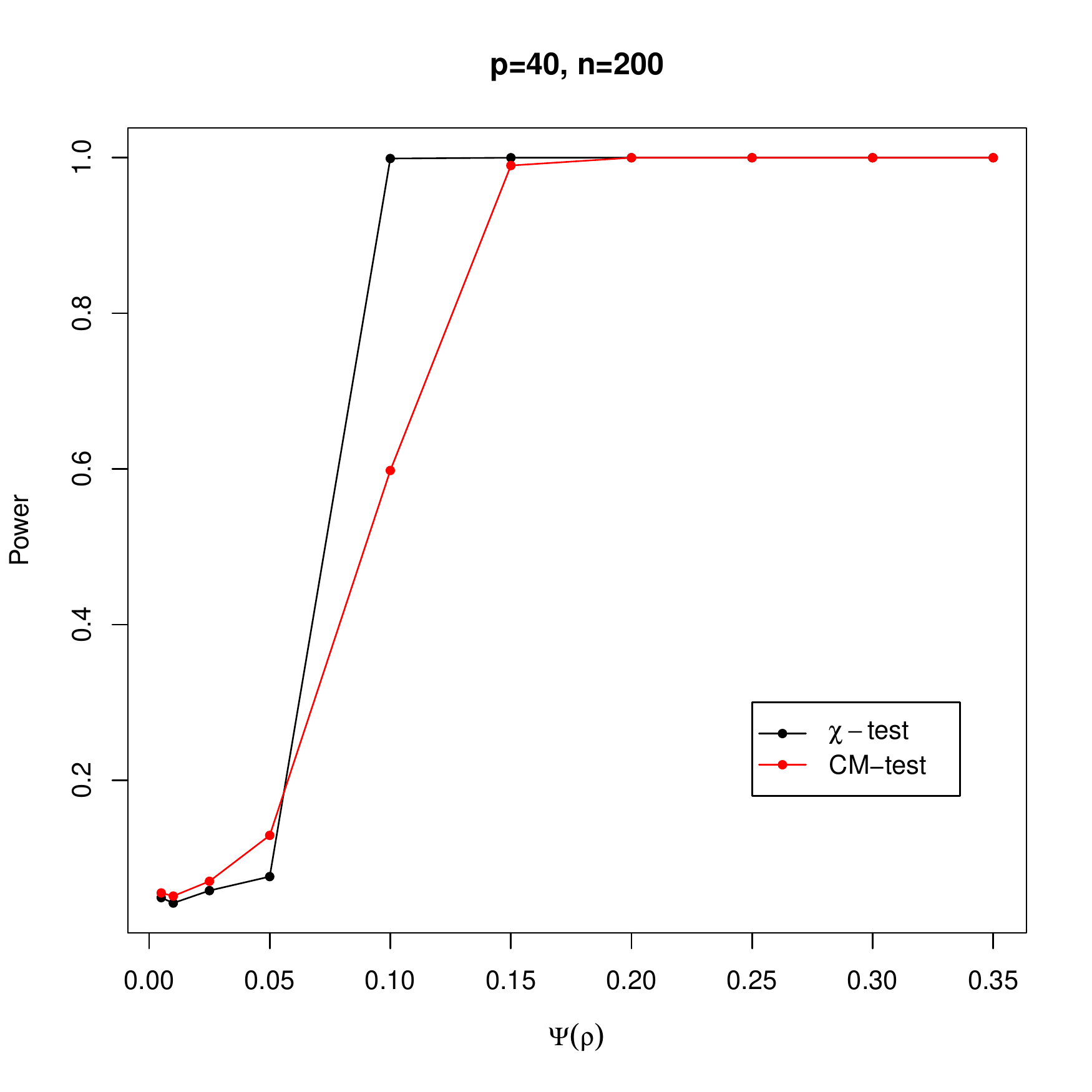}
\includegraphics[width= 5cm, height=5cm]{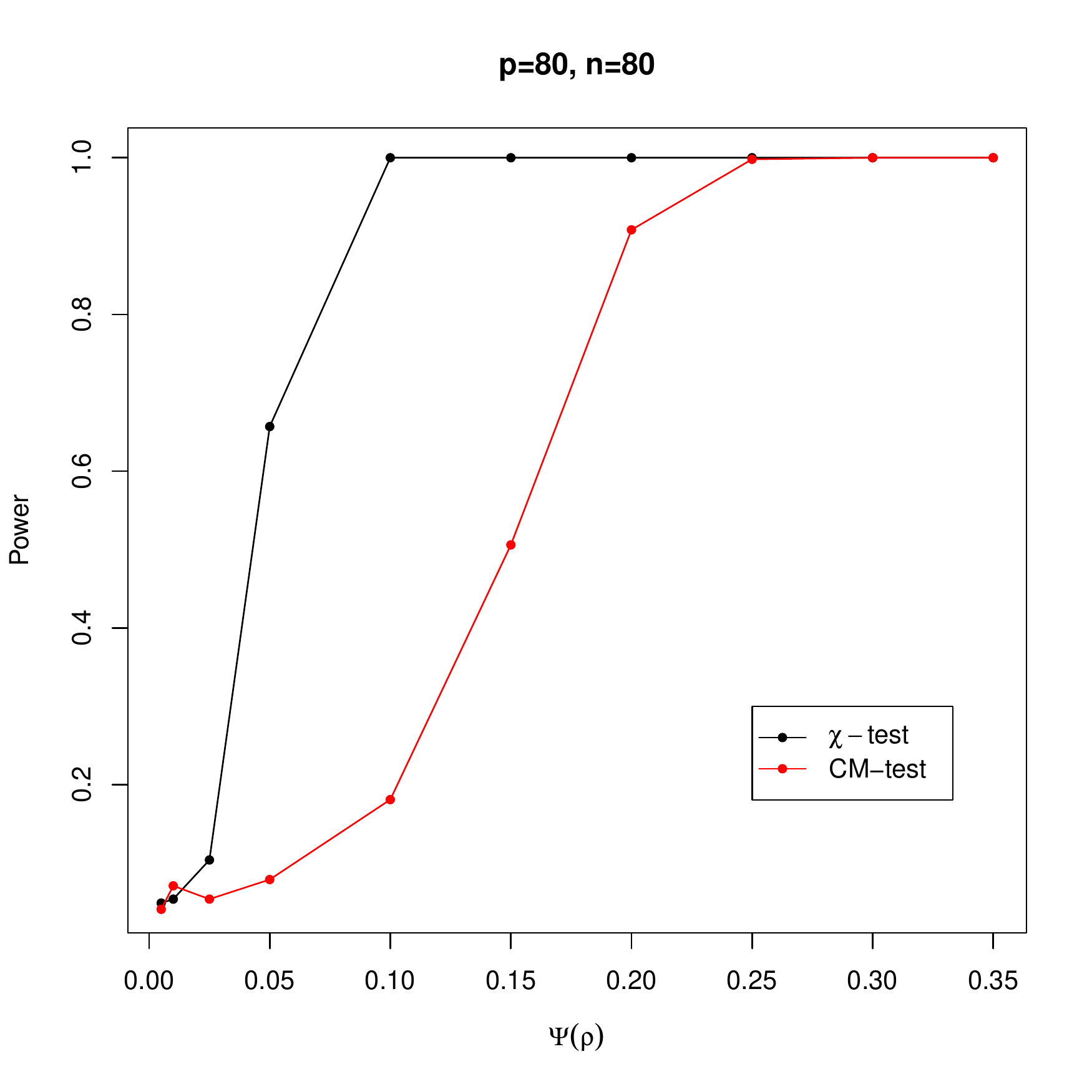}
\includegraphics[width= 5cm, height=5cm]{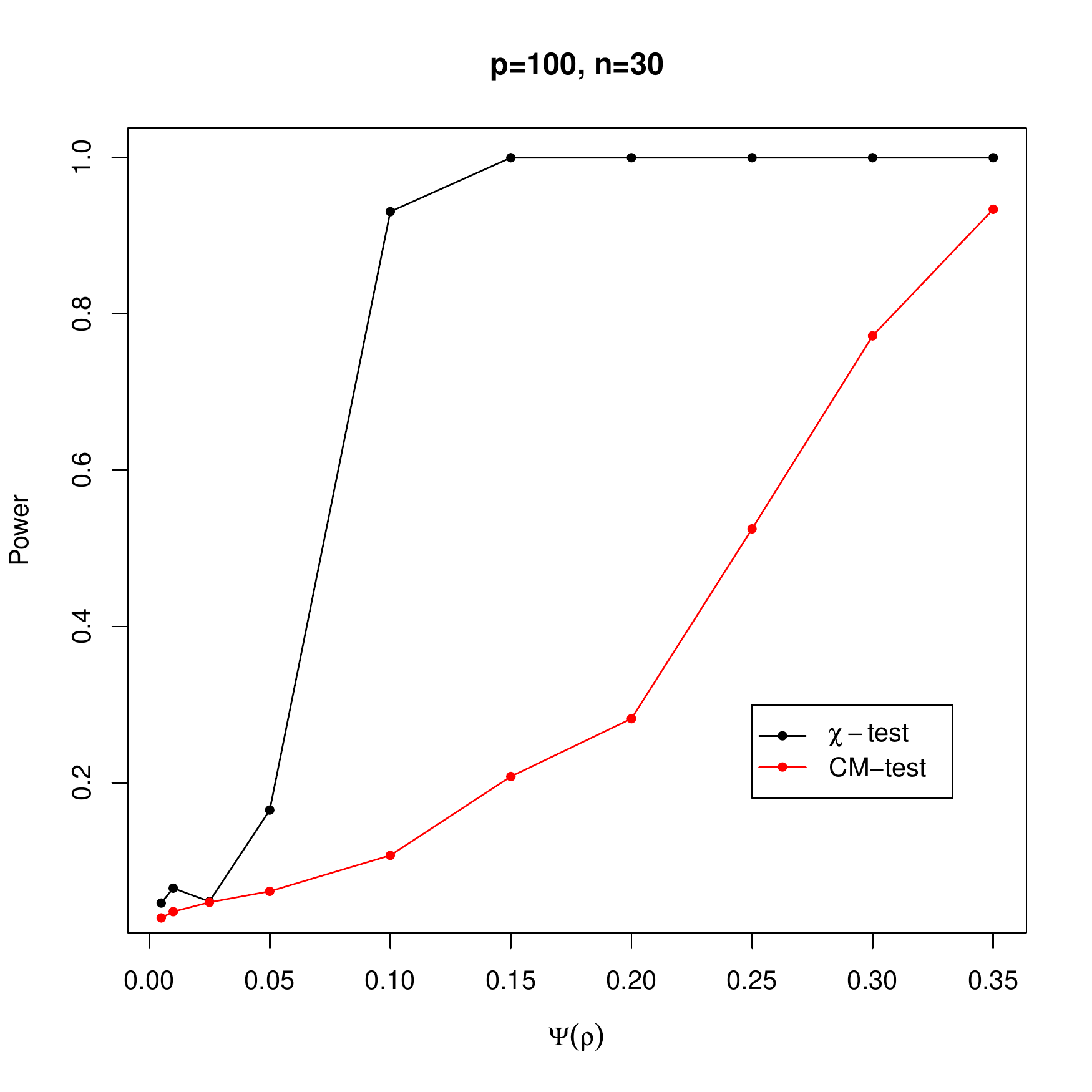}
\caption{Power curves of the $\chi$-test  and the CM-test as functions of $\psi(\rho)$, when the alternative consists of tridiagonal  matrices }
\label{tridiagonal}
\end{figure}
\end{center}
Figure \ref{tridiagonal} shows that, the $\chi$-test performs better than the U-test, in the three cases : $p$ smaller than $n$, $p$ equal $n$ and $p$ larger than $n$. Moreover,  we see that  the power curves of the $\chi$-test and the CM-test are closer, when the ratio $p/n$ is  smaller.
We expect even better results in this particular example if we use a larger value of $\alpha$, or the procedure defined by \eqref{esttoeplitz2} and \eqref{parameters}. The question arises of a test statistic free of parameters $\alpha$, respectively $A$, which is beyond the scope of this paper.


\section{Proofs}
\label{sec:proofs}

\begin{proof}[Proof of Theorems \ref{theo:optimalrates} and \ref{theo:sharprates}]
Recall the assumptions $n,\, p \to +\infty$, $\psi \to 0 $ and $T/p \asymp 1/(p \psi^{1/\alpha}) \to 0$.

\textbf{Lower bounds} :
%
In order  to show the lower bound, we  first reduce the set of parameters to a convenient parametric family. Let $\Sigma^*=T_p(\{\sigma_{k}^*\}_{k \geq 1}  )$ be the Toeplitz matrix  such that,
\begin{equation}
\label{sigma*}
 \sigma_{k}^*= \sqrt{\lambda}\left( 1- (\ds\frac{k}{T})^{2\alpha}\right)_+^{\frac{1}{2}}
 \text{ ~~~ for }  1 \leq k \leq p-1,
 \end{equation}
with $\lambda$  and $T$ are given by (\ref{parameters1}).

Let us define
$G^*$  a subset of $G ( \mathcal{T}(\alpha, L) , \psi)$ as follows
\[
G^* = \{ \Sigma^*_U : \Sigma^*_U = T_p( \{u_{k}\sigma_{k} \}_{k \geq 1} )~,  ~U \in \mathcal{U \}} ,
\]
where
\begin{equation*}
\mathcal{U} = \{U= T_p(\{u_{k}\}_{k \geq 1 }) -I_p \, \mbox{ and } \, u_{k} = \pm 1 \cdot I(k \leq T-1), \mbox{ for  } 1 \leq k \leq T-1 \}.
\end{equation*}
The cardinality of $\mathcal{U}$ is $2^{T-1}$.

From Proposition 3 in  \cite{ButuceaZgheib2014A}, we can see that if $ \alpha > 1/2 $,   for all $  U \in \mathcal{U} $, the matrix  $\Sigma^*_U $ is positive  definite, for $\psi >0$ small enough. In contrast with  \cite{ButuceaZgheib2014A}, we change the signs randomly on each diagonal of the upper triangle of $\Sigma^*$ and not of all its elements. That allows us to stay into the model of Toeplitz covariance matrices and will actually change the rates of these lower bounds.

Assume that $X_1, \dots,X_n \sim N(0, I)$ under the null hypothesis and denote by $P_I$ the likelihood of these random variables. Moreover  assume that $X_1, \dots,X_n \sim N(0, \Sigma^*_U)$ under the alternative, and we denote $P_U$ the associated likelihood. In addition let
\[
 P_{\pi} = \frac{1}{2^{T-1}} \ds\sum_{ U \in \mathcal{U}} P_U
\]
 be the average likelihood  over $G^*$.

  The problem can be reduced to the test $H_0: X_1,...,X_n \sim P_I$ against the averaged distribution $H_1: X_1,...,X_n \sim P_\pi$, in the sense that
\begin{eqnarray*}
\inf\limits_{ \chi: \eta(\chi) \leq w }  \beta ( \, \chi \, , G ( \mathcal{T}(\alpha, L) , \psi)) &=& \inf\limits_{ \chi: \eta(\chi) \leq w } ~ \sup_{\Sigma \in G ( \mathcal{T}(\alpha, L) , \psi)} \mathbb{E}_{\Sigma}(1 - \chi) \nonumber 
 \geq \inf\limits_{ \chi: \eta(\chi) \leq w } ~  \sup_{\Sigma \in G^*} \mathbb{E}_{\Sigma}(1 - \chi) \nonumber \\
 & \geq &  \inf\limits_{ \chi: \eta(\chi) \leq w }  ~ \ds\frac{1}{2^{T-1}}\mathbb{E}_{\Sigma}(1 - \chi) \nonumber 
 = \inf\limits_{ \chi: \eta(\chi) \leq w }  \mathbb{E}_{\pi}(1 - \chi) := \inf\limits_{ \chi: \eta(\chi) \leq w } \beta( \, \chi \, , \{P_{\pi}\}) 
 \end{eqnarray*}
and that
\begin{eqnarray*}
\inf\limits_{ \chi }  \gamma ( \, \chi \, , G ( \mathcal{T}(\alpha, L) , \psi)) &\geq&
\inf\limits_{ \chi }  \gamma ( \, \chi \, , \{P_\pi\})+ o(1)
\end{eqnarray*}
where, with an abuse of notation, $\beta( \, \chi \, , \{P_{\pi}\}) = \mathbb{E}_{\pi}(1 - \chi)$ and $ \gamma ( \,  \chi \, , \{P_\pi\}) = \mathbb{E}_I( \chi) + \mathbb{E}_{\pi}(1 - \chi)$.

It is therefore sufficient to show that, when $u_n \asymp 1$,
\begin{equation}\label{beta}
\inf\limits_{ \chi: \eta(\chi) \leq w }  \beta ( \chi , \{P_\pi \}) \geq \Phi( z_{1-w} -np b(\psi))) +o(1)
\end{equation}
and that
\begin{equation}\label{gamma}
\inf\limits_{ \chi }  \gamma ( \,  \chi \, , \{ P_\pi \}) \geq 2 \Phi(-n p \,\frac{b(\psi)}2) + o(1),
\end{equation}
while, for $u_n= o(1)$, we need that
\begin{equation}
\label{gammatendsto1}
\gamma ( \,  \chi \, , \{ P_\pi \} ) \to 1.
\end{equation}

\begin{lemma}\label{lemma:loglike}
Assume that $\psi \to 0$ such that $p\psi^{1/\alpha} \to \infty$ and let $f_\pi$ be the probability density associated to the likelihood $P_\pi$ previously defined. Then
\begin{equation}\label{lan}
L_{n,p}:= \log \frac{f_\pi}{f_I}(X_1,...,X_n) = u_n Z_n -\frac{u_n^2}2  +o_P(1), \mbox{ in } P_I \mbox{ probability},
\end{equation}
where $Z_n$ is asymptotically distributed as a standard Gaussian distribution and  $u_n = n p b(\psi)$ is such that either $u_n \to 0$ or $u_n \asymp 1$.
Moreover, $L_{n,p}$ is uniformly integrable.
\end{lemma}
In order to obtain (\ref{beta}) and (\ref{gamma}), we apply results in Section 4.3.1 of \cite{IngsterSuslina03} giving the sufficient condition is (\ref{lan}).

It is known that $ \gamma ( \,  \chi \, , \{ P_\pi\}) = 1 - \ds\frac{1}{2} \| P_I - P_{\pi} \|_1$ and we bound the $L_1$ norm by the Kullback-Leibler divergence
\[
\ds\frac{1}{2} \| P_I - P_{\pi} \|_1^2 \leq  K(P_I, P_{\pi}).
\]
Therefore to show (\ref{gammatendsto1}), we apply Lemma~\ref{lemma:loglike} to see that the log likelihood $\log f_\pi/f_I(X_1,...,X_n)$ is an uniformly integrable sequence.  This implies that
$K(P_I,P_\pi) = \mathbb{E}_I (\log f_\pi/f_I(X_1,...,X_n)) \to 0$.
\end{proof}

\textbf{Upper bounds} :
By the  Proposition \ref{prop:esttoeplitz}, we have that under the null hypothesis $n(p-T)\widehat{\mathcal{A}}_n \to \mathcal{N}(0,1)$ .
Then we can deduce that  the Type I error probability of $\chi^*$ has the following form :
\[
  \eta(\chi^*) = \mathbb{P} ( \widehat{\mathcal{A}}_n > t) = 1 - \Phi(npt) + o(1) .
\]
For the Type II error probability of $\chi^*$, we shall distinguish two cases, when $n^2p^2b^2(\psi)$ tends to infinity or is bounded by some finite  constant. First, assume that $\psi/ \widetilde{\psi} \to + \infty$ or, equivalently, that $n^2 p^2 b^2(\psi) \to + \infty$. Then by the Markov inequality,
\begin{eqnarray*}
   \mathbb{P}_{\Sigma} (\widehat{\mathcal{A}}_n \leq t )
    &\leq& \mathbb{P}_{\Sigma} (|\widehat{\mathcal{A}}_n - \mathbb{E}_\Sigma(\widehat{\mathcal{A}}_n)| \geq \mathbb{E}_\Sigma(\widehat{\mathcal{A}}_n) - t)
   \leq \ds\frac{\Var_{\Sigma}(\widehat{\mathcal{A}}_n)}{(\mathbb{E}_\Sigma(\widehat{\mathcal{A}}_n) - t)^2}
   \end{eqnarray*}
for all $\Sigma \in G ( \mathcal{T}(\alpha, L) , \psi)$ and $t \leq c \cdot  b(\psi)$  such that  $0<c<1$. Recall that under the alternative, we have  $\mathbb{E}_\Sigma(\widehat{\mathcal{A}}_n) \geq b(\psi)$ which gives:
\begin{equation}
\label{denominateur}
\mathbb{E}_\Sigma(\widehat{\mathcal{A}}_n) - t \geq (1-c) \mathbb{E}_\Sigma(\widehat{\mathcal{A}}_n) \geq (1-c)b(\psi).
\end{equation}
Therefore from the first part of the inequality   (\ref{denominateur}) and the variance expression of $\widehat{\mathcal{A}}_n$ under $H_1$, given in Proposition 1, we have:
   \begin{eqnarray*}
\mathbb{P}_{\Sigma} (\widehat{\mathcal{A}}_n \leq t )    &\leq& \ds\frac{R_1}{n(n-1)(p-T)^4(1-c)^2\mathbb{E}^2_\Sigma(\widehat{\mathcal{A}}_n)} + \ds\frac{R_2}{n(p-T)^2(1-c)^2\mathbb{E}^2_\Sigma(\widehat{\mathcal{A}}_n)}  := U_1 + U_2.
  \end{eqnarray*}
Let us bound from above $U_1$, using (\ref{R_1}) and the second part of the inequality (\ref{denominateur}):
\begin{eqnarray*}
U_1 & \leq &\ds\frac{1+ o(1)}{n(n-1)(p-T)^2(1-c)^2 b^2(\psi)}    + \ds\frac{O(\sqrt{T}) + O(T^{3/2 - 2\alpha})}{n(n-1)(p-T)^2 b(\psi)}+ \frac{O(T^2) }{n(n-1)(p-T)^2}. \\
 \end{eqnarray*}
We have $T^{(3/2 - 2 \alpha)} b(\psi) \asymp T^{2} b^2(\psi) \asymp \psi^{4- \frac{1}{\alpha}} = o(1) , \text{ for all } \alpha > 1/4,$
which proves that :
\[
 U_1 \leq \ds\frac{1+ o(1)}{n(n-1)(p-T)(1-c)^2b^2(\psi)} =o(1).
\]
Indeed,  $n^2(p-T)^2 b^{2}(\psi) \to + \infty $, since $n^2 p^2b^2(\psi)  \to + \infty$ and $ T/p \to 0.$

\noindent
We can check using (\ref{R_2}) that the  term $U_2$ tends to zero as well :
\[
\begin{array}{lcl}
 U_2 &\leq& \ds\frac{o(1)}{n(p-T) b(\psi)}  +   \ds\frac{O(T^{1/4})+ O(T^{3/4 - \alpha) }}{n(p-T) b^{1/2}(\psi)} + \frac{O(T) }{n (p-T)}  \\ \\
  &=& o(1) \text{ for all } \alpha > 1/4, \text{ as soon as } n^2 p^2 b^{2}(\psi) \longrightarrow + \infty.
  \end{array}
\]

Finally, when $ \psi$ is of the same order of the separation rate, i.e.  $n^2 p^2 b^2(\psi) \asymp 1 $, we may have either  $  \mathbb{E}_\Sigma(\widehat{\mathcal{A}}_n) / b(\psi) $ tends to infinity, or $\mathbb{E}_\Sigma(\widehat{\mathcal{A}}_n) = O( b(\psi)) $. In the first case  it is easy to see  that $ U_1 + U_2 \longrightarrow 0$.
In the latter the Proposition \ref{prop:asymptoticnormality} gives the asymptotic normality of $n(p-T)( \widehat{\mathcal{A}}_n -  \mathbb{E}_\Sigma( \widehat{\mathcal{A}}_n))$.
Thereby,
\begin{eqnarray*}
\sup_{\Sigma \in G ( \mathcal{T}(\alpha, L) , \psi)}\mathbb{P}_{\Sigma} (\widehat{\mathcal{A}}_n \leq t )
& \leq & \sup_{\Sigma \in G ( \mathcal{T}(\alpha, L) , \psi)}\Phi(np\cdot (t - \mathbb{E}_{\Sigma} (\widehat{\mathcal{A}}_n  ))) +o(1)\\
& \leq & \Phi (np \cdot (t - \inf_{\Sigma \in G ( \mathcal{T}(\alpha, L) , \psi)} \mathbb{E}_{\Sigma} (\widehat{\mathcal{A}}_n  ))) +o(1) \\
& = & \Phi(np  \cdot (t -b(\psi)))+o(1).
\end{eqnarray*}


\bibliographystyle{plain}



\newpage
\section{Supplementary material}

\subsection{Additional proofs for the results in Section~\ref{sec:toeplitz}}

 \begin{proof}[Proof of Lemma~\ref{lemma:loglike}]
We need to study the log-likelihood ratio:
\begin{equation*}
L_{n,p}:= \log \frac{f_\pi}{f_I}(X_1,...,X_n) = \log \mathbb{E}_U \exp\left(
- \frac 12 \sum_{k=1}^n X_k^\top ((\Sigma^*_U)^{-1}-I) X_k - \frac n2 \log \det(\Sigma^*_U)\right),
\end{equation*}
where $U$ is seen as a randomly chosen matrix with uniform distribution over the set $\mathcal{U}$.

Moreover, let us denote $\Delta_U = \Sigma^*_U - I$ which is a symmetric matrix with null diagonal. Recall that for all $  U \in \mathcal{U} $, $tr(\Delta_U) = 0$ and that $\|\Delta_U\| = O(\psi^{1-1/(2\alpha)})$. Remember also that $\sigma_k^* = 0$ for all $|k| \geq T$.

The matrix Taylor expansion gives
\begin{eqnarray*}
(\Sigma^*_U)^{-1} - I &=& -\Delta_U + \Delta^2_U + O(1) \cdot \Delta^3_U, \\
\log \det(\Sigma^*_U) &=& - \frac 12 tr(\Delta^2_U) + O(1) \cdot tr(\Delta_U^3) .
\end{eqnarray*}
On the one hand,
$
tr(\Delta_U^2) = \underset{1 \leq i \neq j \leq p}{\ds\sum } (\sigma_{|i-j|}^*)^2,
$
does not depend on $U$. Moreover,
\begin{equation}\label{trD3}
tr(\Delta_U^3) \leq \|\Delta_U\| \cdot \|\Delta_U\|^2_F = O(p \psi^{3 - \frac 1{2 \alpha}}) = O(np \psi^{2 +  \frac 1{2\alpha}} \cdot \frac{ \psi^{1 - \frac 1{\alpha} } }{n})= o(1) \quad  \text{ for } \alpha > 1.
\end{equation}
Thus we get
\begin{equation}\label{P1}
\frac n2 \log \det(\Sigma^*_U) = \frac n2 \underset{1 \leq i \neq j \leq p}{\ds\sum } (\sigma_{|i-j|}^*)^2 + o(1).
\end{equation}
On the other hand, we see that
\begin{equation}
X_k^\top \Delta_U X_k = \underset{1 \leq i,j \leq p}{\ds\sum } X_{k,i} u_{|i-j|} \sigma_{|i-j|} X_{k,j} = 2 \underset{1 \leq r <T}{\ds\sum} u_r \sigma^*_r \sum_{i=1 +r}^p X_{k,i} X_{k, i-r}
\end{equation}
and that
\begin{eqnarray}
X_k^\top \Delta_U^2 X_k &=& \underset{1 \leq i,j \leq p}{\ds\sum } X_{k,i} X_{k,j} \sum_{\substack{h=1 \\ h \not \in \{i, j \}}}^p u_{|i-h|} u_{|j-h|}
\sigma^*_{|i-h|} \sigma^*_{|j-h|} \nonumber \\
& = & \sum_{i=1}^p X_{k,i}^2 \sum_{ \substack{h=1 \\ h \neq i}}^p (\sigma^*_{|i-h|})^2 +  \underset{1 \leq i \ne j \leq p}{\ds\sum } X_{k,i} X_{k,j} \sum_{\substack{h=1 \\ h \not \in \{i, j \}}}^p u_{|i-h|} u_{|j-h|}
\sigma^*_{|i-h|} \sigma^*_{|j-h|} \nonumber \\
  &:=& S_1 + S_2 . \nonumber
\end{eqnarray}
In the term $S_2$, we change the variables $i$ and $j$ into $l=i-h$ and $m=j-h$ and due to the constraints we have  $|l| , |m| \in \{ 1, \dots , T-1 \} $ and $l \ne m $, while $h$ varies in the set $\{1 \vee ( 1-l) \vee (1-m) \, , \, p \wedge (p-l) \wedge (p-m) \}$ for each fixed pair $(l,m)$. Therefore,
\[
S_2 = \underset{ \underset{1 \leq |l| , |m| < T}{ l \neq m }}{\ds\sum }  ~ \sum_{h=1 \vee (1-l) \vee (1-m)}^{p \wedge (p-l) \wedge (p-m)} u_{|l|} u_{|m|} \sigma^*_{|l|} \sigma^*_{|m|} X_{k, l +h} X_{k, m+h} .
\]
We split the previous sums over $l \ne m$ such that sign$(l \cdot m) > 0$ and get
\[
S_{2,1} :=  \underset{1 \leq l \ne m < T}{\ds\sum } u_l u_m \sigma^*_l \sigma_m^* \Big( \sum_{h=1}^{(p-l) \wedge (p-m)}  X_{k, h+l} X_{k, h+m} + \ds\sum_{h=(1+l) \vee (1+m)}^p X_{k,h-l}X_{k,h-m} \Big)
\]
respectively, over $l,m$ of opposite signs: sign$( l \cdot m) < 0$ and get
\begin{eqnarray*}
S_{2,2} &= & 2 \underset{ 1 \leq l,m < T}{\ds\sum } \, \sum_{h=1 +m}^{p-l} u_l u_m \sigma^*_{l} \sigma_{m}^* X_{k, h+l} X_{k, h-m} \nonumber \\
&=& 2 \sum_{l=1}^{T-1} \sum_{h=1+l}^{p-l} \sigma_l^{*2}X_{k, h+l} X_{k, h-l} +  2 \underset{ 1 \leq l \ne m < T}{\ds\sum } \, \sum_{h=1 +m}^{p-l} u_l u_m \sigma^*_{l} \sigma_{m}^* X_{k, h+l} X_{k, h-m}.
\end{eqnarray*}
In conclusion, we can group terms differently and write
\begin{eqnarray}
X_k^\top \Delta_U^2 X_k &=& \underset{ 1 \leq l \ne m < T}{\ds\sum } u_l u_m \sigma^*_{l} \sigma_{m}^* \left(  \sum_{h=1}^{(p-l) \wedge (p-m)}  X_{k, h+l} X_{k, h+m} +  \ds\sum_{h=(1+l) \vee (1+m)}^p X_{k,h-l}X_{k,h-m} \right. \nonumber \\
&& \left. + 2 \sum_{h=1 +m}^{p-l}   X_{k, h+l} X_{k, h-m} \right) + \sum_{i=1}^p X_{k,i}^2  \sum_{ \substack{h=1 \\ h \neq i}}^p (\sigma^*_{|i-h|})^2 + 2 \sum_{l=1}^{T-1} \sum_{h=1+l}^{p-l} \sigma_l^{*2}X_{k, h+l} X_{k, h-l}\nonumber \\
&=& \underset{ 1 \leq l \ne m < T}{\ds\sum } u_l u_m \sigma^*_{l} \sigma_{m}^*V_{p}(l,m,k)
+ \sum_{i=1}^p X_{k,i}^2  \sum_{ \substack{h=1 \\ h \neq i}}^p (\sigma^*_{|i-h|})^2 + 2 \sum_{l=1}^{T-1} \sum_{h=1+l}^{p-l} \sigma_l^{*2}X_{k, h+l} X_{k, h-l}, \nonumber \\
\end{eqnarray}
where
$$
V_{p}(l,m,k) :=  \sum_{h=1}^{(p-l) \wedge (p-m)}  X_{k, h+l} X_{k, h+m} +  \ds\sum_{h=(1+l) \vee (1+m)}^p X_{k,h-l}X_{k,h-m}
+ 2 \sum_{h=1 +m}^{p-l}   X_{k, h+l} X_{k, h-m}.
$$
Now, let us see that:
\[
\mathbb{E}_I ( X_k^\top \Delta^3_U X_k) = \mathbb{E}_I (tr( X_k^\top \Delta^3_U X_k)) =  \mathbb{E}_I (tr( X_k X_k^\top \Delta^3_U )) =  tr(\Delta^3_U \mathbb{E}_I (X_k X_k^\top  ))= tr(\Delta^3_U)
\]
and recall (\ref{trD3}) to get
$$
\mathbb{E}_I (\sum_{k=1}^n X_k^\top \Delta^3_U X_k) = O( np  \psi^{3 - \frac 1{2 \alpha}})=o(1).
$$
Moreover, we have $\mathbb{E}_I ( X_k^\top \Delta^3_U X_k)^2 = tr^2(\Delta^3_U) + 2 tr(\Delta^6_U)$ by Proposition A.1 in \cite{ChenZZ10}, which implies that
\[
\Var_I( \sum_{k=1}^n X_k^\top \Delta^3_U X_k) =2 n tr(\Delta^6_U)
\leq 2 n \|\Delta_U\|^4 \|\Delta_U\|_F^2 = O(np \psi^{6 - 4/(2\alpha)})= o(1).
\]
Then, using Chebyshev's inequality we obtain,
\begin{eqnarray}\label{P4}
&& \sum_{k=1}^n X_k^\top \Delta^3_U X_k  = o_P(1).
\end{eqnarray}
Thus we replace (\ref{P1}) to (\ref{P4}) in $L_{n,p}$ and get
\begin{eqnarray}
L_{n,p} &=&  \log  \mathbb{E}_U \exp \left( \underset{1 \leq r < T}{\ds\sum}  u_r \sigma^*_r \sum_{i=1 +r}^p \sum_{k=1}^n X_{k,i}  X_{k, i-r}
 - \frac 12\underset{ 1 \leq l \ne m < T}{\ds\sum } u_l u_m \sigma^*_{l} \sigma_{m}^* \sum_{k=1}^nV_{p}(l,m,k) \right) \nonumber\\
& -&    \frac 12  \sum_{i=1}^p \sum_{k=1}^n X_{k,i}^2   \sum_{ \substack{h=1 \\ h \neq i}}^p (\sigma^*_{|i-h|})^2 - \underset{1 \leq l \leq T-1}{\ds\sum} \!\!\!\! \sigma_l^{*2}\sum_{h=1+l}^{p-l}  \sum_{k=1}^n X_{k, h+l} X_{k, h-l}  + \frac n4 \underset{1 \leq i \neq j \leq p}{\ds\sum } (\sigma_{|i-j|}^*)^2 + o_P(1) . \nonumber
\end{eqnarray}
 Denote  by
$  W_{l,m} : =  \sum_{k=1}^n X_{k,l} X_{k,m} $.
 Now, we evaluate the expected value with respect to the i.i.d. Rademacher variables
$u_{r}$, $u_{l} u_{m}$  for all $1 \leq r  < T$ and $1 \leq l \ne m<T$ to get
\begin{eqnarray*}
 L_{n,p} &=&
 \log \Big(\prod_{1 \leq r \leq T-1}  \cosh(\sigma^*_r \sum_{i=r+1}^p W_{i, i-r}) \Big) +  \log \Big( \underset{1 \leq l \neq m <T}{\ds\prod } \hspace{-0.2cm} \cosh \Big( \ds\frac 12 {\sigma^*_{l} \sigma^*_{m}} \sum_{k=1}^nV_{p}(l,m,k) \Big) \Big) \nonumber \\
& -&  \frac 12  \sum_{i=1}^p W_{i,i} \sum_{j:j\ne i} (\sigma^{*}_{|i-j|})^2 -  \underset{1 \leq l \leq T-1}{\ds\sum} \sigma_l^{*2}\sum_{h=1+l}^{p-l}  W_{ h+l,h-l }  + \frac n4 \underset{1 \leq i\ne j \leq p}{\ds\sum }(\sigma^{*}_{|i-j|})^2 + o_P(1). \nonumber
\end{eqnarray*}
We get that
\begin{eqnarray*}
L_{n,p}&=&  \underset{1 \leq r  \leq T-1}{\ds\sum}    \log \cosh(\sigma^*_r \sum_{i=r+1}^p W_{i, i-r})+ \underset{1 \leq l \neq m < T}{\ds \sum }  \log\cosh \Big( \ds\frac 12 {\sigma^*_{l} \sigma^*_{m}}  \sum_{k=1}^nV_{p}(l,m,k)  \Big)  \\
&- &   \frac 12  \sum_{i=1}^p W_{i,i} \sum_{j:j\ne i} (\sigma^{*}_{|i-j|})^2 -  \underset{1 \leq l \leq T-1}{\ds\sum} \sigma_l^{*2}\sum_{h=1+l}^{p-l}  W_{ h+l,h-l }   + \frac n4 \underset{1 \leq i\ne j \leq p}{\ds\sum }(\sigma^{*}_{|i-j|})^2 + o_P(1) . \nonumber\\
\end{eqnarray*}
Note that
$$
\sum_{k=1}^nV_{p}(l,m,k) =  \sum_{h=1}^{(p-l) \wedge (p-m)} \hspace{-0.5cm} W_{h+l, h+m} +  \ds\sum_{h=(1+l) \vee (1+m)}^p \hspace{-0.5cm}  W_{h-l,h-m} +  2 \!\!\!\sum_{h=1+m}^{p-l} \hspace{-0.2cm} W_{h+l, h-m}  .
$$
 We use several times the Taylor expansion  $\log \cosh (u) = \ds\frac{u^2}{2} - \frac{u^4}{12}(1 +o(1))$ for $|u| \to 0$. On the one hand, by Chebyshev's inequality,  $| \sigma_r^* \ds\sum_{i= r+1}^p W_{i, i-r} | = O_P( \ds\sqrt{\lambda np}) = O_P( \psi^{1/4 \alpha} \ds\sqrt{np b(\psi)}) = o_P(1)$, as soon as $\psi \to 0$. Then,
  \[
  \log \cosh ( \sigma_r^* \ds\sum_{i= r+1}^p W_{i, i-r} ) = \ds\frac{1}{2 } \sigma_r^{*2}(\sum_{i=r+1}^p W_{i, i-r} )^2 - \ds\frac{(1 +o_P(1))}{12} \cdot \sigma^{*4}_r (\sum_{i=r+1}^p W_{i, i-r} )^4.
  \]
  On the other hand,
  \begin{eqnarray}
   && \Big|  \ds\frac {\sigma^*_{l} \sigma^*_{m}}{2}  \Big(\sum_{h=1}^{(p-l) \wedge (p-m)} \hspace{-0.5cm} W_{h+l, h+m} +  \ds\sum_{h=(1+l) \vee (1+m)}^p \hspace{-0.5cm}  W_{h-l,h-m} +  2 \!\!\!\sum_{h=1+m}^{p-l} \hspace{-0.2cm} W_{h+l, h-m}   \Big) \Big|  \nonumber \\
  & \leq &  \ds\frac{\lambda}2  \cdot \Big| \sum_{h=1}^{(p-l) \wedge (p-m)} W_{h+l, h+m}  \Big| +  \frac{\lambda}2 \cdot  \Big|  \ds\sum_{h=(1+l) \vee (1+m)}^p \hspace{-0.5cm}  W_{h-l,h-m} \Big| + \lambda \cdot \Big| \ds\sum_{l=1+m}^{p-l} W_{h+l, h-m}  \Big|\nonumber \\[0.3cm]
  & \leq & O_p( \lambda \ds\sqrt{np}) = O_p( \psi^{1 /2 \alpha} \ds\sqrt{npb(\psi)}) = o_P(1). \nonumber
  \end{eqnarray}
Thus  we have to study now
\begin{eqnarray}
 L_{n,p} & =&  \log \frac{f_\pi}{f_I}(X_1,...,X_n) \nonumber \\
&=& \underset{1 \leq r < T}{\ds\sum}  \Big \{  \ds\frac{1}{2} \cdot  \sigma^{*2}_r (\sum_{i=r+1}^p W_{i, i-r} )^2 - \frac{(1 + o_P(1))}{12} \cdot  \sigma^{*4}_r (\sum_{i=r+1}^p W_{i, i-r} )^4  \Big\}  \nonumber  \\
& +& \ds\frac 14 \underset{1 \leq l \neq m <T }{\ds\sum }   \sigma^{*2}_l \sigma^{*2}_m \Big( \sum_{k=1}^nV_{p}(l,m,k) \Big)^2 (1 +o_P(1)) \label{LR1} \\
&-&  \frac 12  \sum_{i=1}^p W_{i,i} \sum_{j:j\ne i} (\sigma^{*}_{|i-j|})^2 -  \underset{1 \leq l \leq T-1}{\ds\sum} \sigma_l^{*2}\sum_{h=1+l}^{p-l}  W_{ h+l,h-l }   + \frac n4 \underset{1 \leq i\ne j \leq p}{\ds\sum } (\sigma^{*}_{|i-j|})^2 + o_P(1).   \nonumber
\end{eqnarray}
Let us treat each term of (\ref{LR1}) separately.
 We first decompose $(\sum_{i=r+1}^p W_{i, i-r} )^2$  as follows,
\begin{eqnarray}
A &:=& (\ds\sum_{i=r+1}^p W_{i, i-r} )^2 \nonumber \\
 &= & \underset{1 +r \leq i_1 ,i_2 \leq p}{\ds\sum  }  \Big(  \sum_{k=1}^n \sum_{\substack {l=1 \\ l \neq k}}^n  X_{k, i_1}  X_{k, i_1 -r }  X_{l, i_2}  X_{l, i_2 -r } +
 \ds\sum_{k=1}^n  X_{k, i_1}  X_{k, i_1 -r }  X_{k, i_2}  X_{k, i_2 -r } \Big) \nonumber \\
& =& \underset{1 +r \leq i_1 ,i_2 \leq p}{\ds\sum  } ~  \sum_{k=1}^n \sum_{\substack {l=1 \\ l \neq k}}^n  X_{k, i_1}  X_{k, i_1 -r }  X_{l, i_2}  X_{l, i_2 -r } \nonumber \\
& + &  \underset{1 +r \leq i_1  \neq i_2 \leq p}{\ds\sum  } ~ \ds\sum_{k=1}^n  X_{k, i_1}  X_{k, i_1 -r }  X_{k, i_2}  X_{k, i_2 -r } + \underset{1 +r \leq i \leq p}{\ds\sum} ~ \sum_{k=1}^n X_{k,i}^2 X_{k,i-r}^2 \nonumber  \\[0.3cm]
& := & A_1 + A_2 + A_3 \nonumber
\end{eqnarray}
The term $A_3$ will be taken into account as it is later on.

The dominant  term giving the asymptotic distribution is :
\begin{eqnarray}
\ds\frac 12 \underset{ 1 \leq r < T}{\ds\sum} \sigma_r^{*2} \cdot A_1  & = & \ds\frac 12 \underset{ 1 \leq r < T}{\ds\sum} \sigma_r^{*2} \underset{1 +r \leq i_1 ,i_2 \leq p}{\ds\sum  } ~  \sum_{k=1}^n \sum_{\substack {l=1 \\ l \neq k}}^n  X_{k, i_1}  X_{k, i_1 -r }  X_{l, i_2}  X_{l, i_2 -r } \nonumber \\
& =&  \ds\frac 12 \underset{ 1 \leq r < T}{\ds\sum} \sigma_r^{*2} \underset{1 + T \leq i_1 ,i_2 \leq p}{\ds\sum  } ~  \sum_{k=1}^n \sum_{\substack {l=1 \\ l \neq k}}^n  X_{k, i_1}  X_{k, i_1 -r }  X_{l, i_2}  X_{l, i_2 -r } \nonumber \\
& + &  \underset{ 1 \leq r < T}{\ds\sum} \sigma_r^{*2} \underset{1 + r \leq i_1 \leq T}{\ds\sum } ~  \underset{1 + T \leq i_2 \leq p}{\ds\sum }  ~ \sum_{k=1}^n \sum_{\substack {l=1 \\ l \neq k}}^n  X_{k, i_1}  X_{k, i_1 -r }  X_{l, i_2}  X_{l, i_2 -r } \nonumber \\
& + &  \ds\frac 12 \underset{ 1 \leq r < T}{\ds\sum} \sigma_r^{*2} \underset{1 + r \leq i_1 ,i_2 \leq T}{\ds\sum  } ~  \sum_{k=1}^n \sum_{\substack {l=1 \\ l \neq k}}^n  X_{k, i_1}  X_{k, i_1 -r }  X_{l, i_2}  X_{l, i_2 -r } \nonumber \\
& := & A_{1,1} + A_{1,2} + A_{1,3}\, , \quad \text{say}. \nonumber
\end{eqnarray}
Recall that   $ \sigma_r^{*2} = 2 w_r^{*} b( \psi) $ and then $A_{1,1} = n(p-T) \widehat{\mathcal{A}}_n \cdot n(p-T) b(\psi)$. By Proposition \ref{prop:esttoeplitz}, $n(p-T)\widehat{\mathcal{A}}_n \overset{\mathcal{L}}{\to} \mathcal{N}(0,1)$ and thus $A_{1,1}$ can be written  $u_n Z_n$ with $Z_n \overset{\mathcal{L}}{\to} \mathcal{N}(0,1)$.

Next, under $\mathbb{P}_I$ all variables in the multiple sums of $A_{1,2}$ are uncorrelated (as well as for $A_{1,3}$). Thus,
\begin{eqnarray}
\Var_I(A_{1,2}) &= & 2    \underset{ 1 \leq r< T}{\ds\sum} \sigma_r^{*4} \underset{1 + r \leq i_1 \leq T}{\ds\sum } ~  \underset{1 + T \leq i_2 \leq p}{\ds\sum }  ~ \sum_{k=1}^n \sum_{\substack {l=1 \\ l \neq k}}^n  \Var_I ( X_{k, i_1}  X_{k, i_1 -r }  X_{l, i_2}  X_{l, i_2 -r }) \nonumber \\
&= & 2    \underset{ 1 \leq r < T}{\ds\sum} \sigma_r^{*4} (p-T)(T-r)n(n-1)  \leq n^2 pT   \underset{1 \leq r < T}{\ds\sum} \sigma_r^{*4} = 2 n^2 p T b^2(\psi) \nonumber \\
& =&  2 \cdot \frac{T}{p } \cdot u_n^2 = o(u_n^2), \quad \text{ as } T/p \to 0. \nonumber
\end{eqnarray}
And, similarly,
\begin{eqnarray}
\Var_I (A_{1,3} ) &=& \ds\frac 14 \underset{ 1 \leq r < T}{\ds\sum} \sigma_r^{*4} \underset{1 + r \leq i_1 ,i_2 \leq T}{\ds\sum  } ~  2 \sum_{k=1}^n \sum_{\substack {l=1 \\ l \neq k}}^n   \Var_I(X_{k, i_1}  X_{k, i_1 -r }  X_{l, i_2}  X_{l, i_2 -r }) \nonumber \\
& \leq &\ds\frac 12 \cdot   T^2 n(n-1) b^2(\psi) = O \Big( \Big( \frac{T}{p} \Big)^2 \cdot u_n^2 \Big) = o(u_n^2).  \nonumber
\end{eqnarray}
Therefore, $A_{1,1} + A_{1,2} + A_{1,3} = u_n Z_n + o_p(u_n),$ where $Z_n \overset{\mathcal{L}}{\to} \mathcal{N}(0,1)$.
For the same reason, we have,
\begin{eqnarray*}
\Var_I( \ds\frac 12 \underset{ 1 \leq r < T}{\ds\sum}  \sigma_r^{*2} \cdot A_2 ) &=& \ds\frac{1}{4}  \underset{1 \leq r < T}{\ds\sum}  \sigma^{*4}_r \underset{1 +r \leq i_1  \neq i_2 \leq p}{\ds\sum  } \ds\sum_{k=1}^n  \Var_I( X_{k, i_1}  X_{k, i_1 -r }  X_{k, i_2}  X_{k, i_2 -r }) \nonumber \\
& \leq  & \frac 14 \cdot  n p^2 b^2(\psi) =  O \Big( \frac{1}{n} \cdot u_n^2 \Big) = o(1),
\end{eqnarray*}
as soon as $n\to \infty$ or $u_n \to 0$. We want to show that
\[
B = \frac 1{12}  \underset{1 \leq r < T}{\ds\sum} \sigma^{*4}_r (\sum_{i=r+1}^p W_{i, i-r} )^4 = \frac{u_n^2}{2} + o_p(1) .
\]
Indeed,
\begin{eqnarray}
 \mathbb{E}_I(B) & =& \frac 1{12}  \underset{1 \leq r < T}{\ds\sum} \sigma^{*4}_r \cdot  \mathbb{E}_I(\sum_{i=r+1}^p W_{i, i-r} )^4 =  \frac 1{12}  \underset{1 \leq r < T}{\ds\sum} \sigma^{*4}_r \cdot  \mathbb{E}_I(\sum_{i=r+1}^p \sum_{k=1}^n X_{k,i} X_{k,i-r})^4 \nonumber \\
 &=&  \frac 1{12}  \underset{1 \leq r < T}{\ds\sum} \sigma^{*4}_r \sum_{k=1}^n  \Big( \sum_{i=r+1}^p  \mathbb{E}_I( X_{k,i}^4 X_{k,i-r}^4) + 3 \!\!\!\! \underset{ 1 +r \leq i_1 \neq i_2 \leq p}{\ds\sum } \mathbb{E}_I( X_{k,i_1}^2 X_{k,i_1 - r}^2)  \mathbb{E}_I( X_{k ,i_2}^2 X_{k,i_2-r}^2)  \Big)\nonumber \\
 & + & \frac{3}{12}   \underset{1 \leq r < T}{\ds\sum} \sigma^{*4}_r  \underset{1 \leq k_1 \neq  k_2 \leq n}{\ds\sum } ~ \underset{ 1 +r \leq i_1 \neq i_2 \leq p}{\ds\sum } \mathbb{E}_I( X_{k_1,i_1}^2 X_{k_1,i_1 - r}^2)  \mathbb{E}_I( X_{k_2 ,i_2}^2 X_{k_2,i_2-r}^2) \nonumber \\
 &= &   \frac 3{4}  \underset{1 \leq r < T}{\ds\sum} \sigma^{*4}_r   \cdot n(p-r) +  \frac 14   \underset{1 \leq r < T}{\ds\sum} \sigma^{*4}_r   \cdot n(p-r)^2 +  \frac{1}{4}   \underset{1 \leq r < T}{\ds\sum} \sigma^{*4}_r \cdot n^2 (p-r)^2  \label{expectedvalueofB}
\end{eqnarray}
\label{E(B)}
Recall that $2b^2(\psi) = \sum_{j } \sigma_j^{*4}$, thus
\begin{equation*}
\mathbb{E}_I(B) =  \frac 32 \cdot  npb^2(\psi)(1 +o(1)) + \frac12 \cdot np^2b^2(\psi)(1 +o(1))  +\frac 12 \cdot  n^2 p^2 b^2(\psi) (1+o(1)) = \frac{u_n^2}{2} (1 +o(1)).
\end{equation*}
Moreover,
\begin{eqnarray}
\Var_I(B) &= & \ds\frac{1}{12^2}   \underset{1 \leq r < T}{\ds\sum} \sigma^{*8}_r \cdot \Var_I( (\sum_{i=r+1}^p W_{i, i-r} )^4)  \nonumber \\
& + &   \ds\frac{1}{12^2}   \underset{ 1 \leq r \neq r' < T}{\ds\sum} \sigma^{*4}_r \sigma^{*4}_{r'} \Cov_I (  (\sum_{i=r+1}^p W_{i, i-r} )^4, (\sum_{i'=r'+1}^p W_{i', i'-r'} )^4 ) . \nonumber
\end{eqnarray}
 As in the calculation of the expected value of $B$, we can see that the term of higher order is obtained when we gather the indices into distinct pairs. Thus following the same reasoning we get
 \[
\Var_I(\sum_{i=r+1}^p \sum_{k=1}^n X_{k,i} X_{k,i-r})^4 = O( n^4 p^4).
\]
Through a very technical calculation, and using similar arguments as previously, we can prove that, for $r\ne r'$,
\[
 \Cov_I((\sum_{i=r+1}^p \sum_{k=1}^n X_{k,i} X_{k,i-r})^4, (\sum_{i'=r'+1}^p \sum_{k'=1}^n X_{k',i'} X_{k',i'-r'})^4  ) = O(n^3p^4).
  \]
Thus,
\[
\Var_I(B)=  O( \lambda^4 T n^4 p^4) + O(b^4(\psi)n^3p^4) = O(\psi^{\frac{3}{\alpha}}n^4p^4b^4(\psi)) + O(\frac{1}{n} \cdot n^4p^4b^4(\psi)) = o(1).
\]
By Chebyshev's inequality we deduce that
\[
\begin{array}{lcl}
  \ds\frac{1}{12} \sum_{r=1}^{ T-1}     \sigma^{*4}_r (\sum_{i=r+1}^p W_{i, i-r} )^4 &=& \mathbb{E}_{I} \Big(  \ds \frac{1}{12} \sum_{r=1}^{ T-1}     \sigma^{*4}_r (\sum_{i=r+1}^p W_{i, i-r} )^4 \Big) + o_P(1) \\
 &  = & \ds\frac{3(1 +o(1))}{12} \cdot 2 n^2 (p-r)^2 b^2(\psi) + o_P(1) = \frac{u_n^2}{2}(1 +o_P(1)).
  \end{array}
\]
Also using that $ \mathbb{E}_I(\sum_{i=r+1}^p W_{i, i-r} )^4 =O(n^2p^2) $, we get
  \begin{eqnarray}
C &:=& \underset{1 \leq l \neq m <T }{\ds\sum } \ds\frac{\sigma^{*2}_l \sigma^{*2}_m}4  \Big(\sum_{h=1}^{(p-l) \wedge (p-m)} \hspace{-0.5cm} W_{h+l, h+m} +  \ds\sum_{h=(1+l) \vee (1+m)}^p \hspace{-0.5cm}  W_{h-l,h-m} +  2 \!\!\!\sum_{h=1+m}^{p-l} \hspace{-0.2cm} W_{h+l, h-m}   \Big)^2 \nonumber \\[0.4cm]
 &=&   O_P( \lambda^2 T^2 np  )  =o_P( \psi^{(2 - \frac 1{2 \alpha})} \cdot u_n) =o_p(1) \nonumber \quad \text{ for } \alpha > 1/4 \text{ and  since }   \psi \to 0.
  \end{eqnarray}
  Moreover,
\[
F:=- \underset{1 \leq l \leq T-1}{\ds\sum} \sigma_l^{*2}\sum_{h=1+l}^{p-l}  W_{ h+l,h-l } = O_P(\ds\sqrt{np} b(\psi)) = o_P(u_n)= o_P(1).
\]
Finally, we group  the remaining terms of (\ref{LR1}) as follows,
\begin{eqnarray*}
G &:=& \frac 12  \sum_{r=1}^{ T-1} \sigma^{*2}_r A_3 - \frac 12  \sum_{i=1}^p W_{i,i} \sum_{j:j\ne i} (\sigma^{*}_{|i-j|})^2    + \frac n4 \underset{1 \leq i\ne j \leq p}{\ds\sum }(\sigma^{*}_{|i-j|})^2  \nonumber \\
&=& \frac{1}{4 } \underset{1 \leq i\ne j \leq p}{\ds\sum } (\sigma_{|i-j|}^*)^2
 \sum_{k=1}^n X_{k,i}^2 X_{k,j}^2  - \frac 12  \underset{1 \leq i\ne j \leq p}{\ds\sum } (\sigma_{|i-j|}^*)^2 \sum_{k=1}^n X_ {k-i}^2 +   \frac n4 \underset{1 \leq i\ne j \leq p}{\ds\sum } (\sigma^{*}_{|i-j|})^2  \nonumber \\
 &=& \frac{1}{4} \underset{1 \leq i\ne j \leq p}{\ds\sum } (\sigma_{|i-j|}^*)^2 \sum_{k=1}^n ( X_{k,i}^2 -1) (X_{k,j}^2 -1) =   O_P( \ds\sqrt{np }\cdot b(\psi) ) =o_P(u_n)=o_P(1).
\end{eqnarray*}
Let us note that throughout the previous proof we also showed that the likelihood ratio $L_{n,p}$ has a variance which tends to 0, for all $n \geq 2$, when  $u_n \to 0$.
\hfill \end{proof}


\begin{proof}[ Proof of Proposition \ref{prop:esttoeplitz} ]
Under the null hypothesis, $\widehat{\mathcal{A}}_n$ is centered, and
\begin{eqnarray*}
\Var_I(\widehat{\mathcal{A}}_n) & =&
\frac 2{n(n-1)(p-T)^4} \Var_I \left( \sum_{j=1}^T w_j^* \underset{1+T \leq i_1, i_2 \leq p}{\ds\sum } X_{1, i_1}X_{1, i_1-j} X_{2, i_2}X_{2, i_2-j} \right)\\
&=& \frac 2{n(n-1)(p-T)^4}  \sum_{j=1}^T w_j^{*2} \underset{1+T \leq i_1, i_2 \leq p}{\ds\sum }  \mathbb{E}_I(X_{1, i_1}^2 X_{1, i_1-j}^2 X_{2, i_2}^2 X_{2, i_2-j}^2 )\\
&=& \frac 2{n(n-1)(p-T)^2}  \sum_{j=1}^T  w_j^{*2}
\end{eqnarray*}
Recall that $\sum_{j=1}^T w_j^{*2}  = 1/2$ to get the desired result.
Under the alternative, for all $ \Sigma \in G( \alpha , L ,\psi)$,
 we decompose  $ \widehat{\mathcal{A}}_n - \mathbb{E}_{\Sigma}(\widehat{\mathcal{A}}_n)$ into a sum of two uncorrelated terms.
\begin{eqnarray}
\label{decomposition}
 \widehat{\mathcal{A}}_n - \mathbb{E}_{\Sigma}(\widehat{\mathcal{A}}_n)
&= \ds\frac{1}{n(n-1)(p-T)^2} \underset{1 \leq k \neq l \leq n}{\ds\sum }\sum_{j=1}^T w_j^* \underset{1+T \leq i_1, i_2 \leq p}{\ds\sum } (X_{k, i_1}X_{k, i_1-j} - \sigma_j)(X_{l, i_2}X_{l, i_2-j} - \sigma_j) \nonumber \\
&  + \ds\frac{2}{n(p-T)} \sum_{k=1}^n \sum_{j=1}^T w^*_j \sum_{i_1 =T+1}^p  (X_{k, i_1}X_{k, i_1-j} - \sigma_j) \sigma_j.
\end{eqnarray}
Then the variance of $\widehat{\mathcal{A}}_n$ will be given  as a sum of two  terms,
\[
\Var_{\Sigma} (\widehat{\mathcal{A}}_n) = \ds\frac{R_1}{n(n-1)(p-T)^4} + \frac{R_2}{n(p-T)^2},
\]
where
\begin{eqnarray*}
R_1 &= & 2 \mathbb{E}_{\Sigma} \Big( \ds\sum_{j=1}^T w_j^* \underset{T+1 \leq i_1, i_2 \leq p}{\ds\sum }(X_{1, i_1}X_{1, i_1-j} - \sigma_j)(X_{2, i_2}X_{2, i_2-j} - \sigma_j) \Big)^2 ,\\
R_2 &=& 4   \mathbb{E}_{\Sigma}\Big( \ds\sum_{j=1}^T w^*_j \sum_{i_1 =T + 1}^p  (X_{1, i_1}X_{1, i_1-j} - \sigma_j) \sigma_j \Big)^2.
\end{eqnarray*}
Let us deal first with $R_1$:
\[
\begin{array}{lcl}
R_1 &= & 2 \underset{1 \leq j, j' <T}{\ds\sum } w_j^*  w_{j'}^* \underset{T +1 \leq i_1, i_3 \leq p}{\ds\sum } \mathbb{E}_{\Sigma}  [(X_{1, i_1}X_{1, i_1-j} - \sigma_j) (X_{1, i_3}X_{1, i_3-j'} - \sigma_j')]\\
 && \cdot \underset{T + 1 \leq i_2, i_4 \leq p}{\ds\sum }\mathbb{E}_{\Sigma} [(X_{2, i_2}X_{2, i_2-j} - \sigma_j)(X_{2, i_4}X_{2, i_4-j'} - \sigma_j')] \\
 &=& 2 \underset{1 \leq j, j' <T}{\ds\sum } w_j^*  w_{j'}^*  \Big( \underset{T +1 \leq i_1, i_3 \leq p}{\ds\sum }  (\sigma_{|i_1 -i_3|} \sigma_{|i_1 -i_3 -j+j'|} + \sigma_{|i_1 -i_3-j|}\sigma_{|i_1 -i_3+j'|} ) \Big)^2 \\
 &=& 2  \underset{1 \leq j, j' <T}{\ds\sum }  w_j^*  w_{j'}^* \Big( \ds\sum_{r=-p+T+1}^{p-(T+1)} (p-T-|r|) (\sigma_{|r|}\sigma_{|r-j+j'|} + \sigma_{|r-j|}\sigma_{|r+j'|}) \Big)^2
\end{array}
\]
Our aim here is to find an upper bound of $R_1$. In $R_1$ we  distinguish two cases: the first one  when  for $j=j'$ and the second one when $j \neq j'$. Let us begin with the case when $j=j'$:
\begin{eqnarray}
R_{1,1} &:=&2 \ds\sum_{j=1}^T  w_j^{*2}  \Big( \sum_{r=-p+T+1}^{p-(T+1)} (p-T -|r|) (\sigma_{|r|}^2+ \sigma_{|r-j|}\sigma_{|r+j|}) \Big)^2 \nonumber \\
 &=& 2 \ds\sum_{j=1}^T  w_j^{*2}   \Big( (p-T)(\sigma_0^2 + \sigma_j^2) + 2\sum_{r=1}^{p-(T+1)} (p-T-r) (\sigma_{r}^2+ \sigma_{|r-j|}\sigma_{|r+j|})   \Big)^2 \nonumber \\
 &= & 2 \ds\sum_{j=1}^T  w_j^{*2} \Big[ (p-T)^2  (\sigma_0^2 + \sigma_j^2)^2 + 4  \Big( \sum_{r=1}^{p-(T+1)} (p-T-r) (\sigma_{r}^2+ \sigma_{|r-j|}\sigma_{|r+j|}) \Big)^2 \nonumber \\
  && ~~~+  4(p-T)(\sigma_0^2 + \sigma_j^2)  \ds\sum_{r=1}^{p-(T+1)} (p-T-r) (\sigma_{r}^2+ \sigma_{|r-j|}\sigma_{|r+j|}) \Big] \nonumber.
 \end{eqnarray}
Let us bound from above each term on the right-hand side of the previous equality:
\begin{eqnarray}
R_{1,1,1} &:=& 2 \ds\sum_{j=1}^T  w_j^{*2}  (p-T)^2(\sigma_0^2 + \sigma_j^2)^2 = 2 (p-T)^2 \Big( \ds\sum_{j=1}^T  w_j^{*2} + 2  \ds\sum_{j=1}^T  w_j^{*2}\sigma_j^2 + \ds\sum_{j=1}^T  w_j^{*2}\sigma_j^4 \Big) \nonumber  \\[0.5cm]
&\leq & 2 (p-T)^2 \Big( \ds\frac{1}{2} + 3 L \cdot (\sup\limits_{j} w_j^*)^2  \Big) = (p-T)^2( 1 + o(1) ). \label{R_{1,1,1}}
\end{eqnarray}
Now we give an upper bound for the second term of (\ref{R_{1,1}}). Using Cauchy-Schwarz inequality we get,
\begin{eqnarray*}
R_{1,1,2} &:=& 8 \ds\sum_{j=1}^T  w_j^{*2}   \Big[ \sum_{r=1}^{p-(T+1)} (p-T-r) (\sigma_{r}^2+ \sigma_{|r-j|}\sigma_{|r+j|}) \Big]^2  \nonumber \\
&\leq&  8 (p-T)^2 \ds\sum_{j=1}^T  w_j^{*2}  \Big[ \sum_{r=1}^{p-(T+1)} \sigma_{r}^2+ (\sum_{r=1}^{p-(T+1)} \sigma_{|r-j|}^2)^{1/2} (\sum_{r=1}^{p-(T+1)} \sigma_{|r+j|}^2)^{1/2} \Big]^2 \nonumber \\
& \leq & 16 (p-T)^2 \ds\sum_{j=1}^T  w_j^{*2} \Big[ \Big( \sum_{r=1}^{p-(T+1)} \sigma_{r}^2 \Big)^2 +
 (\sum_{r=1}^{p-(T+1)} \sigma_{|r-j|}^2) (\sum_{r=1}^{p-(T+1)} \sigma_{|r+j|}^2) \Big]. \nonumber \\
\end{eqnarray*}
Again we will treat each term of the previous inequality apart. Let us see first, that if  $(r \leq j  \Longrightarrow  w^*_j \leq w^*_r)$. In addition  to the previous remark we use the class property to get:
\begin{eqnarray}
R_{1,1,2,1} &:=& \ds\sum_{j=1}^T w_j^{*2}  \Big( \sum_{r=1}^{p-(T+1)} \sigma_{r}^2 \Big)^2 \leq \ds\sum_{j=1}^T  w_j^{*2} \Big( \sum_{r=1}^{j}\sigma_{r}^2 + \sum_{r= j +1}^{p-(T+1)} \ds\frac{r^{2 \alpha}}{j^{2 \alpha}} \sigma_{r}^2  \Big)^2  \nonumber\\
&\leq & 2  \ds\sum_{j=1}^T \Big(  \sum_{r=1}^{j} w_r^* \sigma_{r}^2 \Big)^2  + 2(\sup\limits_j w_j^*)^2 \ds\sum_{j=1}^T \frac{1}{j^{4 \alpha}} \Big( \sum_{r= j +1}^{p-(T+1)} r^{2 \alpha} \sigma_{r}^2  \Big)^2  \nonumber \\
& \leq & 2 \cdot T \cdot \mathbb{E}_\Sigma^2(\widehat{\mathcal{A}}_n) +  (\sup\limits_j w_j^*)^2 \cdot k_0(\alpha, L) \label{R_{1,1,2,1}}.
\end{eqnarray}
Indeed, for $\alpha >1/4$, we have, $\sum_{j=1}^T j^{-4\alpha} \leq (4 \alpha - 1)^{-1}$ and we can take $k_0(\alpha, L) = 2 L^2 (4 \alpha - 1)^{-1}$. Using similar arguments we prove that,
\begin{eqnarray}
R_{1,1,2,2} &:= & 
\ds\sum_{j=1}^T w_j^{*2}  (\sum_{\substack{r=1 \\ |r-j| < j }}^{p-(T+1)} \sigma_{|r-j|}^2 + \sum_{\substack{r=1 \\ |r-j| \geq j }}^{p-(T+1)} \sigma_{|r-j|}^2 )(\sum_{r=1}^{p-(T+1)}\sigma_{r+j}^2) \nonumber \\
 &\leq &  \ds\sum_{j=1}^T w_j^*  (\sum_{\substack{r=1 \\ |r-j| < j }}^{p-(T+1)} w_{|r-j|}^* \sigma_{|r-j|}^2 ) (\sum_{r=1}^{p-(T+1)}\sigma_{r+j}^2)  \nonumber \\
   & +&  \ds\sum_{j=1}^T w_j^{*2} (\sum_{\substack{r=1 \\ |r-j| \geq j }}^{p-1}\ds\frac{|r-j|^{2 \alpha}}{j^{2 \alpha}} \sigma_{|r-j|}^2 )(\sum_{r=1}^{p-(T+1)}\ds\frac{(r+j)^{2 \alpha}}{j^{2 \alpha}}\sigma_{r+j}^2) \nonumber \\
& \leq &  (\sup\limits_j w_j^*) \cdot T \cdot  \mathbb{E}_\Sigma(\widehat{\mathcal{A}}_n)  \cdot L +  (\sup\limits_j w_j^*)^2 \cdot k_0(\alpha ,L).  \label{R_{1,1,2,2}}
\end{eqnarray}
The third term in $R_{1,1}$ is treated by similar arguments:
\begin{eqnarray}
R_{1,1,3} & =& (p-T) \ds\sum_{j=1}^T w_j^{*2} (\sigma_0^2 + \sigma_j^2) \sum_{r=1}^{p-(T+1)} (p-T-r)( \sigma_r^2 + \sigma_{|r-j|} \sigma_{|r+j|}) \nonumber \\
& \leq & (p-T)^2 \cdot \sup\limits_{j}(\sigma_0^2 + \sigma_j^2) \cdot \Big\{ \sum_{j=1}^T w_j^*  \sum_{r=1}^j w_r^* \sigma_r^2  + (\sup\limits_j w_j^{*2}) \sum_{j=1}^T \ds\frac{1}{j^{2 \alpha}} \sum_{r =j+1}^{p-(T+1)} r^{2\alpha}\sigma_r^2  \nonumber \\ 
&+& (\sup\limits_j w_j^{*2}) \ds\sum_{j=1}^T (\sum_{r=1}^{p-(T+1)} \sigma_{|r-j|}^2 )^{1/2} (\sum_{r=1}^{p-(T+1)}\ds\frac{(r+j)^{2 \alpha}}{j^{2 \alpha}}\sigma_{r+j}^2)^{1/2} \Big\} \nonumber \\
& \leq & 2(p-T)^2 \Big\{ O(\ds\sqrt{T}) \cdot \mathbb{E}_\Sigma(\widehat{\mathcal{A}}_n)  + (\sup\limits_j w_j^{*2}) \cdot \Big( O( \max\{ 1, T^{-2 \alpha +1} \}) +  O( \max\{ 1, T^{- \alpha +1} \} \Big) \Big\} \nonumber \\
& \leq &  2(p-T)^2 \cdot \Big\{ O(\ds\sqrt{T}) \cdot \mathbb{E}_\Sigma(\widehat{\mathcal{A}}_n) +o(1) \Big\} . \label{R_{1,1,3}}
\end{eqnarray}
Put together bounds in (\ref{R_{1,1,1}}) to (\ref{R_{1,1,3}}), we can deduce that,
\begin{equation}\label{R_{1,1}}
R_{1,1} \leq (p-T)^2(1+ o(1)) + (p-T)^2 \cdot  \mathbb{E}_\Sigma(\widehat{\mathcal{A}}_n) \cdot O(\ds\sqrt{T}) + (p-T)^2 \cdot \mathbb{E}^2_\Sigma(\widehat{\mathcal{A}}_n) \cdot O(T).
\end{equation}
Now, we will treat the case when, $j \neq j'$.
\[
\begin{array}{lcl}
R_{1,2} &:=& 2 \underset{1 \leq j \neq j' \leq T}{\ds\sum \ds\sum}  w_j^*  w_{j'}^* \Big( \ds\sum_{r=-p+T+1}^{p-(T+1)} (p-|r|) (\sigma_{|r|}\sigma_{|r-j+j'|} + \sigma_{|r-j|}\sigma_{|r+j'|}) \Big)^2  \\
& \leq &   4 (p-T)^2 \underset{1 \leq j \neq j' \leq T}{\ds\sum \ds\sum}  w_j^*  w_{j'}^* \Big[ \Big( \ds\sum_{r=-p+T+1}^{p-(T+1)}  \!\!\!\! |\sigma_{|r|}\sigma_{|r-j+j'|}| \Big)^2
 +  \Big( \ds\sum_{r=-p+T+1}^{p-(T+1)} \!\!\!\!  |\sigma_{|r-j|}\sigma_{|r+j'|}| \Big)^2  \Big].
\end{array}
\]
These last two terms are treated similarly, so let us deal with the first one.
By  using the same arguments as previously, we have 
\begin{eqnarray}
 R_{1,2,2} 
 & :=&  \underset{1 \leq j \neq j' \leq T}{\ds\sum } w_j^*  w_{j'}^* \Big( |\sigma_{|j'-j| }| + \ds\sum_{ \substack{r=-p+T+1 \\ r \neq 0}}^{p-(T+1)}  |\sigma_{|r|}\sigma_{|r-j+j'|}| \Big)^2 \nonumber \\
& \leq &  2 \underset{1 \leq j \neq j' \leq T}{\ds\sum } w_j^*  w_{j'}^*
 \sigma_{|j'-j| }^2 + 4\underset{1 \leq j \neq j' \leq T}{\ds\sum } w_j^*  w_{j'}^*  (\ds\sum_{r=1}^{p-(T+1)}  \sigma_{r}^2) (\ds\sum_{ \substack{r=-p+T+1 \\ r \neq 0}}^{p-(T+1)} \sigma_{|r-j+j'|}^2 ) .
 \end{eqnarray}
 We decompose the sum over $j\ne j'$ over sets where $\{|j'-j| \leq j \}$ and $\{|j'-j| > j\}$ and use $1 \leq |j'-j|^{2\alpha}/j^{2 \alpha}$ over the later, then similarly for sums over $r$:
 \begin{eqnarray}
  R_{1,2,2} & \leq & 2 \underset{\underset{|j'-j| < j}{1 \leq j \neq j' \leq T}}{\ds\sum } w_{j'}^*  w_{|j'-j|}^*  \sigma_{|j'-j| }^2  + 2 \underset{\underset{|j'-j| > j}{1 \leq j \neq j' \leq T}}{\ds\sum } w_j^*  w_{j'}^* \frac{|j'-j|^{2 \alpha}}{j^{2 \alpha}} \sigma_{|j'-j| }^2 +4 \underset{1 \leq j \neq j' \leq T}{\ds\sum }   \Big( \ds\sum_{r=1}^{j}   w_r^* \sigma_{r}^2 \nonumber \\
 & + &    w_j^*  \sum_{r=j +1 }^{p-(T+1)}  \ds\frac{r^{2 \alpha}}{j^{2 \alpha}}\sigma_{r}^2 \Big)   \Big(\ds\sum_{\substack{r=-p+T+1 \\ |r-j+j'| < j' }}^{p-(T+1)} \hspace{-0.2cm}  w_{|r-j+j'|}^* \sigma_{|r-j+j'|}^2 + w_{j'}^* \hspace{-0.2cm} \ds\sum_{\substack{r=-p+T+1 \\ |r-j+j'| \geq j'}}^{p-(T+1)} \ds\frac{|r-j+j'|^{2\alpha}}{(j')^{ 2 \alpha}} \sigma_{|r-j+j'|}^2 \Big) \nonumber \\
  &\leq &  4 \cdot (\sup\limits_{j} w^*_j) \cdot T \cdot  \mathbb{E}_\Sigma(\widehat{\mathcal{A}}_n) + 4L \cdot (\sup\limits_{j} w^*_j)^2 \cdot O(\max \{1, T^{-2 \alpha +1} \})  + O(T^2) \cdot  \mathbb{E}_{\Sigma}^2(\widehat{\mathcal{A}}_n) \nonumber   \\
  &+ & \!\!\! 16L \cdot (\sup\limits_j w_j^*)  \cdot  T   \cdot O(\max \{1, T^{-2 \alpha +1} \})  \cdot  \mathbb{E}_\Sigma(\widehat{\mathcal{A}}_n)  +   16L^2 \cdot (\sup\limits_j w_j^*)^2 \cdot O(\max \{1, T^{-4 \alpha +2 } \} ) . \nonumber
\end{eqnarray}
As consequence, for all $\alpha > 1/4$,
\begin{equation}
 \label{R_{1,2}}
R_{1,2} \leq (p-T)^2 \{\mathbb{E}_{\Sigma}(\widehat{\mathcal{A}}_n) \cdot  O(\ds\sqrt{T}) + \mathbb{E}_{\Sigma}^2(\widehat{\mathcal{A}}_n) \cdot O(T^2) + \mathbb{E}_{\Sigma}(\widehat{\mathcal{A}}_n) \cdot O(T^{3/2 - 2 \alpha}) +o(1)   \} .
\end{equation}
Finally put together (\ref{R_{1,1}}) and (\ref{R_{1,2}}) to get (\ref{R_1}).
In order to find an upper bound for the variance of $\widehat{\mathcal{A}}_n$ we still have to bound from above $R_2$.
\[
\begin{array}{lcl}
R_2 &=& 4  \underset{1 \leq \, j,j' <T}{\ds\sum } w^*_j w^*_{j'}\sigma_j \sigma_{j'} \underset{T+1 \leq i_1, i_2 \leq p}{\ds\sum }
 \mathbb{E}_{\Sigma} [(X_{i, i_1}X_{i, i_1-j} - \sigma_j)(X_{i, i_2}X_{i, i_2-j'} - \sigma_{j'})] \\
&=& 4   \underset{1 \leq \, j,j' <T}{\ds\sum }  w^*_j w^*_{j'}\sigma_j \sigma_{j'} \underset{T+1 \leq i_1, i_2 \leq p}{\ds\sum }
(\sigma_{|i_1 -i_2|} \sigma_{|i_1 -i_2 -j+j'|} + \sigma_{|i_1 -i_2-j|}\sigma_{|i_1 -i_2+j'|} ) \\
& =& 4  \underset{1 \leq \, j,j' <T}{\ds\sum } w^*_j w^*_{j'}\sigma_j \sigma_{j'}   \ds\sum_{r=-p+T+1}^{p-(T+1)} (p-T-|r|) (\sigma_{|r|}\sigma_{|r-j+j'|} + \sigma_{|r-j|}\sigma_{|r+j'|}).
\end{array}
\]
Let us begin by the first case when $j=j'$. It is easily seen that,
\begin{eqnarray}
R_{2,1} &:=& 4\ds\sum_{j=1}^T  w^{*2}_j \sigma_j^2  \sum_{r=-p+T+1}^{p-(T+1)} (p-T-|r|)(\sigma_{|r|}^2+ \sigma_{|r-j|}\sigma_{|r+j|})\nonumber \\
&\leq &8L \cdot p \cdot (\sup\limits_j w_j^* ) \cdot \mathbb{E}_{\Sigma}(\widehat{\mathcal{A}}_n)  \label{R_{2,1}}
\end{eqnarray}
While, when $j \neq j'$, we can prove that,
\begin{eqnarray}
&R_{2,2}& := 4 \underset{ 1 \leq j \neq  j' \leq T}{\ds\sum  \sum } w^*_j w^*_{j'}\sigma_j \sigma_{j'}  \ds\sum_{r=-p+T+1}^{p-(T+1)} (p-T-|r|) (\sigma_{|r|}\sigma_{|r-j+j'|} + \sigma_{|r-j|}\sigma_{|r+j'|})  \nonumber \\
& \leq & \hspace{-0.4cm}  4 \Big( \underset{ 1 \leq j \neq  j' \leq T}{\ds\sum  \sum } w^*_j w^*_{j'}\sigma_j^2 \sigma_{j'}^2 \Big)^{\frac 12}
\Big(\underset{ 1 \leq j \neq  j' \leq T}{\ds\sum  \sum } w^*_j w^*_{j'}
\Big( \ds\sum_{r=-p+T+1}^{p-(T+1)} (p-T-|r|) (\sigma_{|r|}\sigma_{|r-j+j'|} + \sigma_{|r-j|}\sigma_{|r+j'|} )  \Big)^2  \Big)^{\frac{1}{2}}  \nonumber \\
&\leq  & \hspace{-0.4cm}
4 \, \mathbb{E}_{\Sigma}(\widehat{\mathcal{A}}_n) \cdot (R_{1,2})^{1/2}.
\nonumber
\end{eqnarray}
We use the bound obtained in (\ref{R_{1,2}}) to  deduce that:
\begin{equation}
\label{R_{2,2}}
R_{2,2} \leq (p-T) \left( \mathbb{E}^2_{\Sigma}(\widehat{\mathcal{A}}_n) \cdot O(T)   +   \mathbb{E}^{3/2}_{\Sigma}(\widehat{\mathcal{A}}_n) \cdot (O(T^{1/4}) + O(T^{3/4 - \alpha})) +  \mathbb{E}_{\Sigma} (\widehat{\mathcal{A}}_n) \cdot o(1) \right).
\end{equation}
Put together (\ref{R_{2,1}}) and (\ref{R_{2,2}}) to get (\ref{R_2}).
\end{proof}


\begin{proof}[Proof of Proposition \ref{prop:asymptoticnormality}]
Assume that $n (p-T) \cdot \mathbb{E}_{\Sigma}(\widehat{\mathcal{A}}_n)  \asymp 1 $, to prove the asymptotic normality of $n (p-T) \cdot \widehat{\mathcal{A}}_n $, we use the decomposition (\ref{decomposition}) of the test statistic. first let us show that,
\[
\widehat{\mathcal{A}}_{n,1} := 2 \sum_{k=1}^n \sum_{j=1}^T w^*_j \sum_{i_1 =T + 1}^p  (X_{k, i_1}X_{k, i_1-j} - \sigma_j) \sigma_j \stackrel{P}{\longrightarrow} 0
\]
By Markov inequality we have, $\forall \varepsilon >0$,
\begin{eqnarray}
\mathbb{P}_{\Sigma} \Big( |2 \sum_{k=1}^n \sum_{j=1}^T w^*_j \sum_{i_1 =T +1}^p  (X_{k, i_1}X_{k, i_1-j} - \sigma_j) \sigma_j | > \varepsilon) 
& \leq & \frac{n\cdot R_2}{\varepsilon^2} \nonumber
\end{eqnarray}
According to (\ref{R_2}), and under the assumption that $n p \cdot \mathbb{E}(\widehat{\mathcal{A}}_n) \asymp 1$, we can see that,
\begin{eqnarray}
n \cdot R_2 & \leq &  n \cdot (p-T)  \{\mathbb{E}_{\Sigma} (\widehat{\mathcal{A}}_n) \cdot o(1)  +  \mathbb{E}^{3/2}_{\Sigma}(\widehat{\mathcal{A}}_n ) \cdot( O(T^{1/4})+ O(T^{3/4 - \alpha} )) +\mathbb{E}^2_{\Sigma}(\widehat{\mathcal{A}}_n) \cdot O(T)  \} \nonumber \\
& \leq & o(1) + O \Big( \ds\frac{T^{1/4} + T^{3/4 - \alpha}}{\sqrt{n(p-T)}} \Big) + O \Big( \frac{T}{n(p-T)} \Big)  =o(1) \quad \text{ since } T/p \longrightarrow 0 \text{ and for all } \alpha > 1/4. \nonumber
\end{eqnarray}
Which involves by Slutsky theorem that for proving the asymptotic normality it is sufficient to show that,
\begin{equation}
\label{convinlaw}
\widehat{\mathcal{A}}_{n,2} := \ds\frac{1}{n(p-T)} \underset{1 \leq k \neq l \leq n}{\ds\sum }\sum_{j=1}^T w_j^* \underset{T +1 \leq i_1, i_2 \leq p}{\ds\sum } (X_{k, i_1}X_{k, i_1-j} - \sigma_j)(X_{l, i_2}X_{l, i_2-j} - \sigma_j) \stackrel{L}{\longrightarrow} N(0,1)
\end{equation}
In order to prove this previous convergence, we are led to apply theorem 1 of \cite{Hall84}. This result is an application of the more general theorem of asymptotic normality for martingale differences, see e.g. \cite{shiryaev96}. $ \widehat{\mathcal{A}}_{n,2}$ is a centered, 1-degenerate, U-Statistic of second order, with kernel $H_n(X_1, X_2)$ defined by,
\[
H_n(X_1, X_2) := \ds\frac{1}{n(p-T)} \sum_{j=1}^T w_j^* \underset{T+1 \leq i_1, i_2 \leq p}{\ds\sum } (X_{1, i_1}X_{1, i_1-j} - \sigma_j)(X_{2, i_2}X_{2, i_2-j} - \sigma_j)
\]
 Therefore we should check that $ \mathbb{E}_{\Sigma}(H^2_n(X_1,X_2)) < + \infty $ and
 \begin{equation}
 \label{conditionNA}
 \ds\frac{\mathbb{E}_{\Sigma}(G^2_n(X_1, X_2)) + n^{-1} \mathbb{E}_{\Sigma}(H_n^4(X_1, X_2))}{\mathbb{E}_{\Sigma}^2(H_n^2(X_1, X_2))} \longrightarrow 0
 \end{equation}
where $ G_n(x,y):= \mathbb{E} (H_n(X_1,x)H_n(X_1,y)) $, for $x,y \in \mathbb{R}^p$.
The proof of \eqref{conditionNA} is given separately hereafter.

The asymptotic normality under $\Sigma=I$ (the null hypothesis) is only simpler as $\sigma_j=0$ for all $j\geq 1$, for $n,\, p \to \infty$. However, under the null hypothesis we prove separately (hereafter) that
\begin{equation} \label{ANH0}
n(p-T) \widehat{\mathcal{A}}_{n} \rightarrow \mathcal{N}(0,1), \mbox{ for } p \to \infty \mbox{ and for any fixed } n \geq 2 . 
\end{equation}
\end{proof}


\begin{proof}[Proof of \eqref{conditionNA}]
To show \eqref{conditionNA}, we first calculate $G_n(x,y)$ and $\mathbb{E}_{\Sigma}(H_n^2(X_1, X_2))$. That is,
\begin{eqnarray}
G_n(x,y) &=& \ds\frac{1}{n^2 (p-T)^2} \underset{1 \leq j_1, j_2<T}{\ds\sum } w_{j_{1}}^* w_{j_{2}}^* \ds\sum_{r= -p+T +1}^{p-(T+1)} (p -T-|r|)(\sigma_{|r|}\sigma_{|r- j_1 + j_2|} + \sigma_{|r- j_1|}\sigma_{|r+ j_2|}) \nonumber \\
 && \underset{1 \leq i_1, i_2 \leq p}{\ds\sum } (x_{i_1} x_{i_1 - j_1}- \sigma_{j_1})(y_{ i_2}y_{ i_2-j_2} - \sigma_{j_2})
\end{eqnarray}
Note that, under the assumption $n p \cdot \mathbb{E}(\widehat{\mathcal{A}}_n) \asymp 1$, $\alpha > 1/4 $, $p \psi^{1/ \alpha} \to + \infty$ and using (\ref{R_1}), we have,
\[
\mathbb{E}_{\Sigma}(H_n^2(X_1, X_2)) = \ds\frac{1 + o(1)}{2 n^2 }
\]
Now, let us verify that, uniformly over $\Sigma $,
\begin{equation}
\label{conditionNA1}
\mathbb{E}_{\Sigma}(G^2_n(X_1, X_2)) / \mathbb{E}_{\Sigma}^2(H_n^2(X_1, X_2)) = o(1).
\end{equation}
We write
\begin{eqnarray}
 && \ds\frac{\mathbb{E}_{\Sigma}(G^2_n(X_1, X_2))}{\mathbb{E}_{\Sigma}^2(H_n^2(X_1, X_2))} = 4 n^4 \cdot  \mathbb{E}_{\Sigma}(G^2_n(X_1, X_2)) \nonumber \\
& = &\ds\frac{4}{(p-T)^4} \underset{1 \leq j_1, j_2,j_3,j_4 <T}{\ds\sum }  w_{j_{1}}^* w_{j_{2}}^* w_{j_{3}}^* w_{j_{4}}^* \underset{-p+T+1\leq r_1, r_2 \leq p-(T+1)}  {\ds\sum } (p -T-|r_1|) (p-T -|r_2|) \nonumber \\[0.2cm]
&&  \cdot  (\sigma_{|r_1|}\sigma_{|r_1- j_1 + j_2|} + \sigma_{|r_1- j_1|}\sigma_{|r_1+ j_2|}) (\sigma_{|r_2|}\sigma_{|r_2- j_3 + j_4|} + \sigma_{|r_2- j_3|}\sigma_{|r_2+ j_4|}) \nonumber \\[0.2cm]
 &&  \cdot \underset{T+1 \leq i_1 , i_3 \leq p }{\ds\sum }  \mathbb{E}_{\Sigma} [(X_{1_,i_1} X_{1, i_1 - j_1}- \sigma_{j_1}) (X_{1_,i_3} X_{1, i_3 - j_3}- \sigma_{j_3}) ]  \nonumber \\[0.2cm]
  && \cdot \underset{T+1 \leq i_2 , i_4 \leq p }{\ds\sum } \mathbb{E}_{\Sigma}[(X_{2, i_2} X_{2, i_2-j_2} - \sigma_{j_2}) (X_{2, i_4} X_{2, i_4-j_4} - \sigma_{j_4}) ]
 \end{eqnarray}
We calculate each expected value, and bound from above by the absolute value, we obtain:
 \begin{eqnarray}
&&4n^4 \cdot \mathbb{E}_{\Sigma}(G^2_n(X_1, X_2)) \nonumber \\
 & \leq &   4 \underset{1 \leq j_1, j_2,j_3,j_4 <T}{\ds\sum  }  w_{j_{1}}^* w_{j_{2}}^* w_{j_{3}}^* w_{j_{4}}^* \underset{-p+T+1\leq r_1, r_2,r_3,r_4 \leq p-(T+1)}  {\ds\sum }  \nonumber \\
 && \cdot  (|\sigma_{|r_1|}\sigma_{|r_1- j_1 + j_2|}| + |\sigma_{|r_1- j_1|}\sigma_{|r_1+ j_2|}|) (|\sigma_{|r_2|}\sigma_{|r_2- j_3 + j_4|}| + |\sigma_{|r_2- j_3|}\sigma_{|r_2+ j_4|}|) \nonumber \\
 &&  \cdot  (|\sigma_{|r_3|}\sigma_{|r_3- j_1 + j_3|}| + |\sigma_{|r_3- j_1|}\sigma_{|r_3+ j_3|}|) (|\sigma_{|r_4|}\sigma_{|r_4- j_2 + j_4|} | + |\sigma_{|r_4 - j_2|}\sigma_{|r_4 + j_4|}|)
 \label{firstratio}
\end{eqnarray}
In (\ref{firstratio}) there are sixteen terms, that are all treated the same way, then we deal with,
\[
\begin{array}{lcl}
\mathcal{G} &:=& 4 \underset{1 \leq j_1, j_2,j_3,j_4 <T}{\ds\sum }  w_{j_{1}}^* w_{j_{2}}^* w_{j_{3}}^* w_{j_{4}}^* \underset{-p+T+1\leq r_1, r_2,r_3,r_4 \leq p-(T+1)}  {\ds\sum }  \\
&& \cdot | \sigma_{|r_1|}\sigma_{|r_1- j_1 + j_2|} \sigma_{|r_2|}\sigma_{|r_2- j_3 + j_4|} \sigma_{|r_3|}\sigma_{|r_3- j_1 + j_3|} \sigma_{|r_4|}\sigma_{|r_4- j_2 + j_4|} |
  \end{array}
\]
To bound from above this previous quantity, we distinguish four cases, based on the indices $j_1, j_2, j_3$ and $j_4$.
Let us begin by the the first case, when $j_1=j_2=j_3=j_4$ :
\[
\mathcal{G}_1 :=    4 \ds\sum_{j_{1}=1}^T w_{j_1}^{*4} \underset{-p +T+1  \leq r_1, r_2, r_3, r_4 \leq p-(T+1)}{ \sum}
\sigma_{|r_1|}^2 \sigma_{|r_2}^2  \sigma_{|r_3|}^2 \sigma_{|r_4|}^2 \leq 4 \cdot (\sup\limits_{j} w_{j}^*)^4 \cdot T \cdot (2L)^4 = O(\ds\frac{1}{T}) = o(1)
\]
We consider the second case, where there are two different values of indices, either two groups of two, or one group of three and one separate index. For the first one, let us assume that ($j_1 =j_4$, $j_2 =j_3$ and $j_1 \neq j_2)$,
\begin{eqnarray}
\mathcal{G}_2 & := &  4 \underset{1 \leq j_1 \neq j_2 < T}{\ds\sum }   w_{j_1}^{*2} w_{j_2}^{*2}   \underset{-p +T+1  \leq r_1, r_2, r_3, r_4 \leq p-(T+1)}{ \sum} | \sigma_{|r_1|}\sigma_{|r_1- j_1 + j_2|} \sigma_{|r_2|}\sigma_{|r_2- j_2 + j_1|}| \label{G2} \\
&& \hspace{7cm} \cdot  |\sigma_{|r_3|}\sigma_{|r_3- j_1 + j_2|} \sigma_{|r_4|}\sigma_{|r_4- j_2 + j_1|} | \nonumber \\
&=& 4  \underset{1 \leq j_1 \neq j_2 < T}{\ds\sum }   w_{j_1}^{*2} w_{j_2}^{*2} \cdot \Big( 2\, | \sigma_0 \, \sigma_{|j_1 - j_2|} | + \sum_{\substack{r_1 = -p +T+1 \\ r_1 \ne 0, \, r_1 \ne j_1 -j_2}}^{p-(T+1)} | \sigma_{|r_1|}\sigma_{|r_1- j_1 + j_2|} | \Big)^2 \nonumber\\
&& \cdot \Big( 2\, |  \sigma_0 \, \sigma_{|j_1 - j_2|} | + \sum_{\substack{r_2 = -p +T+1 \\ r_2 \ne 0, \, r_2 \ne j_2 -j_1}}^{p-(T+1)} | \sigma_{|r_2|}\sigma_{|r_2- j_2 + j_1|} | \Big)^2 \nonumber 
\end{eqnarray}
We apply the Cauchy-Schwarz inequality with respect to  $r_1$  and $r_2$ separately to get :
\begin{eqnarray*}
\mathcal{G}_2 & \leq &  4 \cdot  ( 2 + 2L)^2 \underset{1 \leq j_1 \neq j_2 < T}{\ds\sum }    w_{j_1}^{*2} w_{j_2}^{*2} \cdot  \Big\{ 4 \,  \sigma_{|j_1- j_2|}^2 + 2( \sum_{\substack{r_1 = -p +T+1 \\ r_1 \ne 0, \,  r_1 \ne  j_1 -j_2}}^{p-(T+1)} \sigma_{|r_1|}^2) ( \sum_{\substack{r_1 = -p +T+1 \\ r_1 \ne 0, \,  r_1 \ne  j_1 -j_2}}^{p-(T+1)}\sigma_{|r_1- j_1 + j_2|}^2) \, \Big\} \nonumber \\ 
& \leq & 16\cdot  (2 +2L)^2 \cdot   \Big(  \underset{1 \leq j_1 \neq j_2 < T}{\ds\sum }    w_{j_1}^{*2} w_{j_2}^{*2}   \sigma_{|j_1- j_2|}^2 +   2L \cdot  \underset{1 \leq j_1 \neq j_2 < T}{\ds\sum }    w_{j_1}^{*2} w_{j_2}^{*2}  \sum_{r_1 \ne 0} \sigma_{|r_1|}^2 \, \Big) \nonumber  \\ 
& \leq & 16 \cdot  (2 +2L)^2 \cdot \Big\{ ( \sup\limits_{j_2} w_{j_2}^{*2})  \cdot  \ds\sum_{j_1 =1}^T    w_{j_1}^{*2}  \sum_{j_2 =1}^T   \sigma_{|j_1- j_2|}^2 \nonumber \\
&+& 2L \cdot \sum_{j_2 =1}^T w_{j_2}^{*2} \cdot  \Big( \sum_{j_1 =1}^T w_{j_1}^* \sum_{\substack{ r_1; r_1 \ne 0 \\ |r_1| \leq j_1 }} w_{|r_1|} \sigma_{|r_1|}^2  +   \sum_{j_1 =1}^T w_{j_1}^{*2}  \sum_{\substack{ r_1; r_1 \ne 0 \\ |r_1|> j_1 }}\ds\frac{|r_1|^{2 \alpha}}{j_1^{2 \alpha}} \sigma_{|r_1|}^2 \, \Big\} \nonumber \\
& \leq & O \Big( \ds\frac 1T \Big) + O( \ds\sqrt{T}) \cdot \mathbb{E}(\widehat{\mathcal{A}}_n) + O \Big( \ds\frac 1T \Big) \cdot \max\{1, T^{-2\alpha +1} \} = o(1) ~  \text{ since }  \mathbb{E}(\widehat{\mathcal{A}}_n) \asymp 1/np  \text{ and } T/p \to 0.
\end{eqnarray*}
 Similar argument to prove that for $j_1 = j_3 =j_4$ and $j_1 \neq j_2$, we have,
 \[
  4 \underset{1 \leq j_1 \neq j_2 < T}{\ds\sum}  w_{j_1}^{*3} w_{j_2}^{*} \underset{-p +T+1  \leq r_1, r_2, r_3, r_4 \leq p-(T+1)}{\sum} \sigma_{|r_1|}\sigma_{|r_1- j_1 + j_2|} \sigma_{|r_2|}^2 \sigma_{|r_3|}^2  \sigma_{|r_4|} \sigma_{|r_4 - j_2 + j_1|} =o(1)
 \]
 which finishes the second case.
 Now let us assume that  we have three different values, ($j_1 =j_4$ and $j_1 \neq j_2 \neq j_3$), we obtain,
\begin{eqnarray}
\mathcal{G}_3 &:=& 4 \underset{1 \leq j_1 \neq j_2 \neq j_3 <T}{\ds\sum} w_{j_1}^{*2} w_{j_2}^{*}  w_{j_3}^{*} \underset{-p +T+1  \leq r_1, r_2, r_3, r_4 \leq p-(T+1)}{\sum }
 | \sigma_{|r_1|}\sigma_{|r_1- j_1 + j_2|} \sigma_{|r_2|}\sigma_{|r_2- j_3 + j_1|} | \nonumber \\
  && \hspace{7cm} \cdot |\sigma_{|r_3|}\sigma_{|r_3- j_1 + j_3|} \sigma_{|r_4|}\sigma_{|r_4- j_2 + j_1|} | \nonumber \\
  &= & 4 \underset{1 \leq j_1 \neq j_2 \neq j_3 <T}{\ds\sum } w_{j_1}^{*2} w_{j_2}^{*}  w_{j_3}^{*} \Big( 2\,  |\sigma_{|j_2 - j_1|}| + \sum_{\substack{r_1 = -p +T+1 \\ r_1 \ne 0, \, r_1 \ne   j_2 -j_1}}^{p-(T+1)} |\sigma_{|r_1|}\sigma_{|r_1- j_1 + j_2|}| \Big) \nonumber \\
  && \Big( 2\,   | \sigma_{|j_1 - j_3|}| + \sum_{\substack{r_2 = -p +T+1 \\ r_2  \ne 0, \,  r_2 \ne  j_1 -j_3}}^{p-(T+1)} |\sigma_{|r_2|}\sigma_{|r_2- j_1 + j_3|} | \Big)  \cdot  \Big( 2\,   |\sigma_{|j_3 - j_1|}| + \sum_{\substack{r_3 = -p +T+1 \\ r_3 \ne 0, \, r_3 \ne   j_3 -j_1}}^{p-(T+1)} |\sigma_{|r_3|}\sigma_{|r_3- j_3 + j_1|}| \Big) \nonumber \\
  && \Big( 2\,  | \sigma_{|j_1 - j_2|}| + \sum_{\substack{r_4 = -p +T+1 \\ r_4 \ne 0, \,  r_4 \ne  j_1 -j_2}}^{p-(T+1)} |\sigma_{|r_4|}\sigma_{|r_4- j_2+ j_1|} | \Big) \nonumber 
\end{eqnarray}
and hence
\begin{eqnarray}
\mathcal{G}_{3,1} &:=&  \underset{1 \leq j_1 \neq j_2 \neq j_3 <T}{\ds\sum } w_{j_1}^{*2} w_{j_2}^{*}  w_{j_3}^{*} \sigma_{|j_1 - j_2|}^2 \sigma_{|j_1 - j_3|}^2 \leq   ( \sup\limits_{j} w_{j}^*)^2\sum_{j_1 =1}^T w_{j_1}^{*2} \sum_{j_2 =1}^T \sigma_{|j_1 - j_2|}^2 \sum_{j_3 =1}^T \sigma_{|j_1 - j_3|}^2 \nonumber \\
& \leq &  ( \sup\limits_{j} w_{j}^*)^2 \cdot \frac 12 \cdot 4L^2 =o(1). \nonumber 
\end{eqnarray}
Note that  $\sup\limits_r \sigma_r \leq 1$ and by Cauchy-Schwarz we have $ \ds\sum_{\substack{r_4 = -p +T+1 \\ r_4 \ne 0, \,  r_4 \ne  j_1 -j_2}}^{p-(T+1)} |\sigma_{|r_4|}\sigma_{|r_4- j_2+ j_1|} | \leq \sum_{\substack{r_4 \\ r_4 \ne 0 }} \sigma_{r_4}^2 $. Thus we get,
\begin{eqnarray*}
\mathcal{G}_{3,2} & := &  \underset{1 \leq j_1 \neq j_2 \neq j_3 <T}{\ds\sum } w_{j_1}^{*2} w_{j_2}^{*}  w_{j_3}^{*} \sigma_{|j_1 - j_2|} \sigma_{|j_1 - j_3|}^2 \sum_{\substack{r_4 = -p +T+1 \\ r_4 \ne 0, \,  r_4 \ne  j_1 -j_2}}^{p-(T+1)} |\sigma_{|r_4|}\sigma_{|r_4- j_2+ j_1|} | \nonumber \\
& \leq &  (\sup\limits_{j} w_j^*) \cdot \sum_{j_1} w_{j_1}^{*2}  \sum_{j_2} w_{j_2}^* \sum_{j_3} \sigma_{|j_1 - j_3|}^2  \Big( \sum_{\substack{ |r_4| \leq j_2 \\ r_4 \ne 0}} \sigma_{r_4}^2 + \sum_{\substack{ |r_4| > j_2 \\ r_4 \ne 0}} \frac{ |r_4|^{2 \alpha }}{j_2^{2 \alpha}} \sigma_{|r_4|}^2 \Big) \nonumber \\
& \leq & (\sup\limits_{j} w_j^*) \cdot \ds\frac 12 \cdot 2L \cdot  \Big( T \cdot  \mathbb{E}(\widehat{\mathcal{A}}_n)  +  (\sup\limits_{j} w_j^*) \cdot \max \{ 1, T^{-2 \alpha +1} \} \cdot 2L \Big) =o(1).
\end{eqnarray*}
Moreover,
\begin{eqnarray*}
\mathcal{G}_{3,3} & := & \hspace{-0.3cm} \underset{1 \leq j_1 \neq j_2 \neq j_3 <T}{\ds\sum } w_{j_1}^{*2} w_{j_2}^{*}  w_{j_3}^{*} \sigma_{|j_1 - j_2|} \sigma_{|j_1 - j_3|}  \sum_{\substack{r_2 = -p +T+1 \\ r_2  \ne 0, \,  r_2 \ne  j_1 -j_3}}^{p-(T+1)} |\sigma_{|r_2|}\sigma_{|r_2- j_1 + j_3|} |  \sum_{\substack{r_4 = -p +T+1 \\ r_4 \ne 0, \,  r_4 \ne  j_1 -j_2}}^{p-(T+1)} |\sigma_{|r_4|}\sigma_{|r_4- j_2+ j_1|} | \nonumber \\
& \leq &  \sum_{j_1} w_{j_1}^{*2} \cdot \sum_{j_2} w_{j_2}^*  \Big( \sum_{\substack{ |r_2| \leq j_2 \\ r_2 \ne 0}} \sigma_{r_2}^2 + \sum_{\substack{ |r_2| > j_2 \\ r_2 \ne 0}} \frac{ |r_2|^{2 \alpha }}{j_2^{2 \alpha}} \sigma_{|r_2|}^2 \Big)
 \sum_{j_3} w_{j_3}^*  \Big( \sum_{\substack{ |r_4| \leq j_3 \\ r_4 \ne 0}} \sigma_{r_4}^2 + \sum_{\substack{ |r_4| > j_3 \\ r_4 \ne 0}} \frac{ |r_4|^{2 \alpha }}{j_3^{2 \alpha}} \sigma_{|r_4|}^2 \Big) \nonumber \\
 & \leq &\ds\frac 12 \cdot \Big( T \cdot  \mathbb{E}(\widehat{\mathcal{A}}_n)  +  (\sup\limits_{j} w_j^*) \cdot  \max \{ 1, T^{-2 \alpha +1} \} \cdot 2L \Big)^2 =o(1)
\end{eqnarray*}
and
\begin{eqnarray*}
\mathcal{G}_{3,4} & := &\underset{1 \leq j_1 \neq j_2 \neq j_3 <T}{\ds\sum } w_{j_1}^{*2} w_{j_2}^{*}  w_{j_3}^{*}  \sigma_{|j_2 - j_1|} \sum_{\substack{r_2 = -p +T+1 \\ r_2  \ne 0, \,   r_2 \ne j_1 -j_3}}^{p-(T+1)} |\sigma_{|r_2|}\sigma_{|r_2- j_1 + j_3|} |  \sum_{\substack{r_3 = -p +T+1 \\ r_3 \ne 0, \,  r_3 \ne  j_3 -j_1}}^{p-(T+1)} |\sigma_{|r_3|}\sigma_{|r_3- j_3 + j_1|}| \nonumber \\
&& \sum_{\substack{r_4 = -p +T+1 \\ r_4 \ne 0, \,  r_4 \ne  j_1 -j_2}}^{p-(T+1)} |\sigma_{|r_4|}\sigma_{|r_4- j_2+ j_1|} |  \nonumber \\
& \leq & \sum_{j_1} w_{j_1}^{*2} \cdot \sum_{j_2} w_{j_2}^*  \Big( \sum_{\substack{ |r_2| \leq j_2 \\ r_2 \ne 0}} \sigma_{r_2}^2 + \sum_{\substack{ |r_2| > j_2 \\ r_2 \ne 0}} \frac{ |r_2|^{2 \alpha }}{j_2^{2 \alpha}} \sigma_{|r_2|}^2 \Big)
 \sum_{j_3} w_{j_3}^*  \Big( \sum_{\substack{ |r_4| \leq j_3 \\ r_4 \ne 0}} \sigma_{r_4}^2 + \sum_{\substack{ |r_4| > j_3 \\ r_4 \ne 0}} \frac{ |r_4|^{2 \alpha }}{j_3^{2 \alpha}} \sigma_{|r_4|}^2 \Big) \cdot 2L \nonumber \\
 & =& o(1).
\end{eqnarray*}
Similarly we show that 
\begin{eqnarray*}
\mathcal{G}_{3,5} & := &\underset{1 \leq j_1 \neq j_2 \neq j_3 <T}{\ds\sum } w_{j_1}^{*2} w_{j_2}^{*}  w_{j_3}^{*}  \sum_{\substack{r_1 = -p +T+1 \\ r_1 \ne 0, \,  r_1 \ne  j_2 -j_1}}^{p-(T+1)} |\sigma_{|r_1|}\sigma_{|r_1- j_1 + j_2|}| \sum_{\substack{r_2 = -p +T+1 \\ r_2  \ne 0, \,  r_2 \ne  j_1 -j_3}}^{p-(T+1)} |\sigma_{|r_2|}\sigma_{|r_2- j_1 + j_3|} |  \nonumber \\
&& \cdot  \sum_{\substack{r_3 = -p +T+1 \\ r_3 \ne 0, \,  r_3 \ne  j_3 -j_1}}^{p-(T+1)} |\sigma_{|r_3|}\sigma_{|r_3- j_3 + j_1|}| \sum_{\substack{r_4 = -p +T+1 \\ r_4 \ne 0, \,  r_4 \ne  j_1 -j_2}}^{p-(T+1)} |\sigma_{|r_4|}\sigma_{|r_4- j_2+ j_1|} |  =o(1).
\end{eqnarray*}
Finally, when all indices are pairwise distinct. We use the same arguments as previously, and we get, 
\begin{eqnarray*}
\mathcal{G}_4 &:=&  4  \underset{1 \leq j_1 \neq j_2 \neq j_3 \neq j_4 <T}{\ds\sum }  w_{j_{1}}^* w_{j_{2}}^* w_{j_{3}}^* w_{j_{4}}^* \underset{-p +T+1  \leq r_1, r_2, r_3, r_4 \leq p-(T+1)}{\sum }  |\sigma_{|r_1|}\sigma_{|r_1- j_1 + j_2|} \sigma_{|r_2|}\sigma_{|r_2- j_3 + j_4|}| \nonumber \\
&& \hspace{7cm} \cdot | \sigma_{|r_3|}\sigma_{|r_3- j_1 + j_3|} \sigma_{|r_4|}\sigma_{|r_4- j_2 + j_4|}| \nonumber \\
&= &  4  \underset{1 \leq j_1 \neq j_2 \neq j_3 \neq j_4 <T}{\ds\sum }  w_{j_{1}}^* w_{j_{2}}^* w_{j_{3}}^* w_{j_{4}}^* \Big( 2 \, \sigma_0 \, | \sigma_{|j_1 - j_2|} | + \sum_{r_1 \ne 0 , \,  r_1 \ne j_1 - j_2 }  |\sigma_{|r_1|}\sigma_{|r_1- j_1 + j_2|}| \,  \Big) \nonumber \\
&& \Big( 2 \, \sigma_0 \, | \sigma_{|j_3 - j_4|} | + \sum_{r_2 \ne 0 , \, r_2 \ne  j_3 - j_4 }  |\sigma_{|r_2|}\sigma_{|r_2 - j_3 + j_4|}| \,  \Big)  
\Big( 2 \, \sigma_0 \, | \sigma_{|j_1 - j_3|} | + \sum_{r_3 \ne 0 , \,  r_3 \ne j_1 - j_3 }  |\sigma_{|r_3|}\sigma_{|r_3 - j_1 + j_3|}| \,  \Big)
\nonumber \\ 
&&  \Big( 2 \, \sigma_0 \, | \sigma_{|j_2 - j_4|} | + \sum_{r_4 \ne 0 , \,  r_4 \ne j_2 - j_4 }  |\sigma_{|r_4|}\sigma_{|r_4 - j_2 + j_4|}| \,  \Big) \nonumber 
\end{eqnarray*}
Now, we treat each term of $\mathcal{G}_4 $ separately:
\begin{eqnarray*}
\mathcal{G}_{4,1} & :=&     \underset{1 \leq j_1 \neq j_2 \neq j_3 \neq j_4 <T}{\ds\sum}  w_{j_{1}}^* w_{j_{2}}^* w_{j_{3}}^* w_{j_{4}}^*   | \sigma_{|j_1 - j_2|} \sigma_{|j_3 - j_4|}  \sigma_{|j_1 - j_3|}  \sigma_{|j_2 - j_4|}  |  \nonumber \\ 
& \leq & \cdot  \Big( \sum_{j_1,\,j_2} w_{j_{1}}^* w_{j_{2}}^* \sigma_{|j_1 - j_2|}^2 \Big)^{ \frac 12}  \Big( \sum_{j_1,\,j_3} w_{j_{1}}^* w_{j_{3}}^* \sigma_{|j_1 - j_3|}^2 \Big)^{ \frac 12}   \Big( \sum_{j_2,\,j_4} w_{j_{2}}^* w_{j_{4}}^* \sigma_{|j_2 - j_4|}^2 \Big)^{ \frac 12}  \Big( \sum_{j_3,\,j_4} w_{j_{3}}^* w_{j_{4}}^* \sigma_{|j_3 - j_4|}^2 \Big)^{ \frac 12} \nonumber \\
& \leq & \cdot \Big( \sum_{j_1} w_{j_1}^* \sum_{ \substack{j_2 \\ j_2 \leq |j_1 - j_2|}} w_{|j_1 - j_2|}^* \sigma_{|j_1 - j_2|}^2  \,  + ( \sup\limits_j w_j^* ) \cdot \sum_{j_1}  \sum_{ \substack{j_2 \\ j_2 > |j_1 - j_2|}} \frac{|j_1 - j_2|^{2 \alpha}}{j_2^{2 \alpha}} \sigma_{|j_1 - j_2|}^2 \Big)^2 \nonumber \\
& \leq & \Big( O( \ds\sqrt{T} ) \cdot \mathbb{E}_{\Sigma}(\widehat{\mathcal{A}}_n) + O \Big( \ds\frac 1{\sqrt{T}} \Big) \cdot \max\{ 1, T^{-2 \alpha +1 } \} \cdot L \Big)^2 =o(1) \nonumber
\end{eqnarray*}
and
\begin{eqnarray*}
\mathcal{G}_{4,2} & := &  \underset{1 \leq j_1 \neq j_2 \neq j_3 \neq j_4 <T}{\ds\sum }  w_{j_{1}}^* w_{j_{2}}^* w_{j_{3}}^* w_{j_{4}}^*   | \sigma_{|j_1 - j_2|} \sigma_{|j_3 - j_4|}  \sigma_{|j_1 - j_3|} | \sum_{r_4 \ne 0 , \,  r_4 \ne j_2 - j_4 }  |\sigma_{|r_4|}\sigma_{|r_4 - j_2 + j_4|}| \nonumber \\
& \leq & \underset{ j_2, j_4 }{ \sum \sum } w_{j_{2}}^* w_{j_{4}}^*   \Big( \sum_{j_1} \sum_{j_3} w_{j_{1}}^* w_{j_{3}}^* \sigma_{|j_1 - j_3|}^2 \Big)^{ \frac 12} \Big( \sum_{j_1} w_{j_{1}}^*  \sigma_{|j_1 - j_2|}^2 \Big)^{ \frac 12}  \Big(  \sum_{j_3} w_{j_{3}}^* \sigma_{|j_3 - j_4|}^2 \Big)^{ \frac 12} \nonumber \\
&& \cdot  \Big( \sum_{ \substack{ r_4 \ne 0 , \, r_4 \ne  j_2 - j_4 \\ |r_4| \leq  j_2} }   \sigma_{|r_4|}^2  +  \sum_{ \substack{ r_4 \ne 0 , \,  r_4 \ne j_2 - j_4 \\ |r_4 | \leq  j_2} }   \sigma_{|r_4|}^2 \Big)\nonumber \\
& \leq &\Big( \sum_{j_1} \sum_{j_3} w_{j_{1}}^* w_{j_{3}}^* \sigma_{|j_1 - j_3|}^2 \Big)^{ \frac 12} \Big((\sup\limits_j w_{j}^*) \cdot L \Big) \nonumber \\
&& \cdot  \Big(  \, \underset{ j_2, j_4 }{ \sum \sum }   w_{j_{4}}^*  \sum_{ \substack{ r_4 \ne 0 \\ |r_4 | \leq  j_2} }  w_{j_{2}}  \sigma_{|r_4|}^2  +   \underset{ j_2, j_4 }{ \sum \sum } w_{j_{2}}^* w_{j_{4}}^*   \sum_{ \substack{ r_4 \ne 0 \\ |r_4| \leq  j_2} }    \frac{|r_4|^{2\alpha}}{j_2^{2 \alpha}} \sigma_{|r_4|}^2   \Big)  \nonumber \\
& \leq & \Big( O( \ds\sqrt{T} ) \cdot \mathbb{E}_{\Sigma}(\widehat{\mathcal{A}}_n) + O \Big( \ds\frac 1{\sqrt{T}} \Big) \cdot \max\{ 1, T^{-2 \alpha +1 } \} \Big)^{\frac 12} \cdot O \Big( \ds\frac 1{\ds\sqrt{T}} \Big) \nonumber \\
 && \cdot \Big( T^2 \cdot (\sup\limits_j w_j^* ) \cdot \mathbb{E}_{\Sigma}(\widehat{\mathcal{A}}_n) + T \cdot (\sup\limits_j w_{j}^* )^2 \cdot \max \{ 1, T^{-2\alpha +1 } \cdot L \} \Big) \nonumber \\
 & \leq & \Big( O( \ds\sqrt{T} ) \cdot \mathbb{E}_{\Sigma}(\widehat{\mathcal{A}}_n) + O \Big( \ds\frac 1{\sqrt{T}} \Big) \cdot \max\{ 1, T^{-2 \alpha +1 } \} \cdot 2L \Big)^{\frac 12} \nonumber \\
 &&\cdot \Big( T \cdot \mathbb{E}_{\Sigma}(\widehat{\mathcal{A}}_n) + O \Big( \ds\frac 1{\ds\sqrt{T}} \Big) \cdot \max \{ 1, T^{-2\alpha +1 } \}\cdot L \Big) \nonumber \\
 &= & o(1) ~~\text{ since }  \mathbb{E}(\widehat{\mathcal{A}}_n)  \asymp 1/np \text{ and for all } \alpha >1/4. \nonumber 
\end{eqnarray*}
We use similar argument as previously to show that the remaining terms in $\mathcal{G}_4$ tend to zero.
To complete the proof, we need to verify that,
\begin{equation}
\label{conditionNA2}
 \mathbb{E}_{\Sigma}(H_n^4(X_1, X_2)) / \mathbb{E}_{\Sigma}^2(H_n^2(X_1, X_2)) = o(n).  
 \end{equation}
 We write
\begin{eqnarray}
&& \ds\frac{\mathbb{E}_{\Sigma}(H_n^4(X_1, X_2))}{ \mathbb{E}_{\Sigma}^2(H_n^2(X_1, X_2))}  = \ds\frac{1}{(p-T)^4} \ds\sum_{j_1,\,j_2,\,j_3,\, j_4} ~ w_{j_1}^*  w_{j_2}^* w_{j_3}^* w_{j_4}^*~ \underset{T +1 \leq i_1, i_3, i_5, i_7 \leq p}{\sum } ~~ \underset{T +1 \leq i_2, i_4, i_6, i_8 \leq p}{\sum } \nonumber \\
&& \mathbb{E}_{\Sigma} [(X_{1, i_1}X_{1, i_1-j_1} - \sigma_{j_1})(X_{1, i_3}X_{1, i_3-j_2} - \sigma_{j_2})(X_{1, i_5}X_{1, i_5-j_3} - \sigma_{j_3})  (X_{1, i_7}X_{1, i_7-j_4} - \sigma_{j_4})] \nonumber \\
&& \cdot \mathbb{E}_{\Sigma}[(X_{2, i_2}X_{2, i_2-j_1} - \sigma_{j_1}) (X_{2, i_4}X_{2, i_4-j_2} - \sigma_{j_2}) (X_{2, i_6}X_{2, i_6-j_3} - \sigma_{j_3})(X_{2, i_8}X_{2, i_8-j_4} - \sigma_{j_4}) ] \nonumber 
\end{eqnarray}
To bound from above the previous sum, we replace the expected value by it's value, which is a sum of many terms, that are all treated similarly. So let us give an upper bound for the following one :
\begin{eqnarray}
\mathcal{H} &:=& \ds\frac{1}{(p-T)^4} \ds\sum_{j_1,\,j_2,\,j_3,\, j_4} ~ w_{j_1}^*  w_{j_2}^* w_{j_3}^* w_{j_4}^* ~ \underset{T +1 \leq i_1, i_3, i_5, i_7 \leq p}{\sum } ~~ \underset{T +1 \leq i_2, i_4, i_6, i_8 \leq p}{\sum } \sigma_{|i_1-i_3|}\sigma_{|i_1 -i_3-j_1+j_2|} \nonumber \\
&& \hspace{4cm} \cdot \sigma_{|i_5 -i_7|} \sigma_{|i_5 -i_7-j_3 +j_4|} \sigma_{|i_2-i_4|}\sigma_{|i_2 -i_4-j_1+j_2|} \sigma_{|i_6 -i_8|} \sigma_{|i_6 -i_8-j_3 +j_4 |} \nonumber \\
& \leq &  \ds\sum_{j_1,\,j_2,\,j_3,\, j_4} w_{j_1}^*  w_{j_2}^* w_{j_3}^* w_{j_4}^*  \underset{-p +1  \leq r_1, r_2, r_3, r_4 \leq p-1}{\sum} \sigma_{|r_1|} \sigma_{|r_1 -j_1 + j_2|}\sigma_{|r_2|} \sigma_{|r_2 -j_1 + j_2|} \nonumber \\
&& \hspace{8cm} \cdot \sigma_{|r_3|} \sigma_{|r_3 -j_3 + j_4|} \sigma_{|r_4|} \sigma_{|r_4 -j_3 + j_4|} \nonumber 
\end{eqnarray}  
We  see that $\mathcal{H}$ can be treated in the same way as  $\mathcal{G}$. However, we show that $\mathcal{H} = O(1) =o(n)$. 
Let us deal with one of the terms  of $ \mathcal{H}$,  consider the term for which we have $j_1 =j_2 $, $j_3 = j_4$, and $j_1 \ne j_3$ thus we get 
\begin{eqnarray*}
&& \underset{1 \leq j_1 \ne j_3 < T}{ \sum }   w_{j_1}^{*2}  w_{j_3}^{*2}  \underset{-p +1  \leq r_1, r_2, r_3, r_4 \leq p-1}{\sum } \sigma_{|r_1|}^2 \sigma_{|r_2|}^2
  \sigma_{|r_3|}^2 \sigma_{|r_4|}^2 = \underset{1 \leq j_1 \ne j_3 < T}{ \sum }   w_{j_1}^{*2}  w_{j_3}^{*2} \cdot  \Big( \sigma_0^2 +  \sum_{r_1 \ne 0 } \sigma_{|r_1|}^2  \Big)^2 \nonumber \\
  &\leq & 2   \underset{1 \leq j_1 \ne  j_3 < T}{ \sum \sum}   w_{j_1}^{*2}  w_{j_3}^{*2}  +  2\underset{1 \leq j_1 \ne j_3 < T}{ \sum \sum}   w_{j_1}^{*2}  w_{j_3}^{*2} \Big(   \sum_{r_1 \ne 0 } \sigma_{|r_1|}^2  \Big)^2 .
\end{eqnarray*}
It is easily seen that $ \underset{1 \leq j_1 \neq j_2 < T}{ \ds\sum \sum}   w_{j_1}^{*2}  w_{j_2}^{*2} = O(1)$. And so on, we show that all terms in $\mathcal{H}$ are $O(1)$ and thus we get the desired result.
Together with (\ref{conditionNA1}), this proves (\ref{conditionNA}).  In consequence, we apply theorem 1 of  \cite{Hall84}, to get (\ref{convinlaw}).
\hfill \end{proof}

\begin{proof}[Proof of (\ref{ANH0})]
We define $\widehat{\mathcal{B}}_{n,p}$ as follows,
\begin{eqnarray*}
\widehat{\mathcal{B}}_{n,p} &= &\ds\frac{2}{\ds\sqrt{n(n-1)(p-T)(p-T-1)}}  \sum_{ i=T+1}^p \sum_{h=  i+1 }^p \underset{1 \leq k \neq l \leq n}{\ds\sum }\sum_{j=1}^{T-1} w_j^*  ~ X_{k, i}X_{k, i-j} X_{l, h}X_{l, h-j}  
\end{eqnarray*}
We set 
\begin{eqnarray*}
D_{n,p, i} &=& \ds\frac{2}{\ds\sqrt{n(n-1)(p-T)(p-T-1)}} \sum_{h=  i+1 }^p \underset{1 \leq k \neq l \leq n}{\ds\sum }\sum_{j=1}^{T-1} w_j^*  ~ X_{k, i}X_{k, i-j} X_{l, h}X_{l, h-j} \\
& :=& c(n,p,T)  \sum_{h=  i+1 }^p \underset{1 \leq k \neq l \leq n}{\ds\sum }\sum_{j=1}^{T-1} w_j^*  ~ X_{k, i}X_{k, i-j} X_{l, h}X_{l, h-j} 
\end{eqnarray*}
Note that the $\{D_{n,p,i}\}_{T+1 \leq i \leq p}$ is a sequence of martingale differences with respect to  the sequence of  $\sigma$ fields $\{\mathcal{F}_i, i \geq T +1 \}$  such that $\mathcal{F}_i = \sigma\{ X_{. , r}~, r \leq i \}$, we denote by $\mathbb{E}_{i}(\cdot) = \mathbb{E}(\cdot / \mathcal{F}_{i})$, where $\mathbb{E}$ is the expected value under the null hypothesis. Indeed, for all $T+1 \leq i \leq p$, we have,
$\mathbb{E}_{i-1}(D_{n,p, i} )  =0.$
We use sufficient conditions to show the asymptotic normality of a sum of martingale differences $\widehat{\mathcal{B}}_{n,p}$ for all $n \geq 2$, as $(p-T) \to \infty$, see e.g. \cite{shiryaev96}.  Thus it suffices to show that,
\begin{equation}
\label{conditions}
\mathbb{E}\Big(\ds\sum_{i=T+1}^p \mathbb{E}_{i-1}(D_{n,p, i}^2) -1  \Big)^2 \to 0 \quad \text{ and }  \quad   \sum_{i=T+1}^p \mathbb{E}(D_{n,p, i}^4 ) \to 0. 
\end{equation}
We first show the first part of \eqref{conditions}.
\begin{eqnarray*}
 \mathbb{E}_{i-1}(D_{n,p, i}^2 ) 
&=& (c(n,p,T))^2 \ds\sum_{h= i+1}^p \, \underset{1 \leq k \neq l \leq n}{\ds\sum }  \, \sum_{1 \leq j,j_1 \leq < T-1} w_{j}\* w_{j_1}^* X_{k, i-j} X_{k, i-j_1} \mathbb{E}_{i-1}( X_{l, h-j} X_{l, h-j_1} )   \\
&=& (c(n,p,T))^2 \cdot  \Big( \sum_{1 \leq j,j_1 \leq < T-1}   \underset{1 \leq k \neq l \leq n}{\ds\sum }  w_{j}\* w_{j_1}^* X_{k, i-j} X_{k, i-j_1}   \ds\sum_{h= i+1}^{(i+j_1-1) \wedge (i+j̈́-1)} X_{l, h-j}  X_{l, h-j_1} \\
&& \hspace{2cm} + \,  (n-1) \sum_{k=1}^n \sum_{j=1}^T w_j^{*2} X_{k, i-j}^2 (p-i-j+1) \Big)
\end{eqnarray*}
giving 
\begin{eqnarray*}
&& \mathbb{E}(\sum_{i=T+1}^p \mathbb{E}_{i-1}(D_{n,p, i}^2 ) ) \nonumber \\
 &=& c^2(n,p,T) \cdot \Big( n(n-1)\sum_{i=T+1}^p \sum_{j=1}^{T-1} w_j^{*2} (j-1) 
 + n(n-1)\sum_{i=T+1}^p \sum_{j=1}^{T-1} w_j^{*2}  (p-i-j+1) \Big) \\
&=&  \ds\frac 4{(p-T)(p-T-1)} \sum_{j=1}^{T-1} w_j^{*2} \sum_{i=T+1}^p (p-i) =1 .
\end{eqnarray*} 
Thus, to show that $   \mathbb{E}\Big(\ds\sum_{i=T+1}^p \mathbb{E}_{i-1}(D_{n,p, i}^2) -1  \Big)^2 \to 0$, it is sufficient to show that $ \mathbb{E}\Big(\ds\sum_{i=T+1}^p \mathbb{E}_{i-1}(D_{n,p, i}^2)  \Big)^2 = 1+ o(1)$.
Indeed,
\begin{equation}
\label{square}
 \mathbb{E}\Big(\ds\sum_{i=T+1}^p \mathbb{E}_{i-1}(D_{n,p, i}^2)  \Big)^2 = (c(n,p,T))^4 \cdot  \Big(E_1 + E_2 + E_3 +E_4 \Big). 
\end{equation}
where $E_1, E_2$ , $E_3$ and $E_4$ are given by the following.
\begin{eqnarray*}
E_1 
& = & \ds\sum_{T+1  \leq i,i' \leq p }  \underset{1 \leq k \neq l \leq n}{\ds\sum } ~ \underset{1 \leq k' \neq l' \leq n}{\ds\sum } ~ \sum_{ 1 \leq j, j_1 \leq T} \sum_{ 1 \leq j', j'_1 \leq T-1}  \sum_{ h = i+1 }^{ ( i+j-1) \wedge (i+ j_1-1)}   \sum_{ h' = i'+1 }^{ ( i'+j'-1) \wedge (i'+ j'_1-1)} \nonumber \\
&&    w_j^{*} w_{j_1}^{*}  w_{j'}^{*} w_{j'_1}^{*} \, \mathbb{E}(X_{k, i-j}X_{k, i-j_1} X_{l, h-j} X_{l, h-j_1} X_{k', i'-j'} X_{k', i'-j_1'}  X_{l', h'-j'} X_{l', h'-j_1'}) 
\end{eqnarray*}
Now we decompose  $E_1$  into  five sums that depends on the indices  $k,k',l$ and $l'$. We begin by the first case when $k=k'$ and $l=l'$,
\begin{eqnarray*}
E_{1,1} &:=&  \underset{1 \leq k \neq l \leq n}{\ds\sum } \left( \sum_{i=T+1}^p \Big( \, \sum_{j=1}^{T-1}w_j^{*4} \, 3 \cdot (3  (j-1) +  (j-1)(j-2) ) +  \sum_{ 1 \leq j \ne  j' \leq T}  w_j^{*2} w_{j'}^{*2} \, (j-1)(j'-1)  \right.\\
&& \left. \hspace{2cm}  + \, 2   \sum_{1 \leq j \ne j_1 \leq T-1} w_j^{*2} w_{j_1}^{*2}  \Big( (j-1) \wedge (j_1 -1) \Big)  \right. \\ 
&+& \left.  \ds\sum_{T+1  \leq i \ne i' \leq p } \Big( \, \sum_{j=1}^{T-1}w_j^{*4} \, (3  (j-1) +  (j-1)(j-2) )  +   \sum_{ 1 \leq j \ne  j' \leq T}  w_j^{*2} w_{j'}^{*2} \, (j-1)(j'-1)  \right) \\
&=& n(n-1) \, \left( 2 \sum_{i=T +1}^{p}  \sum_{j=1}^{T-1} w_j^{*4} (j-1)(j+1) + 2 \ds\sum_{T+1  \leq i , i' \leq p } \sum_{j=1}^{T-1} w_j^{*4} \, (j-1) \right. \\
&+ & \left. \ds\sum_{T+1  \leq i , i' \leq p } \sum_{ 1 \leq j ,  j' \leq T} w_j^{*2} w_{j'}^{*2} \,   (j-1)(j'-1) + 2 \sum_{i=T+1}^p \sum_{1 \leq j \ne j' \leq T-1}w_j^{*2} w_{j'}^{*2} \Big( (j-1) \wedge (j'-1) \Big)    \right)
\end{eqnarray*}
When $k=l'$ and $l=k'$, we have using similar arguments as previously that,
\begin{eqnarray*}
E_{1,2} &=& n(n-1)  \ds\sum_{T+1  \leq i , i' \leq p } \, \sum_{ 1 \leq j ,  j' \leq T} \, w_j^{*2} w_{j'}^{*2} \, (j-1)(j'-1)
\end{eqnarray*}
We move to the term, when $k=k'$ and $l \ne l'$,
\begin{eqnarray*}
E_{1,3} &:=& \ds\sum_{ \substack{1 \leq k , l, l' \leq n \\ k \ne l,l' , l \ne l'}} \Big\{   \sum_{j=1}^{T-1} w_j^{*4}  \Big( \sum_{i=T+1}^p  3 \, (j-1)^2 + \ds\sum_{T+1  \leq i \ne i' \leq p } (j-1)^2 \Big) \nonumber \\
&+&   \sum_{ 1 \leq j \ne j' \leq T} w_j^{*2} w_{j'}^{*2} \ds\sum_{T+1  \leq i ,i' \leq p }   (j-1)(j'-1) \, \Big\} \\
&=& n(n-1)(n-2) \,  \Big( 2\sum_{i=T+1}^p   \sum_{j=1}^{T-1} w_j^{*4} (j-1)^2 +   \ds\sum_{T+1  \leq i , i' \leq p } \, \sum_{ 1 \leq j , j' \leq T-1} w_j^{*2} w_{j'}^{*2} (j-1)(j'-1) \Big)
\end{eqnarray*}
Now we treat the case when , $k \ne k'$ and $l=l'$,
\begin{eqnarray*}
E_{1,4} &:=& \ds\sum_{ \substack{1 \leq k , k', l \leq n \\ l \ne k, k' ,  k \ne k'}} \Big\{\sum_{j=1}^{T-1} w_j^{*4} \ds\sum_{T+1  \leq i, i' \leq p }(3 (j-1) +(j-1)(j-2) ) \nonumber \\
& + & \sum_{ 1 \leq j \ne j' \leq T} \ds\sum_{T+1  \leq i , i' \leq p } w_j^{*2} w_{j'}^{*2} \cdot (j-1)(j'-1)  \, \Big\} \\
&=& n(n-1)(n-2) \ds\sum_{T+1  \leq i,i' \leq p }  \,  \Big\{ \sum_{j=1}^{T-1} w_j^{*4} \, (j-1)(j+1) + \sum_{ 1 \leq j \ne j' \leq T} w_j^{*2} w_{j'}^{*2} \cdot (j-1)(j'-1) \, \Big\} \\
&=& n(n-1)(n-2) \ds\sum_{T+1  \leq i,i' \leq p } \Big( \sum_{ 1 \leq j , j' \leq T-1} w_j^{*2} w_{j'}^{*2} \, (j-1)(j'-1) + 2\sum_{j=1}^{T-1} w_j^{*4} \, (j-1) \, \Big)
\end{eqnarray*}
Finally, we treat the term for $k \ne k'$ and $l \ne l'$,
\begin{eqnarray*}
E_{1,5} &:=& \underset{ k \ne k' , l \ne l'}{\underset{1 \leq k \neq l \leq n}{\ds\sum } ~ \underset{1 \leq k' \neq l' \leq n}{\ds\sum } } \ds\sum_{T+1  \leq i,i' \leq p }  \sum_{ 1 \leq j, j' \leq T}w_j^{*2} w_{j'}^{*2} \, (j-1)(j'-1) \nonumber \\
&= & n(n-1)^2(n-2) \ds\sum_{T+1  \leq i,i' \leq p } \sum_{ 1 \leq j, j' \leq T}w_j^{*2} w_{j'}^{*2} \, (j-1)(j'-1) 
\end{eqnarray*} 
We group the previous result to get,
\begin{eqnarray*}
E_1 \!\! &=&  \! \! \Big( 2n(n-1) + 2n(n-1)(n-2) + n(n-1)^2(n-2) \Big) \ds\sum_{T+1  \leq i, i' \leq p }  \sum_{ 1 \leq j, j' \leq T}w_j^{*2} w_{j'}^{*2} \, (j-1)(j'-1) \nonumber \\
& +& R_1(n, p,T)
\end{eqnarray*}
where,
\begin{eqnarray*} 
&&R_1(n ,p,T) \\ &= &  2\Big(n(n-1) + n(n-1)(n-2) \Big) \ds\sum_{T+1  \leq i,i' \leq p } \sum_{j=1}^{T-1} w_j^{*4}(j-1) + 2 n(n-1) \sum_{i=T +1}^{p}  \sum_{j=1}^{T-1} w_j^{*4} (j-1)(j+1)  \\
&+&    2 n(n-1)(n-2) \sum_{i=T+1}^p   \sum_{j=1}^{T-1} w_j^{*4} (j-1)^2  + 2n(n-1) \sum_{i=T+1}^p \sum_{1 \leq j \ne j' \leq T-1}w_j^{*2} w_{j'}^{*2} \Big( (j-1) \wedge (j'-1) \Big) \\
&=& o((c(n,p,T)^{-4})
\end{eqnarray*}
Now, let us bound from above the term $E_2$ in \eqref{square}:
\begin{eqnarray*}
 E_2 &:=& (n-1) \, \ds\sum_{T+1  \leq i,i' \leq p } \sum_{1 \leq k \ne l \leq n}\sum_{ 1 \leq j, j_1 \leq T} \sum_{k'=1}^n \sum_{j'=1}^T  w_j^{*} w_{j_1}^{*} w_{j'}^{*2} \mathbb{E}( X_{k, i-j} X_{k, i-j_1}  X_{k', i'-j'}^2 ) (p-i'-j'+1) \\
 && \cdot \sum_{  h=i+1}^{ ( i+j-1) \wedge (i+ j_1-1)}  \mathbb{E}( X_{l, h-j} X_{l, h-j_1}) 
 \end{eqnarray*}
We treat the two cases $k=k'$ and $k \ne k'$ each one apart. We begin by the case when $k \ne k'$,
 \begin{eqnarray*}
 E_{2,1} &:=&  n(n-1)^2 \,  \Big\{ \sum_{j=1}^{T-1} w_j^{*4} \Big( \sum_{i =T+1}^p \cdot 3 ( p-i-j+1)(j-1) + \ds\sum_{T+1  \leq i \ne i' \leq p }   ( p-i'-j+1)(j-1) \Big) \\
 &+&   \sum_{ 1 \leq j \ne j' \leq T} w_j^{*2} w_{j'}^{*2} \Big( \sum_{i=T+1}^p ( p-i-j'+1)(j-1) + \ds\sum_{T+1  \leq i \ne i' \leq p } (p-i'-j'+1)(j-1) \Big)  \\
 &= & n(n-1)^2 \, \Big\{\,  2 \sum_{i =T+1}^p   \sum_{j=1}^{T-1} w_j^{*4} \,( p-i-j+1)(j-1) \\
 &+&
 \ds\sum_{T+1  \leq i , i' \leq p } \sum_{ 1 \leq j , j' \leq T-1} w_j^{*2} w_{j'}^{*2} (p-i'-j'+1)(j-1) \, \Big\}  
\end{eqnarray*}
When $k \ne k'$,
\begin{eqnarray*}
E_{2,2} &:=& n(n-1)^3 \ds\sum_{T+1  \leq i,i' \leq p } \sum_{ 1 \leq j, j' \leq T}w_j^{*2} w_{j'}^{*2}   (p-i'-j'+1)  (j-1) .
\end{eqnarray*}
As consequence
\begin{eqnarray*}
E_2 &=&  \Big( n(n-1)^2 + n(n-1)^3 \Big) \ds\sum_{T+1  \leq i , i' \leq p } \sum_{ 1 \leq j , j' \leq T-1} w_j^{*2} w_{j'}^{*2} (p-i'-j'+1)(j-1) + o((c(n,p,T)^{-4}).
\end{eqnarray*}
Similarly we get, 
\begin{eqnarray*}
E_3 &=& (n-1) \, \ds\sum_{T+1  \leq i,i' \leq p } \sum_{1 \leq k' \ne l' \leq n}\sum_{ 1 \leq j', j'_1 \leq T-1} \sum_{k=1}^n \sum_{j=1}^{T-1}  w_{j'}^{*} w_{j'_1}^{*} w_{j}^{*2} \mathbb{E}( X_{k', i'-j'} X_{k', i'-j'_1}  X_{k, i-j}^2 ) (p-i-j+1) \\
 && \cdot \sum_{  h'=i'+1}^{ ( i'+j'-1) \wedge (i'+ j'_1-1)}  \mathbb{E}( X_{l', h'-j'} X_{l', h'-j'_1}) \\
 &=& n^2 (n-1)^2 \ds\sum_{T+1  \leq i , i' \leq p } \sum_{ 1 \leq j , j' \leq T-1} w_j^{*2} w_{j'}^{*2} (p-i-j+1)(j'-1) \\
 &  +&   2 n(n-1)^2  \sum_{i =T+1}^p \sum_{j=1}^{T-1} w_j^{*4} \,( p-i-j+1)(j-1) 
\end{eqnarray*}
The term $E_4$ of \eqref{square}  is treated as follows,
\begin{eqnarray*}
E_4 &:=& (n-1)^2   \ds\sum_{T+1  \leq i,i' \leq p } \sum_{1 \leq k, k' \leq n} \sum_{ 1 \leq j, j' \leq T} w_j^{*2} w_{j'}^{*2} \mathbb{E}(X_{k,i-j}^2 X_{k', i'-j'}^2) (p-i-j+1)(p-i'-j'+1) \\
&=& n (n-1)^2 \, \Big\{  \sum_{i=T+1}^p \Big( \sum_{ 1 \leq j \leq T} w_j^{*4}  3 (p-i-j+1)^2 +  \sum_{ 1 \leq j \ne j' \leq T}   w_j^{*2} w_{j'}^{*2}  (p-i-j+1)(p-i-j'+1) \Big) \nonumber \\
&& \hspace{1cm} + \ds\sum_{T+1  \leq i \ne i' \leq p }  \sum_{ 1 \leq j , j' \leq T-1}   w_j^{*2} w_{j'}^{*2}   (p-i-j+1) (p-i'-j'+1)  \, \Big\} \\
&+&n (n-1)^3 \ds\sum_{T+1  \leq i  , i' \leq p }  \sum_{ 1 \leq j , j' \leq T-1}  w_j^{*2} w_{j'}^{*2}  (p-i-j+1) (p-i'-j'+1) \Big\} \nonumber \\
&=& n^2(n-1)^2 \!\!\! \ds\sum_{T+1  \leq i , i' \leq p }  \sum_{ 1 \leq j , j' \leq T-1}  \!\!\! w_j^{*2} w_{j'}^{*2}  (p-i-j+1) (p-i'-j'+1) \\
&& +   2 n (n-1)^2  \sum_{i=T+1}^p  \sum_{ 1 \leq j \leq T} w_j^{*4}   (p-i-j+1)^2
\end{eqnarray*}
 Finally we group all the previous terms and  obtain,
\begin{eqnarray*}
\mathbb{E}\Big(\ds\sum_{i=T+1}^p \mathbb{E}_{i-1}(D_{n,p, i}^2 ) \Big)^2 &= & c^4(n,p,T) \, \Big\{ \cdot n^2 (n-1)^2 \ds\sum_{T+1  \leq i , i' \leq p }  \sum_{ 1 \leq j , j' \leq T-1}   w_j^{*2} w_{j'}^{*2}  \Big( (j-1)(j'-1)\\
&+&  (p-i'-j'+1)(j-1) 
+ (p-i-j+1)(j'-1)    \\ 
&+&  (p-i-j+1) (p-i'-j'+1) \Big) + o( (c(n ,p,T))^{-4}) \,  \Big\} \\
&=& \ds\frac{16}{(p-T)^2(p-T-1)^2}  \sum_{ 1 \leq j , j' \leq T-1}   w_j^{*2} w_{j'}^{*2} \ds\sum_{T+1  \leq i , i' \leq p }(p-i)(p-i') ) +o(1). \\
&=& \ds\frac{16}{(p-T)^2(p-T-1)^2} \cdot \frac 14 \cdot \Big( \ds\frac{(p-T-1)(p-T)}2 \Big)^2 +o(1) =1 +o(1)
\end{eqnarray*} 
To achieve the proof, we show that the second condition given in \eqref{conditions} is also verified. Indeed, 
\begin{eqnarray*}
 \sum_{i=T+1}^p \mathbb{E}(D_{n,p, i}^4 )
 &=&  (c(n,p,T))^4   \sum_{i=T+1}^p ~ \sum_{i+1 \leq h_1,h_2,h_3,h_4 \leq p } ~ \underset{1 \leq k_1 \neq l_1 \leq n}{\ds\sum } ~ \underset{1 \leq k_2 \neq l_2 \leq n}{\ds\sum } ~ \underset{1 \leq k_3 \neq l_3 \leq n}{\ds\sum } ~ \underset{1 \leq k_4 \neq l_4 \leq n}{\ds\sum } ~
\\
 && \hspace{-1cm}  \sum_{1 \leq j_1, j_2, j_3, j_4 \leq T-1}  w_{j_1}^*  w_{j_2}^*  w_{j_3}^*  w_{j_4}^* \mathbb{E}( X_{k_1, i}X_{k_2, i}  X_{k_3, i} X_{k_4, i} X_{l_1, h_1-j_1} X_{l_2, h_2-j_2} X_{l_3, h_3-j_3} X_{l_4, h_4-j_4}) \\
 &&\hspace{2cm} \cdot  \mathbb{E}(X_{l_1, h_1}  X_{l_2, h_2} X_{l_3, h_3} X_{l_4, h_4}X_{k_1, i-j_1}  X_{k_2, i-j_2} X_{k_3, i-j_3} X_{k_4, i-j_4}) \\
 &=& O(1) \cdot (c(n,p,T))^4  \cdot \sum_{i=T+2}^p ~ \sum_{T+1 \leq h_1,h_2 \leq p } ~ \underset{1 \leq k_1 \neq l_1 \leq n}{\ds\sum } ~ \underset{1 \leq k_2 \neq l_2 \leq n}{\ds\sum } ~
 \sum_{1 \leq j_1, j_2 \leq T}  w_{j_1}^{*2}  w_{j_2}^{*2} \\
 &=& \frac{O(1)}{(p-T)^2(p-T-1)^2} \cdot (p-T)^3 = o(1).
\end{eqnarray*}
\end{proof}

\subsection{Proofs of results in Section\ref{sec:testprob2}}

\begin{proof}[Proof of Proposition  \ref{prop:est2}]
To show the upper bound for  the variance of $\widehat{\mathcal{A}}_n^{\mathcal{E}}$, we follow the line of proof of Proposition \ref{prop:esttoeplitz}. We use that $\sum_{j \geq 1} 1/(e^{nAj}) = 1/(e^{nA} -1)$ for all $A > 0$ and $n$ finite integer.    As an example, let us bound from above one term of the variance of  $\widehat{\mathcal{A}}_n^{\mathcal{E}}$ :
\begin{eqnarray}
 R_{1,2,2} &:=&  \underset{1 \leq j \neq j' \leq T}{\ds\sum } w_j^* w_{j'}^* \Big( |\sigma_{|j'-j| }| + \ds\sum_{ \substack{r=-p+ T +1 \\ r \neq 0}}^{p-(T +1)}  |\sigma_{|r|}\sigma_{|r-j+j'|}| \Big)^2 \nonumber \\
& \leq &  2 \underset{1 \leq j \neq j' \leq T}{\ds\sum } w_j^* w_{j'}^*
 \sigma_{|j'-j| }^2 + 4\underset{1 \leq j \neq j' \leq T }{\ds\sum } w_j^* w_{j'}^* (\ds\sum_{r=1}^{p-(T+1)}  \sigma_{r}^2) (\ds\sum_{ \substack{r=-p+ T+1 \\ r \neq 0}}^{p-(T+1)} \sigma_{|r-j+j'|}^2 )  \nonumber \\
  & \leq & 2 \underset{\underset{|j'-j| < j}{1 \leq j \neq j' \leq T }}{\ds\sum } w^*_{j'} w^*_{|j'-j|}  \sigma_{|j'-j| }^2  + 2 \underset{\underset{|j'-j| > j}{1 \leq j \neq j' \leq T}}{\ds\sum } w_j^*  w^*_{j'} \frac{e^{2A|j'-j|}}{e^{2Aj}} \sigma_{|j'-j| }^2 +4 \underset{1 \leq j \neq j' \leq T}{ \ds\sum}   \Big( \ds\sum_{r=1}^{j}   w^*r \sigma_{r}^2 \nonumber \\
 & + &   w_j^*  \sum_{r=j +1 }^{p-(T+1)}  \ds\frac{e^{2Ar}}{e^{2Aj}}\sigma_{r}^2 \Big)   \Big(\ds\sum_{\substack{r=-p+T+1 \\ |r-j+j'| < j' }}^{p-(T+1)} \hspace{-0.2cm}  w^*_{|r-j+j'|} \sigma_{|r-j+j'|}^2 + w^*_{j'} \hspace{-0.2cm} \ds\sum_{\substack{r=-p+T+1 \\ |r-j+j'| \geq j'}}^{p-(T+1)} \ds\frac{e^{2A|r-j+j'|}}{e^{2Aj'}} \sigma_{|r-j+j'|}^2 \Big) \nonumber \\
  &\leq &  4 \cdot (\sup\limits_{j} w_j^*) \cdot T \cdot  \mathbb{E}_\Sigma(\widehat{\mathcal{A}}_n^{\mathcal{E}}) + 4L \cdot (\sup\limits_{j} w_j^*)^2 \cdot (1/(e^{2A}-1)) + O(T^2) \cdot  \mathbb{E}_{\Sigma}^2(\widehat{\mathcal{A}}_n^{\mathcal{E}}) \nonumber   \\
  &+ & \!\!\! 16L \cdot (\sup\limits_j w_j^*)  \cdot T   \cdot (1/(e^{2A}-1))   \cdot  \mathbb{E}_\Sigma(\widehat{\mathcal{A}}_n^{\mathcal{E}})  +   16L^2 \cdot (\sup\limits_j w_j^*)^2 \cdot (1/(e^{2A}-1)) ^2 . \nonumber
\end{eqnarray}
The proof of the asymptotic normality of $n(p- T)(\widehat{\mathcal{A}}_n^{\mathcal{E}} - \mathbb{E}_{\Sigma}(\widehat{\mathcal{A}}_n^{\mathcal{E}}))$, when $n(p- T) b(\psi) \asymp 1$ and for $\Sigma \in G(\mathcal{E}(A,L) \, , \psi)$ such that $\mathbb{E}_{\Sigma}(\widehat{\mathcal{A}}_n^{\mathcal{E}}) = O(b(\psi))$,  is also due to Theorem 1 of  \cite{Hall84}. That is, we have to check (\ref{conditionNA}) as in Proposition \ref{prop:asymptoticnormality}. As an example, let us bound from above the term $\mathcal{G}_2$ in \eqref{G2} with the  parameters given in \eqref{parameters}:
\begin{eqnarray}
 \mathcal{G}_2 & := &  4 \underset{1 \leq j_1 \neq j_2 < T}{\ds\sum }   w_{j_1}^{*2} w_{j_2}^{*2}   \underset{-p +T +1  \leq r_1, r_2, r_3, r_4 \leq p-(T+1)}{\sum}| \sigma_{|r_1|}\sigma_{|r_1- j_1 + j_2|} \sigma_{|r_2|}\sigma_{|r_2- j_2 + j_1|}| \nonumber \\
&& \hspace{7cm} \cdot  |\sigma_{|r_3|}\sigma_{|r_3- j_1 + j_2|} \sigma_{|r_4|}\sigma_{|r_4- j_2 + j_1|} |\nonumber \\[0.5 cm]
& \leq &  4 \underset{1 \leq j_1 \neq j_2 < T}{\ds\sum }     w_{j_1}^{*2} w_{j_2}^{*2}  ( \sum_{r_1} \sigma_{|r_1|}^2)^{2} ( \sum_{r_1} \sigma_{|r_1 - j_1 + j_2|}^2) ( \sum_{r_2} \sigma_{|r_2 - j_2  +j_1|}^2) \nonumber \\
& \leq & 16L^2 \underset{1 \leq j_1 \neq j_2 < T}{\ds\sum }     w_{j_1}^{*2} w_{j_2}^{*2}  ( \sum_{ \substack{r_1 \\ |r_1| \leq j_1 }} \sigma_{|r_1|}^2 + \sum_{ \substack{r_1 \\ |r_1| > j_1 }} \sigma_{|r_1|}^2 )^{2} \nonumber \\
& \leq &  16L^2 \Big\{ \underset{1 \leq j_1 \neq j_2 <  T}{\ds\sum }    w_{j_2}^{*2} ( \sum_{ \substack{r_1 \\ |r_1| \leq j_1 }} w_{|r_1|}^* \sigma_{|r_1|}^2)^2 +   \underset{1 \leq j_1 \neq j_2 < T}{\ds\sum }    w_{j_1}^{*2} w_{j_2}^{*2} (\sum_{ \substack{r_1 \\ |r_1| > j_1 }} \frac{e^{2Ar_1}}{e^{2Aj_1}} \sigma_{|r_1|}^2 )^{2} \Big \} \nonumber \\
& \leq & 16 L^2 \Big\{ \sum_{j_1}(\sum_{j_2} w_{j_2}^{*2} ) \cdot \mathbb{E}^2_{\Sigma}(\widehat{\mathcal{A}}_n)  + 4L^2  (\sum_{j_2}  w_{j_2}^*) \cdot (\sum_{j_1} w_{j_1}^{*2} \ds\frac{1}{e^{2Aj}}) \} \nonumber \\
& \leq & 16 L^2 \Big\{ \frac{ T}{2}   \cdot  \mathbb{E}^2_{\Sigma}(\widehat{\mathcal{A}}_n^{\mathcal{E}}) + 4L^2\cdot \frac{1}{2} \cdot (\sup\limits_{j} w_j^{*2})  \cdot \frac{1}{e^{2A}-1} \Big\} \nonumber \\
& \leq & \mathbb{E}^2_{\Sigma}(\widehat{\mathcal{A}}_n^{\mathcal{E}}) \cdot O( T) + o(1) = O(\frac{T}{n^2(p- T)^2}) + o(1) = o(1).
\end{eqnarray}

\hfill \end{proof}

\begin{proof}[Proof of Theorem \ref{theoremanalyticaltern}]
To show the upper bound, we use first the asymptotic normality of the $n(p- T) \widehat{\mathcal{A}}_n ^{\mathcal{E}}$ under $H_0$ to prove that the type I error probability of $ \Delta^*$ : $\eta(\Delta^*) = 1 - \Phi( np b(\psi)) +o(1)$.

To bound from above the type II error probability, we shall distinguish 2 cases. First, when $n^2p^2 b^2(\psi) \to + \infty$, we use the Markov inequality,  \eqref{R'_1} and  \eqref{R'_2}, to show that $ \beta( \Delta^*, G(\psi)) \to 0$. Then, when  $n^2p^2 b^2(\psi) \asymp 1$, we have two possibilities: either $\mathbb{E}_\Sigma(\widehat{\mathcal{A}}_n^{\mathcal{E}}) / b(\psi) \to \infty$, or $\mathbb{E}_\Sigma(\widehat{\mathcal{A}}_n^{\mathcal{E}}) = O(b(\psi))$. We show respectively that either type II error probability tends to zero, or we use the asymptotic normality of $n(p- T) (\widehat{\mathcal{A}}_n^{\mathcal{E}} - \mathbb{E}_{\Sigma}(\widehat{\mathcal{A} }_n^{\mathcal{E}}))$ to get that $\beta( \Delta^*, G(\psi))  \leq \Phi (np(t- b(\psi)) + o(1).$

To show the lower bound, we follow the same sketch of proof of lower bounds of Theorems \ref{theo:optimalrates} and \ref{theo:sharprates}. The key point for ellipsoids $\mathcal{E}(A,L)$ is to check the positivity of the matrix
\[
\Sigma^* = T_P( \{ \sigma_j^* \}_{j \geq 1}) \quad \text{ where }  \quad \sigma_j^* =   \ds\sqrt{\lambda} \Big( 1 - (\frac{e^{j}}{e^{T}})^{2A} \Big)_+^{1/2} \quad \text{ for all }
j \geq 1.
\]
Then we create a parametric family of matrices by changing the sign randomly on each diagonal of $\Sigma^*$, with parameters given in \eqref{parameters}.
\begin{lemma}
\label{lem1}
 For $A>0$, the symmetric Toeplitz matrix $\Sigma^*_U= T_p( \{u_j \sigma^*_{j} \}_{j \geq 1} )$, where $U = \{u_j \}_{j \geq 0 }$ with $u_0 =1$, $u_j = \pm 1$ for all $j \geq 1$, and $\sigma_{j}^* $  defined as previously,  is positive definite, for $\psi>0$ small enough.
\noindent Moreover,  denote by $\lambda^*_{1,U}, ..., \lambda^*_{p,U}$ the eigenvalues of $\Sigma^*_U$, then
$|\lambda^*_{i,U}-1| \leq O(\psi \cdot \ds\sqrt{\ln (1/ \psi)} )$, for all $i$ from 1 to $p$.
\end{lemma}
\begin{proof}[Proof of Lemma \ref{lem1}  ]
Using Gershgorin's Theorem we get that each eigenvalue of $ \Sigma^*_U
 = T_p( \{u_j \sigma^*_{j} \}_{j \geq 1} )$ verifies,  $ | \lambda_{i,U}^* - u_0\sigma_0^*| \leq  2 \ds\sum_{\substack{j \geq 1 }}^p |u_{j} \sigma^*_{j}| = 2 \ds\sum_{ j \geq 1  } \sigma^*_{j}$.
We have,
\[
\begin{array}{lcl}
\ds\sum_{j \geq 1} \sigma^*_{j} &= &  \ds\sqrt{\lambda} \ds\sum_{j \geq 1} \Big( 1 - (\frac{e^{j}}{e^{T}})^{2A} \Big)_+^{1/2}  \leq  \sqrt{\lambda} \sum_{j=1}^{T} \Big( 1 - (\frac{e^{j}}{e^{T}})^{2A} \Big)^{\frac{1}{2}} \\[0.5 cm]
& =& O(1) \sqrt{\lambda } \cdot T \asymp \psi \cdot \ds\sqrt{\ln (1/ \psi)}.
\end{array}
\]
We deduce that the smallest eigenvalue is bounded from below by
$$
\min_{i=1,...,p}\lambda^*_{i,U} \geq  \sigma^*_{0} - 2\ds\sum_{j \geq 1} \sigma^*_{j}   \geq 1 - O(1)  \psi \cdot \ds\sqrt{\ln (1/ \psi)}.
$$
which is strictly positive for $\psi >0$ small enough.
\end{proof}

To complete the proof, we follow the steps of the proof of the lower bound in  Section \ref{theoandproof}.
\end{proof}


\begin{thebibliography}{99}

\bibitem{AneirosVieu14}
Germ{\'a}n Aneiros and Philippe Vieu.
\newblock Variable selection in infinite-dimensional problems.
\newblock {\em Statist. Probab. Lett.}, 94:12--20, 2014.

\bibitem{Bai2009}
Zhidong Bai, Dandan Jiang, Jian-Feng Yao, and Shurong Zheng.
\newblock Corrections to lrt on large-dimensional covariance matrix by rmt.
\newblock {\em The Annals of Statistics}, 37(6B):3822--3840, 12 2009.

\bibitem{ButuceaZgheib2014A}
C.~{Butucea} and R.~{Zgheib}.
\newblock Sharp minimax tests for large covariance matrices.
\newblock {\em ArXiv e-prints}, 2014.

\bibitem{ButuceaMeziani11}
Cristina Butucea and Katia Meziani.
\newblock Quadratic functional estimation in inverse problems.
\newblock {\em Stat. Methodol.}, 8(1):31--41, 2011.

\bibitem{CaiJiang11}
T.~Tony Cai and Tiefeng Jiang.
\newblock Limiting laws of coherence of random matrices with applications to
  testing covariance structure and construction of compressed sensing matrices.
\newblock {\em Ann. Statist.}, 39(3):1496--1525, 2011.

\bibitem{CaiMa13}
T.~Tony Cai and Zongming Ma.
\newblock Optimal hypothesis testing for high dimensional covariance matrices.
\newblock {\em Bernoulli}, 19(5B):2359--2388, 11 2013.

\bibitem{CaiRenZhou13}
Tony {Cai}, Zhao {Ren}, and Harrison {Zhou}.
\newblock Optimal rates of convergence for estimating toeplitz covariance
  matrices.
\newblock {\em Probab. Theory Relat. Fields}, 156:101--143, 2013.

\bibitem{ChenChen11}
Kun Chen, Kehui Chen, Hans-Georg M{\"u}ller, and Jane-Ling Wang.
\newblock Stringing high-dimensional data for functional analysis.
\newblock {\em J. Amer. Statist. Assoc.}, 106(493):275--284, 2011.

\bibitem{ChenZZ10}
Song~Xi Chen, Li-Xin Zhang, and Ping-Shou Zhong.
\newblock Tests for high-dimensional covariance matrices.
\newblock {\em J. Amer. Statist. Assoc.}, 105(490):810--819, 2010.

\bibitem{Comte2001}
Fabienne Comte.
\newblock Adaptive estimation of the spectrum of a stationary gaussian
  sequence.
\newblock {\em Bernoulli}, 7(2):pp. 267--298, 2001.

\bibitem{Cuevas14}
Antonio Cuevas.
\newblock A partial overview of the theory of statistics with functional data.
\newblock {\em J. Statist. Plann. Inference}, 147:1--23, 2014.

\bibitem{iwfos14}
{E.G. Bongiorno}, {A. Goia}, {E. Salinelli}, and {P. Vieu}, editors.
\newblock {\em Contributions in infinite-dimensional statistics and related
  topics}. Societ\`a Editrice Esculapio, 2014.

\bibitem{Ermakov94}
M.~S. Ermakov.
\newblock A minimax test for hypotheses on a spectral density.
\newblock {\em Journal of Mathematical Science}, 68(4):475--483, 1994.

\bibitem{Fisher12}
Thomas~J. Fisher.
\newblock On testing for an identity covariance matrix when the dimensionality
  equals or exceeds the sample size.
\newblock {\em J. Statist. Plann. Inference}, 142(1):312--326, 2012.

\bibitem{FisherGallagher12}
Thomas~J. Fisher and Colin~M. Gallagher.
\newblock New weighted portmanteau statistics for time series goodness of fit
  testing.
\newblock {\em Journal of the American Statistical Association},
  107(498):777--787, 2012.

\bibitem{FisherSunGallagher2010}
Thomas~J. Fisher, Xiaoqian Sun, and Colin~M. Gallagher.
\newblock A new test for sphericity of the covariance matrix for high
  dimensional data.
\newblock {\em Journal of Multivariate Analysis}, 101(10):2554 -- 2570, 2010.

\bibitem{Golubev1993}
G.~Golubev.
\newblock Nonparametric estimation of smooth spectral densities of gaussian
  stationary sequences.
\newblock {\em Theory of Probability \& Its Applications}, 38(4):630--639,
  1994.

\bibitem{GolubevNussbaumZhou10}
G.K. {Golubev}, M.~{Nussbaum}, and H.H. {Zhou}.
\newblock Asymptotic equivalence of spectral density estimation and gaussian
  white noise.
\newblock {\em The Annals of Statistics}, 38:181--214, 2010.

\bibitem{GolubevLevitTsybakov96}
Yuri~K. Golubev, Boris~Y. Levit, and Alexander~B. Tsybakov.
 Asymptotically efficient estimation of analytic functions in gaussian
  noise.
 {\em Bernoulli}, 2(2):167--181, 06 1996.

\bibitem{GuptaBodnar14}
Arjun~K. Gupta and Taras Bodnar.
An exact test about the covariance matrix.
 {\em Journal of Multivariate Analysis}, 125(0):176 -- 189, 2014.

\bibitem{Hall84}
Peter Hall.
 Central limit theorem for integrated square error of multivariate
  nonparametric density estimators.
 {\em J. Multivariate Anal.}, 14(1):1--16, 1984.

\bibitem{IngsterSpatinas09}
Yu.~I. Ingster and T.~Sapatinas.
 Minimax goodness-of-fit testing in multivariate nonparametric
  regression.
{\em Math. Methods Statist.}, 18(3):241--269, 2009.

\bibitem{IngsterSuslina03}
Yu.~I. Ingster and I.~A. Suslina.
 {\em Nonparametric goodness-of-fit testing under {G}aussian models},
  volume 169 of {\em Lecture Notes in Statistics}.
 Springer-Verlag, New York, 2003.

\bibitem{IngsterI93}
Yuri~I. Ingster.
 Asymptotically minimax hypothesis testing for nonparametric
  alternatives. i.
 {\em Mathem. Methods Statist.}, 2:85--114, 171--189, 249--268, 1993.

\bibitem{JiangJiangYang12}
Dandan Jiang, Tiefeng Jiang, and Fan Yang.
Likelihood ratio tests for covariance matrices of high-dimensional
  normal distributions.
{\em J. Statist. Plann. Inference}, 142(8):2241--2256, 2012.

\bibitem{LedoitWolf02}
Olivier Ledoit and Michael Wolf.
Some hypothesis tests for the covariance matrix when the dimension is
  large compared to the sample size.
 {\em Ann. Statist.}, 30(4):1081--1102, 2002.

\bibitem{MarteauSapatinas14}
C.~{Marteau} and T.~{Sapatinas}.
{A unified treatment for non-asymptotic and asymptotic approaches to
  minimax signal detection}.
 {\em ArXiv e-prints}, jun 2014.

\bibitem{Neumann1996}
Michael~H. Neumann.
 Spectral density estimation via nonlinear wavelet methods for
  stationary non-gaussian time series.
 {\em Journal of Time Series Analysis}, 17(6):601--633, 1996.

\bibitem{QiuChen12}
Yumou Qiu and Song~Xi Chen.
 Test for bandedness of high-dimensional covariance matrices and
  bandwidth estimation.
 {\em Ann. Statist.}, 40(3):1285--1314, 06 2012.

\bibitem{shiryaev96}
A.~N. Shiryaev.
 {\em Probability}, volume~95 of {\em Graduate Texts in Mathematics}.
 Springer-Verlag, New York, second edition, 1996.
 Translated from the first (1980) Russian edition by R. P. Boas.

\bibitem{Soulier2000}
Ph. Soulier.
 Adaptive estimation of the spectral density of a weakly or strongly
  dependent {G}aussian process.
 {\em Math. Methods Statist.}, 10(3):331--354, 2001.
 Meeting on Mathematical Statistics (Marseille, 2000).

\bibitem{Srivastava2005}
Muni~S. Srivastava.
 Some tests concerning the covariance matrix in high dimensional data.
 {\em J. Japan Statist. Soc.}, 35(2):251--272, 2005.

\bibitem{Srivastava2014}
Muni~S. Srivastava, Hirokazu Yanagihara, and Tatsuya Kubokawa.
 Tests for covariance matrices in high dimension with less sample
  size.
 {\em Journal of Multivariate Analysis}, 130(0):289 -- 309, 2014.

\bibitem{XiaoWu13}
H.~Xiao and W.B. Wu.
 Asymptotic theory for maximum deviations of sample covariance matrix
  estimation.
 {\em Stochastic Processes and their Applications}, 123:2899--2920,
  2013.

\end{thebibliography}
\end{document}